\documentclass[11pt, reqno]{amsart}
\usepackage[margin=2.75cm]{geometry}

\usepackage{import}
\usepackage{common_package,random_graph,Gibbs}
\usepackage{soul}

\def\stackrel#1#2{\mathrel{\mathop{#2}\limits^{#1}}}

\let\oldref\ref
\renewcommand{\ref}[1]{(\oldref{#1})}

\newcommand{\sbrac}[1]{\left[#1\right]}

\newcommand{\optset}{{\mathcal M}}
\newcommand{\openset}{{\mathcal O}}

\newcommand{\RRR}{{R}}
\newcommand{\BB}{\mathcal{B}}

\newcommand{\logb}[2]{\log\brac{\frac{#1}{#2}}}

\numberwithin{equation}{section}

\DeclareFontFamily{U}{mathx}{\hyphenchar\font45}
\DeclareFontShape{U}{mathx}{m}{n}{
      <5> <6> <7> <8> <9> <10>
      <10.95> <12> <14.4> <17.28> <20.74> <24.88>
      mathx10
      }{}
\DeclareSymbolFont{mathx}{U}{mathx}{m}{n}
\DeclareFontSubstitution{U}{mathx}{m}{n}
\DeclareMathAccent{\widecheck}{0}{mathx}{"71}
\DeclareMathAccent{\wideparen}{0}{mathx}{"75}

\setlist{nosep}
\newcommand{\brac}[1]{\left(#1\right)}

\newcommand{\E}{\mathbb{E}}

\usepackage{newunicodechar}
\newunicodechar{ℓ}{$\ell$}
\newunicodechar{ᵖ}{$p$}

\newtheoremstyle{break}
  {\topsep}{\topsep}%
  {\itshape}{}%
  {\bfseries}{}%
  {\newline}{}%
\theoremstyle{break}


\newcounter{parentnumber}

\title[Atypical consensus and splitting Gibbs measures on random regular graphs]{Gibbs conditioning, atypical consensus and splitting Gibbs measures on random regular graphs}\thanks{IHC, KR and SY were supported by the Office of Naval Research under the Vannevar Bush Faculty Fellowship
N0014-21-1-2887.  IHC and KR  also acknowledge support from the National Science Foundation under grant DMS-2246838, IHC acknowledges support from a GSSA scholarship from the Ministry of Education, Taiwan, IL acknowledges support from a SPRINT UTRA award from Brown University, and  SY acknowledges support from the Prime Minister Early Career Research Grant ANRF/ECRG/2024/005529/PMS}

\author{I-Hsun Chen, Ivan Lee, Kavita Ramanan, Sarath Yasodharan}

\begin{document}

\begin{abstract}
Given $n$ independent Bernoulli($p$) random variables $X_i, i = 1, \ldots, n$, representing the opinions of individuals connected by  an underlying random $\kappa$-regular graph $G_n$ with edge set $E(G_n)$, we show that when conditioned on an atypical empirical consensus $\frac{1}{n} \sum_{i = 1}^n \sum_{j:ij \in E(G_n)} X_i X_j$,   the joint distribution of the random variables converges,  as $n \rightarrow \infty$, to  an Ising measure on the infinite $\kappa$-regular tree $T^\k$ with a specific external field that depends only on the bias parameter  $p$, and  a  temperature that depends on both $p$ and the atypical consensus value.    In particular, conditional on the empirical consensus being smaller than typical, the limit is a  translation-invariant splitting ($\TIS$) antiferromagnetic Ising measure on $T^k$. When the empirical consensus is larger than typical, the conditional limit is a ferromagnetic  Ising measure on $T^\k$  with plus or minus  boundary condition if the bias is positive or negative, respectively, whereas if the bias is zero, there is a phase transition in the nature of the limit depending on whether or not  the consensus  exceeds $\k/(\k-1)$. More generally, when the random variables are independent and identically distributed (i.i.d.)  on a finite space $\X$, we show that when conditioned on  an atypical  generalized consensus $\frac{1}{n} \sum_{i = 1}^n \sum_{j:ij \in E(G_n)} h(X_i, X_j)$,  for any symmetric edge potential  $h: \X \times \X \to \R$,   the limiting joint distribution lies in the set of translation invariant splitting Gibbs measures on $T^\k$, whose vertex marginals are supported  on  a (possibly strict) subset of $\X.$ The proofs leverage a tractable form of the large deviation rate function for component empirical measures of random regular graphs with i.i.d.\! marks and Gibbs conditioning principles, and  entail careful analyses of associated non-convex constrained optimization problems. As a by-product of our results, we obtain an (asymptotic) analog  of the  maximum entropy principle for  Gibbs measures on random regular graphs. \\

\end{abstract}

\maketitle

\section{Introduction}

\subsection{Background and discussion of results}

In many areas of probability, one is often interested in understanding the most probable behavior of a random object conditioned on a certain rare event of interest. 
One approach to answering 
this question, at least in an asymptotic sense, is via large deviations theory.  
The classical setting is as follows: given a sequence $\{X_1, \ldots, X_n\}$ of independent and identically distributed (i.i.d.) random variables on a finite set $\X$ with common law $\nu$, one wants to understand how  this  sequence behaves given that its empirical 
 average $S_n = \frac{1}{n} \sum_{i=1}^n X_i$ deviates   from its  
typical value, $\text{mean}(\nu): = \int_{\R} x \nu (dx)$.  
  In this case, one first appeals to  Sanov's theorem \cite{San57}, which shows that the sequence of empirical measures $\nu_n = \frac{1}{n} \sum_{i=1}^n \delta_{X_i},$ $n \in \N$, where $\delta_x$ is the Dirac delta measure at $x$,  satisfies a large deviation principle (LDP) that characterizes the (logarithmic) asymptotics of probabilities of the form 
$\mathbb{P}( \nu_n \in C) \approx e^{-n \inf_{\nu' \in C} H(\nu'\|\nu)},$ for sufficiently regular subsets $C$ of probability measures on $\X$, 
with the ``rate function"   given by the usual relative entropy functional 
$H(\cdot\|\nu)$.  Then the Gibbs conditioning principle \cite{Vas80,Csi84} (see also \cite[Section 7.3]{Dembo_Zeitouni_2010}) can be applied to show that conditioned on the event $S_n  \geq c$, the  empirical measure $\nu_n$ converges,  as $n \rightarrow \infty$, to the minimizer $\nu^*$ of a  constrained   optimization problem: \[ \nu^*  = \arg\min_{\nu': \text{mean}(\nu') \geq c} H(\nu'\|\nu). \]
Due to the convexity of relative entropy and linearity of the constraints, this optimization is particularly tractable and $\nu^*$ can be identified explicitly.  
An additional change of measure argument is then required
(e.g., as in the proof of the large deviations lower bound in 
Cram\'{e}r's theorem 
\cite{Cra38},\cite[Section 2.2]{ Dembo_Zeitouni_2010}), 
 to argue that  
when conditioned on the event $S_n \geq c$,  the random variables $\{X_i, i = 1, \ldots, n\}$ 
behave asymptotically like an i.i.d. sequence with common law $\nu^*$. 
There have been numerous generalizations of this Gibbs conditioning principle,  including early extensions to empirical measures for Markov chains in  \cite{CsiCovCho87,MedNey99}.

In this paper, we consider an analog of this classical question on random regular graphs. Fix an  integer $\kappa \geq 2$ and let $G_n$ be the random $\kappa$-regular graph with vertex set $V(G_n) = \{1, \ldots, n\}$, that is,  $G_n$ is uniformly distributed on the set of $\k$-regular graphs on $n$ vertices.  
 Once again, let $\{X_i\}_{i \in \N}$  be a sequence of i.i.d.\,random variables taking values in a finite set $\X$ with law $\nu$. 
 Also assume that for every $n$, the vector 
 $X^n = (X_1, \ldots, X_n)$,  representing  vertex marks, is independent of the edge structure $E(G_n)$ of the graph $G_n.$
 Given any symmetric {\em edge potential} function $h : \X \times \X \to \R,$  we study the following question for $c \in \R:$
\begin{quote}
conditioned on the event that the empirical $h$-consensus is large, that is,  
\begin{align*}
 \frac{1}{n} \sum_{u \in V(G_n)} \sum_{v \sim u} h(X_u, X_v)  =  \frac{1}{n} \sum_{u,v = 1}^n 1_{\{uv \in E(G_n)\}} h(X_u, X_v) \geq \cc,
\end{align*}
what is the typical structure of the marked graph $(G_n, X^n)$, 
as $n \rightarrow \infty$?  
\end{quote}
 Of course, this question is most interesting  when the event  we are conditioning on is rare, that is, $\cc$ is larger than $\sref$, the limit  of the empirical $h$-consensus, $\frac{1}{n} \sum_{v \in G_n}\sum_{u \sim v} h(X_u^n, X_v^n)$, as $n \rightarrow \infty.$ 
In particular, we would like to understand whether the typical marked graph $(G_n, X^n),$ conditioned on this rare event, exhibits heterogeneity in the  marginal distributions of vertex marks,  dependency among adjacent vertex marks, long-range correlations, etc.

Our first main result (Theorem \ref{thm:Ising-phase-transition} of  Section \ref{subs:Isingresult}) answers the above question completely in the case of two-spin systems with the {\em consensus function,} that is, when  $\X = \{-1, 1\}$, $\nu$ is a Bernoulli distribution with 
parameter $p = \nu(\{1\}) \in (0,1),$ 
and $h(x,y) = m(x,y) := xy$ for $x, y \in \X.$  
In this case the typical value of the empirical consensus is  
  $\smref := c_{m,\text{ref}} = \k(2p-1)^2$, and {\em a priori} an  atypical 
  consensus could be achieved 
in various ways, for instance, by just modifying the mark distribution $\nu$, 
or by introducing some  dependence  between  spins,  or a mixture of both.
  We show that the typical structure of $(G_n, X^n),$ conditioned on an atypical empirical consensus, looks like an Ising measure on the infinite $\k$-regular tree $T^\k$ with a specific  inverse temperature and external field. Specifically, if $p\neq 1/2,$ then $(G_n, X^n),$ converges  (locally weakly) in probability to a ferromagnetic (resp. antiferromagnetic) Ising measure on $T^\k$ when conditioned on the empirical consensus being greater than $c > \smref$ (resp. less than $\cc < \smref$).    
  On the other hand, if   $p = 1/2$ (in which case $\smref = 0$), 
  and one conditions on the consensus exceeding a  positive value $\cc$, 
    there is a phase transition in the behavior of $(G_n, X^n)$. For small  $\cc$, namely when  $0 < \cc \leq c_*:= \k/(\k-1)$, the conditional limit is the unique Ising measure  on $T^\k$ with inverse temperature $\tanh^{-1}(\cc/\k)$ and no external field, whereas when $\cc$ is larger than $c_*,$  the conditional limit points are ferromagenetic  Ising measures on $T^\k$ with no external field and either a plus or minus boundary condition.  Note that $\tanh^{-1}(c_*/\k) = \tanh^{-1}(1/(\k-1))$ corresponds to the uniqueness threshold, that is, the inverse temperature below which there exists a unique Ising measure on  $T^\k$. 
 Furthermore, if $p = 1/2$ and the consensus is less than
 $\cc < 0$, then the  
  conditional limit is always an    antiferromagnetic Gibbs measure 
  with no external field.  
  It is  interesting to note  that despite the  symmetry of $\nu$  when $p = 1/2$, 
    unlike in the case of a positive consensus, 
  when the 
  consensus is negative there is no phase transition and the conditional distribution of  $(G_n, X^n)$ always  converges locally to a uniquely defined limit, even when the corresponding inverse temperature lies in the  non-uniqueness regime, that is, where  multiple antiferromagnetic Ising measures on $T^\k$ co-exist 
    (for  insight into why this is the case, see the discussion after Theorem \ref{thm:Ising-phase-transition}). 

In all regimes above, 
the value of the external field of the conditional limit is always given by $B(p) = \log(p/(1-p)),$ 
irrespective of the value of $\cc,$ 
and the inverse temperature is  such that the consensus under the corresponding Ising measure is exactly $\cc$.
As elaborated in 
Remark \ref{rmk:external-field-p}, the emergence of the Ising measure and the form of the external field show  that 
the most likely way to achieve atypical consensus is not by adding a bias to  the mark distribution, but rather by introducing correlations between  neighboring spins. 
Although it may seem natural that the most likely way to achieve a lower consensus is by increasing the probability of disagreements between neighboring spins (antiferromagnetic interaction), this is not completely obvious since there  also do exist ferromagnetic Ising measures on $T^\k$ 
that induce a consensus that is lower than typical (see Example \ref{example:ferro-lower-consensus}).   
 In Section \ref{subs:connecting-back-to-finite-graphs}, we also discuss when one can deduce the conditional structure of the  marked graph sequence $(G_n, X^n)$ from knowledge of its (conditional) local limit.

Our next result concerns the case of general spins (i.e., where $\nu$ is a distribution on an arbitrary finite set $\X$) and general edge potential $h.$ Our main result here (Theorem \ref{thm:Gibbs-deg-and-nondeg} of  Section \ref{subs:Isingresult}) shows that all conditional limit points lie in the set of translation-invariant splitting $(\TIS)$ Gibbs measures on the $\k$-regular tree (see Definition \ref{def:TISGM}). Although in general these limiting Gibbs measures could be degenerate, in the sense that their vertex marginals could be supported on a strict subset of $\X$ (as shown in  Example \ref{example:boundary-global-minimizer}),  for two-spin systems  with $\nu$ being the  uniform distribution,  we  prove that the limit points are always non-degenerate $\TIS$ Gibbs measures,  that is, the vertex marginals of the limit Gibbs measures have full support on $\X$ (see Theorem \ref{thm:Gibbs-component-two-spin-uniform}). This leaves open the interesting question of precisely identifying the conditions on  $\nu$ and $h$ under which one is guaranteed to get  non-degenerate conditional limit points.  We conjecture this is true whenever $\nu$ is the uniform distribution on any finite set $\X$ (see Conjecture \ref{conj:Gibbs-component}). 

The proofs of our main results leverage the large deviation principle (LDP) for the sequence $\{U_n\}$ of component empirical measures of random regular graphs with i.i.d.\ vertex marks obtained in \cite{RamYas23,BalPerRei26}; see Section \ref{sec:empirical-measures} for a precise definition of $U_n$. 
Then, by invoking corresponding Gibbs conditioning principles \cite{Leonard2010EntropicProjections}, the conditional limit points can be characterized as minimizers of a constrained optimization problem on measures on rooted trees (representing component empirical measures) with the objective function being the rate function, and the constraints dictated by the rare event of interest. In our analysis, by using the specific form of the rate function (stated in Theorem \ref{thm:ldp-Un}), we first reduce this  optimization problem to a constrained  optimization problem over measures on rooted trees of depth-1 (corresponding to neighborhood empirical measures). Then, by partially solving the latter problem,  we obtain a further 
reduction (in Proposition \ref{prop:optsimp}) to a constrained minimization problem over edge marginals, that is, (symmetric) probability measures on $\X \times \X.$ Throughout our analysis, we crucially use the succinct form of the rate function for the component empirical measure, in terms of sums of relative entropies, obtained in \cite[Theorem 4.13]{RamYas23}.

The remaining technical work is then devoted to the analysis of this  constrained edge minimization problem. Any minimizer $\pi_*$  that has full support on $\X \times \X$  (which we refer to as an interior minimizer) must be a critical point of the Lagrangian of the optimization problem.  This Lagrangian (perhaps
unsurprisingly) coincides with the negative of the Bethe rate function, defined in  \cite{DemMonSlySun14} for all factor models, and used therein to justify the Bethe  prediction for the limiting free energy density 
of ferromagnetic Potts models on graph sequences converging locally to $T^\k$ when $\k$ is even. 
In \cite[Proposition 1.7]{DemMonSlySun14},  the critical points of the Bethe rate function (equivalently, of the Lagrangian) 
are expressed explicitly in terms of the fixed points $\ell$ of an associated cavity map (or belief propagation recursion).  The well known fact that these fixed points act as boundary laws for corresponding 
$\TIS$  Gibbs measures (see \cite[Theorem 4.1]{Zac83})  
then shows that any interior minimizer must be 
the edge marginal of a $\TIS$  Gibbs measure (see Lemma \ref{lemma:TISGM-to-cavity}). 
Combining this with a lifting lemma (see Lemma \ref{lemma:unimodular-extension-of-splitting-Gibbs-measure}) 
that expresses $\TIS$ Gibbs measures in terms of their edge marginals via unimodular extensions, and  a parallel relation  between minimizers of the edge and component minimization problems 
(established in Proposition \ref{prop:optsimp} by exploiting the form of the rate function for the  component empirical measure  LDP obtained in \cite{RamYas23}),  
we conclude that any minimizer of the component empirical measure optimization problem 
whose edge marginal has full support must be a $\TIS$ Gibbs measure. 

However, (rather surprisingly) this is far from the end the story since {\em boundary minimizers} of the edge optimization problem  (i.e., minimizers whose support is a strict subset of $\X\times \X$) also  play a substantial role. As shown in Example \ref{example:boundary-global-minimizer}, even when $\X=\{1,-1\}$, there exist edge potentials  $h$, mark distributions $\nu$ and constraint values  $\cc$ such that the unique minimizer of the corresponding edge constrained minimization problem is a {\em boundary minimizer}. Fortunately,   under mild conditions on the constraint value $\cc$, we are able to show that  every {\em boundary minimizer} $\pi_*$ must be degenerate in the sense that the support of $\pi_*$ takes the form $\X^\downarrow\times \X^\downarrow$ for some $\X^\downarrow\subsetneq \X$ (see Lemma \ref{lemma:pos_marginal_cannot_be_locmin}). Thus, by further restricting the minimization problem to (symmetric) probability measures on $\X^\downarrow\times \X^\downarrow$, we turn $\pi_*$ into an interior minimizer of a modified edge constrained minimization problem with $\X$ replaced by $\X^\downarrow$, and invoke the above argument for interior minimizers in this modified setting to conclude that conditional limits  must be (possibly degenerate) $\TIS$ Gibbs measures.   

Unlike in the classical setting of Sanov's theorem mentioned above, the optimization problems we deal with are  non-convex, and therefore the analysis is substantially more involved. Nevertheless, in the special case of two-spin systems with the consensus function,  we completely  solve this problem in all parameter regimes, and thereby obtain the full picture  stated in Theorem \ref{thm:Ising-phase-transition}. 
It is clear from our proof that more can be said about conditional limits in the general setting under specific assumptions on $\nu$ and $h$,   by combining partial results that we obtain for multi-spin systems with further analysis of the 
corresponding optimization problem.  
 But we do not pursue this direction in this paper.

 Note that the empirical $h$-consensus can also be written as a functional of just the simpler 
 neighborhood empirical measure, 
$L_n = \frac{1}{n} \sum_{u=1}^n \delta_{(x_u, x_v, v \sim u)},$ which is a random probability measure on $\X^{1+\k}$ (or even the edge empirical measure), rather than of the component empirical measure $U_n$.    However, consideration of the LDP and associated Gibbs conditioning principle for only the sequence $\{L_n\}$ (or of edge empirical measures) would only yield information on the most likely conditional behavior of the neighborhood (or edge) marginals 
 (e.g., see \cite[Proposition 3.5]{RamYas23} for the special case when 
$h(x,y) = g(y),$ for some $g : \X \to \R$).
Here we consider the more complicated Gibbs conditioning principle  for the sequence $\{U_n\}$ of component empirical measures  so as to be able  to shed  light on the conditional behavior of the full joint distribution of  $\{X_i\}_{i \in \N}$.  The framework  developed here could also be useful for understanding large deviation properties of interacting stochastic processes on sparse graphs, which  was posed as an open problem in \cite[Section 5.2]{Ram23}.  

\subsection{A new maximum ``entropy'' principle}
Our results can also be viewed as identifying a new {\em maximum entropy principle} for Gibbs measures on random regular graphs, at least in an asymptotic sense. It is well known that
a Gibbs measure on $n$-particle configurations with spin values in a finite set ${\X}$, of the form 
\[ \mu(\sigma)=\frac{1}{Z_n}\exp \pr{\beta \calH_n(\sigma)},\quad \sigma=(\sigma_i)_{i\in [n]}\in \X^n \]
arises as the unique probability measure  on the finite set ${\mathcal Z}=\X^n$ that maximizes the 
Shannon entropy 
\begin{equation}
    \label{defn:entropy}
H(\mu) = 
- \sum_{z \in {\mathcal Z} }\mu(z) \log \mu(z), 
\end{equation}
subject to the  ``energy'' constraint $\sum_{z\in \mathcal{Z}}\mu(z)\calH_n(z)=\cc$. 
While the appearance of Shannon entropy in the 
classical maximum entropy principle is formally justified by viewing it as a measure of information, and then viewing the maximizing Gibbs measure  as the ``least biased representation of our knowledge'', 
a more rigorous justification for it being the right choice  (though only in limited cases and only in an asymptotic sense), is provided by large deviations theory. 
Then invoking Sanov's theorem   
and the associated Gibbs conditioning principle, it follows that under the energy constraint (rare event), the most likely measure is the one that minimizes this rate function, thereby maximizing the entropy, which is the Gibbs measure.  For a comprehensive treatment of this correspondence, we refer the readers to \cite{Tou09} and the references therein, as well as the pioneering works by Ellis \cite{Ell95,Ell99,Ell12}, Oono \cite{Oon89} and Lanford \cite{Lan73} (reprinted version \cite{Lan07}). 

In the present setting we analogously leverage large deviations theory to show that ($\TIS$) Gibbs measures on infinite regular trees arise as (asymptotic)  maximizers of a different entropy-type  functional, namely the negative of the large deviation rate function for the component empirical measure (see the functional $I^\k$ defined in Section \ref{sec:ldp-graph}), 
which  takes the place of the usual {\em entropy}, subject to an energy constraint expressed in terms of a consensus functional.  In contrast to the classical case, where one obtains a unique Gibbs measure as the maximizer by virtue of the  concavity of the entropy functional, in the present setting,  the non-convexity of the rate function $I^\k$ yields multiple maximizers for the new  entropy functional. 
In particular, this manifests as a phase transition in Theorem \ref{thm:Ising-phase-transition}.

\subsection{Outline. }
The rest of this paper is organized as follows. In Section \ref{sec:notation} below, we gather some common notation used in this paper. In Section \ref{sec:main-result}, we describe the mathematical model and present our main results.  
The proofs  are presented in Section \ref{sec:cond-to-opt}--Section \ref{sec:Ising-phase-transition-proof}. 
In Section \ref{sec:cond-to-opt}, we reduce the problem of identifying conditional limits to the analysis of a certain constrained optimization problem on the space of probability measures on marked trees. In Section \ref{sec:characterization_of_the_optimizers}, we analyze this optimization problem and prove our conditional limit theorem for general spin systems. Section \ref{sec:sufficient-condition-for-nondegeneracy} and Section \ref{sec:Ising-phase-transition-proof} contain proofs or results on   two-spin systems.   Appendices \ref{sec:pf-of-TISGM-consistency}-\ref{sec:c=c'} contain  proofs  of various technical lemmas, while 
Appendix \ref{sec:boundaryminimizers} and   
Appendix \ref{sec:example-ferro-ising-lower-consensus} 
  contains illustrative examples.

\subsection{Common notation and terminology}
\label{sec:notation}
We start by collecting some standard notation and terminology used throughout the paper.  Given $x \in \R,$ we let ${\rm sgn}(x)$ denote the sign function, that is, we define ${\rm sgn}(x)$ as $-1,$ $0$ or $1$ based on whether $x < 0$, $x=0$, or $x > 0,$ respectively. Given a set $\Z$, for any function $f:\Z^2\rightarrow \R$, define its {\em range} by
\begin{equation}
    \mathsf{Range}(f)\coloneqq \pr{f(z,z'):z,z'\in \Z}.\label{eqn:range}
\end{equation}
Fix a topological space $\calS$ and a subset $A\subset \calS$. We say that a set $\mathcal{U}_A$ is {\em relatively open} in $A$ if there exists an open set $\mathcal{U}$ in $\calS$ such that $\mathcal{U}_A=\mathcal{U}\cap A.$

\noindent 
\subsubsection{Measure-theoretic notation}\label{subsec:measure-notation}
For fixed $p\in [0,1]$, let $\Ber(p)$ denote the Bernoulli distribution on $\{1,-1\}$ with the probability of $1$ equal to  $p$. Given a Polish space $\mathcal{Z}$ equipped with its Borel $\sigma$-algebra, let $\P(\mathcal{Z})$ denote the space of probability measures on $\mathcal{Z}$ endowed with the topology of weak convergence. Given $\mu\in \P(\mathcal{\mathcal{Z}}),$ let $\E_\mu$ denote expectation with respect to $\mu$. For any bounded measurable function $f:\mathcal{Z}\to \R$, we write 
\begin{align*}
    \langle f,\mu\rangle\coloneqq \int_{\mathcal{Z}}f(z)d\mu(z).
\end{align*}
Given $\mu,\nu\in\P(\mathcal{Z}),$ we use   $\mu\ll \nu$ to indicate that $\mu$  is absolutely continuous  with respect to $\nu.$ Also, let $H(\mu\|\nu)$ denote the relative entropy of $\mu$ with respect to $\nu$, defined by
\begin{align}\label{eqn:relative-entropy}
    H(\mu\|\nu)\coloneqq \begin{dcases}
        \E_\mu \left[\log\left(\frac{d\mu}{d\nu}\right)\right] & \text{ if } \mu \ll \nu,\\
        \infty & \text{ otherwise}.
    \end{dcases}  
\end{align}

When  $\Z$ is finite,  let $\mathsf{Unif}(\Z)$ denote the uniform distribution on $\X.$ We will often consider random variables $(X_o, X_1)$ with some law $\pi \in \P (\Z^2).$ In this case we denote the marginal law of $X_o$ by $\pi^o$ and the conditional law of $X_1$ given $X_o$ by $\pi^{1\mid o}.$

\noindent 
\subsubsection{Graph notation:}\label{subsubsec:topology} Given a graph $G = (V,E)$, with vertex set $V = V(G)$ and edge set $E = E(G)$, we will often abuse notation and use $G$ to also denote the vertex set; the meaning should be clear from the context. Let $\calG$ denote the set of graphs. For any subset $\Lambda\subseteq V$, define its boundary as
\begin{equation}\label{eqn:graph-boundary}
    \partial \Lambda\coloneqq \pr{v\in V\setminus \Lambda: \exists u\in \Lambda \text{ such that }\pr{u,v}\in E}.
\end{equation}
For any graph $G\in \calG$, the graph distance $d_G$ is defined as follows: for any two vertices $u,v\in G$, define
\begin{equation}\label{eqn:graph-distance}
    d_G(u,v)\coloneqq \inf\pr{k\in \NZ: \exists v_0=v,v_1,\ldots,v_k=u \in V\text{ such that }\pr{v_{i-1},v_i}\in E,\forall i\in [k]}.
\end{equation}

Let $\Gstar$ denote the set of (unlabeled) rooted graphs, that is, graphs with a distinguished vertex called the root. An element in $\Gstar$ is often denoted by $(G,o)$ or $G$ when the root is clear from context.
An acyclic rooted graph $(T,o)$ is called a tree.  The set of all (unlabeled) rooted trees is denoted by $\Tstar$.
Fix $r\in \N$. For any rooted tree $(T,o)\in \Tstar$ and any vertex $v\in T$, we say that $v$ is a depth-$r$ vertex (of $T$) if $d_T(v,o)=r$. We say that a rooted tree $(T,o)\in \Tstar$ has  depth $r$ if $\sup\pr{d_T(v,o):v\in T}=r$. Given $(T,o)\in \Tstar$,  its depth-$r$ subtree $T_r=(T,o)_r$ is defined to be  the subgraph of $(T,o)$ containing all vertices with depth at most $r$ rooted at $o$. The set of (unlabeled) rooted trees with depth less (or equal) than $r$ is denoted by $\calT_{*,r}$.  Elements in $\Tstarone$ will be referred to as  ``stars''. 
 Given a rooted tree $(T,o)\in \Tstar$ and any $v\in T \setminus \{o\}$,  define its ancestor $a(v)$ as the vertex $u\in N_v(T)$ such that $d_T(u,o)=d_T(v,o)-1$. Note that $a(v)$ is uniquely  defined since $(T,o)$ is a rooted tree.

For any positive integer $\kappa$, let $T^\k$ denote the infinite unlabeled rooted $\k$-regular tree, that is an acyclic graph that has a distinguished vertex, called the root, which has $\kappa$ neighbors, and where  all other vertices have $\kappa$ neighbors. Note that $T^\k\in \Tstar$. For consistency, we shall 
denote the sets $\{T^\k\}$ and $\{T^\k_r\}$ by just $\Tstar^\k$ and $\calT_{*,r}^\k$, respectively. 

\subsubsection{Marked Graph notation}\label{subs:marked-graph-notation}

Next, given a nonempty finite set  $\X,$ let $\Gstar[\X]$ denote the set of marked rooted graphs with $\X$-valued vertex marks, that is, the set of rooted graphs with each vertex carrying an $\X$-valued mark. A typical element of $\Gstar[\X]$ is written as  $(G,X),$ where $G\in \Gstar$ and $X = (X_v)_{v \in G}$  denotes the vertex marks. We equip $\Gstar[\X]$ with the local topology, which makes it a Polish space (see \cite{AldLyo07} or \cite[Lemma 3.4]{Bor16}). Moreover, let $\Tstarone[\X]$ denote the set of ``marked stars'' with $\X$-valued vertex marks, that is, the set of unlabeled rooted trees in $\Tstarone$ with each vertex carrying an $\X$-valued mark. A typical element of $\Tstarone[\X]$ is written as $(\tau,X)$, where $\tau\in \Tstarone$ and $X=(X_v)_{v\in \tau}$ denotes the vertex marks. Furthermore, for $\k\in \N$, let $\TkappaX{1}\subset \Tstarone[\X]$ denote the collection of ``marked $\k$-stars'', that is, marked stars whose root degree is $\k$. Note that $\TkappaX{1}$ is a finite set, and we equip it with the discrete topology.

\begin{definition}[Depth-$r$ marginals]\label{def:rho_h}
    For any $r\in \N$ and $\rho\in \P\cpr{\Gstar[\X]}$, define the depth-$r$ marginal $\rho_r\in \P\cpr{\calG_{*,r}[\X]}$ of $\rho$ as the law of $(\pmb\tau,\bfX)_r$, where $\Law\cpr{\pmb\tau,\bfX}=\rho.$ Note that $\rho_1\in \P\cpr{\Tstarone[\X]}$.
\end{definition}

\begin{remark}\label{remark:strong-weak-top}
    Since $\TkappaX{1}$ is a finite set, the weak topology on $\P(\TkappaX{1})$ is equivalent to the strong topology 
    generated by the total variation metric (e.g., see, \cite[Exercise 3.2.11]{Dur19}). 
    
\end{remark}

\begin{remark}\label{rmk: labeling scheme}
Although $\P(\TkappaX{1})$ is the space of probability measures on unlabeled marked $\kappa$-stars, it is sometimes convenient to view it as the collection of probability measures on labeled marked stars where the leaves are assigned labels uniformly at random from $\{1,\ldots,\k\}$.  In the sequel, whenever vertex labels are used for a $\TkappaX{1}$ random element $(\tau,X),$ it is understood that $(\tau,X)$ is viewed as a labeled marked star where the labels are assigned uniformly at random. Note that in this case the leaf marks $(X_1, \ldots X_\k)$ are exchangeable.
\end{remark}

\section{Main results: How atypical consensus occurs}
\label{sec:main-result}
Throughout, we fix a positive integer $\kappa\geq 2$, a finite set $\X$ and a mark distribution  $\nu \in \P(\X)$, which we assume without loss of generality to satisfy  $\nu (x) > 0$ for all $x \in {\mathcal X}$.

\subsection{Random graph model}\label{sec:empirical-measures}

For each $n\in \N$, let $G_n$ be the random $\kappa$-regular graph on $n$ vertices, that is, $G_n$ is  uniformly sampled from the set of $\k$-regular graphs on $n$ vertices.   Let $(G_n,X^n)$ denote  the corresponding marked graph consisting of the graph $G_n$ with its vertices carrying  independent and identically distributed (i.i.d.) marks  $X^n = (X^n_v)_{v \in G_n}$ with common law $\nu$, that 
  are also independent of $G_n$. The joint law of $(G_n,X^n)$ can be viewed as a prior distribution and we are interested in how it changes when conditioned on a rare event.  Let  $U_n$ denote the corresponding component empirical measure: 
\begin{equation*}
    U_n=U_n(G_n,X^n)\coloneqq \frac{1}{n}\sum_{v\in G_n}\delta_{\Conn_v(G_n,X^n)},
\end{equation*}
where $\delta_z$ denotes the Dirac measure at $z$ and $\Conn_v(G_n,X^n)$ denotes the marked rooted graph induced by the connected component of $(G_n,X^n)$ rooted at $v$, viewed as an element of $\Gstar[\X]$. Note that $U_n$ is a $\P(\Gstar[\X])$-valued random element.

Similarly, let $L_n$ denote the  corresponding neighborhood empirical measure: 
\begin{align*}
    L_n = L_n(G_n, X^n)\coloneqq\frac{1}{n}\sum_{v\in G_n}\delta_{cl_v(G_n,X^n)}, 
\end{align*}
where $cl_v(G_n,X^n)$ denotes the marked $\kappa$-star induced by the first neighborhood of $(G_n,X^n)$ rooted at $v$, viewed as an element of $\TkappaX{1}$ (i.e., the unlabeled marked $\kappa$-star that has root  $v$ and the set of leaves equal to the   neighbors of $v$ in $G_n$,  along with their corresponding i.i.d.\,marks). Note that $L_n$ is a $\P(\TkappaX{1})$-valued random element, and  $L_n$ is the depth-$1$ marginal of $U_n$, that is, $L_n=(U_n)_1$, as specified in Definition \ref{def:rho_h}.

\begin{remark}[Typical Asymptotic Behavior and the True Law]
\label{remark:true-law} 
Define  
$\eta \in \P(\TkaXinf)$ 
to be the law of an infinite $\k$-regular tree with independent vertex marks having common law $\nu$. 
It is well known  
(e.g., see \cite[Theorem 3.2]{LacRamWu23} or \cite[Theorem 5.8]{OliReiSto20}) that 
\begin{equation}
    \label{ln-lln}
L_n \stackrel{p}{\rightarrow} \eta_1 \quad \text{ and }\quad U_n\stackrel{p}{\rightarrow}\eta \quad \mbox{ as } n \rightarrow \infty, \end{equation}
where $\stackrel{p}{\rightarrow}$ represents convergence in probability.  We refer to $\eta$ as  the true law of the infinite (marked) tree and $\eta_1$ as the true law of the (marked) root neighborhood. 
\end{remark}
In this article, we are interested in 
Gibbs conditioning principles that describe the asymptotic behavior of the above marked graphs and their component empirical measures, when conditioned on a natural class of rare events.

\subsection{Conditional limits and  phase transitions in two-spin systems} 
Our first main result concerns the asymptotic behavior of $\{U_n\}$ conditioned on a rare event associated with $\{L_n\}$ for two-spin systems.  We fix $\X:=\{1, -1\}$, and let $\prone \in (0,1)$ be such that $\nu = \Ber(\prone)$.  Also, define the {\em consensus function} 
\begin{align}  
\label{def-spech}
\spech (x,y) := xy, \qquad x, y \in \X.  
\end{align}

Then the empirical consensus, defined as $\E_{L_n} \sqpr{\sum_{v=1}^\k\spech(X_o,X_v)}$, measures the amount of agreement between  neighboring vertices  in the graph. By Remark \ref{remark:true-law}  the empirical consensus  converges to a limit:
\begin{equation}
    \E_{L_n} \sqpr{\sum_{v=1}^\k\spech(X_o,X_v)}\stackrel{p}{\rightarrow} \Exp_{\eta_1}\sqpr{\sum_{v=1}^\k\spech(X_o,X_v)} = \k (2\prone -1)^2\eqqcolon \smref. \label{eqn:smref}
\end{equation}
We consider rare events wherein the empirical consensus deviates from its typical value. Specifically, for $c \neq \smref,$ we consider the event
\begin{align*}
    \Conset_n(c)\coloneqq
    \begin{dcases}
       \pr{ \E_{L_n} \sqpr{\sum_{v=1}^\k\spech(X_o,X_v)} \geq c} \quad \text{ if } c > \smref ,\\
       \pr{ \E_{L_n} \sqpr{\sum_{v=1}^\k\spech(X_o,X_v)} \leq c} \quad \text{ if } c < \smref . 
    \end{dcases}
\end{align*}
We refer to such an event as an  {\em atypical consensus}, with $\cc < \smref$ referring to events with low consensus and $\cc > \smref$ referring to events with  high consensus. Our first result, Theorem \ref{thm:Ising-phase-transition} presented in Section \ref{subs:Isingresult}, 
states that when  conditioned on an atypical consensus, $\{U_n\}$ asymptotically always behaves like an Ising measure on $T^\k$ with a suitable inverse temperature and  external field. The nature of the  Ising measure depends both on whether one is conditioning on low or high consensus, and whether the original mark distribution is biased or unbiased.
Additionally, the most likely asymptotic conditional behavior of the   
joint distribution of the marks $\{X_i\}_{i \in [n]}$ conditioned on an atypical consensus as described in  Section \ref{subs:connecting-back-to-finite-graphs}.  
 In order to  state these results, we first define the class of  Ising measures on infinite regular trees in Section \ref{subs-Ising}  and summarize some of their relevant properties in Section \ref{subs-Isingprops}. 

\subsubsection{Ising measure on infinite regular trees}
\label{subs-Ising}

In this section we introduce the  class of Ising measures that appear as conditional limits. We begin with the definition of Ising measures (equivalently, Ising Gibbs measures) on regular trees, which are Markov random fields (MRF) on regular trees with specifications defined below.  Recall that $\partial \Lambda$ denotes the boundary of a subset $\Lambda$, as stated in \eqref{eqn:graph-boundary}.

\begin{definition}[Ising measures on the infinite regular tree]\label{def:ising-DLR}
    An Ising measure on $T^\k$ with interaction parameter $\beta\in \R$ and external field $B\in \R$, abbreviated to Ising $(T^\k,\beta, B)$ measure, is a probability measure $\ising\in \P\cpr{\pr{1,-1}^{T^\k}}$ such that for any finite set $\Lambda \subset T^\k$ and 
    $\omega_{\Lambda^c}\in \pr{1,-1}^{\Lambda^c}$, 
    the  conditional marginal distribution $\ising_\Lambda (\cdot|\omega_{\Lambda^c})$, of $\ising$ on $\Lambda$ given $\omega_{\Lambda^c}$,  is equal to 
    $\gamma_\Lambda (\cdot |\omega_{\partial \Lambda})$, where the family of kernels $\gamma_\Lambda$ are defined as follows: for any $\sigma_\Lambda\in \pr{1,-1}^\Lambda$, 
    \begin{equation} 
    \label{defn:gamma}
\gamma_\Lambda \cpr{\sigma_\Lambda|\omega_{\partial \Lambda}} :=\frac{1}{Z_\Lambda(\beta,B,\omega)}\exp \pr{\beta\sum_{\stackrel{\pr{u,v}\in E(T^\k)}{u,v\in \Lambda}}\sigma_u\sigma_v+B\sum_{u\in \Lambda}\sigma_v+\beta\sum_{\stackrel{\pr{u,v}\in E(T^\k)}{u\in \Lambda,v\in \Lambda^c}}\sigma_u\omega_v},
    \end{equation}
    where $Z_\Lambda(\beta,B,\omega)$ is the normalization constant that makes $\gamma_\Lambda(\cdot|\omega_{\partial \Lambda})$ a probability measure.  
 For fixed  $\omega_{\partial \Lambda} \in \pr{1,-1}^{\partial \Lambda}$, we refer to  the measure $\gamma_\Lambda\cpr{\cdot|\omega_{\partial \Lambda}}$  in $\P\cpr{\pr{1,-1}^\Lambda}$ as the finite volume Ising measure on $\Lambda$ with interaction parameter $\beta$, external field $B$ and boundary condition $\omega_{\partial \Lambda}$. 
\end{definition}

\begin{remark}
The regimes $\beta>0$ and $\beta < 0$, correspond, respectively, to ferromagnetic and antiferromagnetic Ising measures with inverse temperature $|\beta|$.
\end{remark}

\subsubsection{Ising measures on infinite regular trees}
\label{subs-Isingprops}

    It is well known that there exist multiple Ising measures on $T^\k$ for certain subset of parameters $(\beta,B)$ referred to as the co-existence regime.   When characterizing  them,   a special role is played by the following map. 
    
    \begin{definition}[Ising Cavity Map]
        \label{def:cavitymap}
        Given   $(\k,\beta,B) \in \{2,3,  \ldots \} \times \R^2,$ the Ising cavity map 
        \[ \Gamma=\Gamma(\k,\beta,B):\mathbb R\rightarrow \mathbb R
        \]
        is 
        defined by
\begin{equation}\label{eqn:Ising-cavity-map}
    \Gamma(\exth)\coloneqq B+(\k-1)\tanh^{-1}(\tanh(\beta)\tanh(\exth)).
\end{equation}
\end{definition}

\begin{remark}
    \label{rem:cavitymap}
Using \cite[Proposition 12.24]{georgii2011gibbs}, every fixed point $\exth$ of the cavity map $\Gamma$ corresponds to an Ising $(T^\k, \beta, B)$ measure $\ising = \ising_\theta$ on $T^\k$. 
In fact, the Ising measures obtained in this manner are called {\em translation-invariant splitting (TIS) Ising measures} on $T^\k$ (or {\em completely homogeneous Markov chains} in \cite{georgii2011gibbs}). We provide more details on the latter class of Gibbs measures in a more general setting in  Section \ref{subs:TISGM}. 
\end{remark}

The next lemma summarizes well known properties of fixed point(s) of the Ising cavity map $\Gamma$ in different parameter regimes. These are subsequently used to characterize  corresponding Ising measures.

\begin{lemma}[Known Properties of the Ising Cavity Map]
\label{lem:cavitymap}
Fix $\k\geq 2$ and $B\in \R$. In the ferromagnetic regime $\beta>0$, there exists a critical threshold $\betacrit=\betacrit(\k,B)>0$
 such that the map $\Gamma = \Gamma(\k,\beta,B)$ has a unique fixed point $\exth^*$ if and only if $(\beta, B)$ lie in the uniqueness regime 
 \begin{equation}
        \label{regime:unique}
           \pr{(\beta',B'): 0<\beta'<\betacrit(\k,B)\text{ or }(\beta',B')=(\betacrit(\k,0),0)}.
 \end{equation}
  Moreover, the following  properties hold in the ferromagnetic regime: 
\begin{enumerate}
\item 
Suppose $B > 0$. Then $\Gamma$  has a  unique fixed point $\exth^+$ in   the interval  $(0,\infty)$, and   
\begin{enumerate} 
\item 
if $\beta = \betacrit$,  then 
$\Gamma$ has exactly one additional fixed point $\exth^- < 0$; 
\item 
if $\beta > \betacrit$,   then 
$\Gamma$ has exactly two additional fixed points  $\exth^-<\exth^\sharp<0$. 
\end{enumerate}
  \item Suppose  
 $B < 0$. Then $\Gamma$  has a  unique fixed point  $\exth^-$  in the interval  $(-\infty,0)$, and 
 \begin{enumerate}
     \item 
 if $\beta = \betacrit$,  then 
$\Gamma$ has exactly one additional fixed point $\exth^+ > 0$; 
 \item 
 if $\beta > \betacrit$, then $\Gamma$ 
has exactly two additional fixed points  $0 < \exth^\sharp < \exth^+$.
\end{enumerate}
\item 
When  $B = 0$,  $\exth^\sharp = 0$ is always a fixed point of $\Gamma$, 
and $\Gamma$ has exactly three fixed points $\exth^-<\exth^\sharp=0<\exth^+$ if 
 $\beta  >\betacrit = \betacrit(\k,0)$.
\end{enumerate}
  Furthermore, $\betacrit(2,B)=\infty$ for every  $B\in \R$, and for $\k\geq 3$,
\begin{itemize}
        \item $\betacrit(\k,0)=\tanh^{-1}(\frac{1}{\k-1})$;
        \item $\betacrit(\k,B)=\betacrit(\k,-B)$, for every $B\in \R$.
\end{itemize}

In the antiferromagnetic regime $\beta<0$, the map $\Gamma=\Gamma(\k,\beta,B)$ has a unique fixed point $\exth^\sharp$ for all $B\in \R$. In particular, $\sgn(\exth^\sharp)=\sgn(B)$.
\end{lemma}
\begin{proof}
We begin with the ferromagnetic regime $\beta>0$. First, \eqref{regime:unique} follows from \cite[Lemma 12.27(i)]{georgii2011gibbs}, on noting that $\beta^*(\k, B)$ is the 
inverse of the map $J\mapsto h(J,d)$ with $J=\beta$, $h=B$ and $d=\k-1$. For Case (1), the existence and uniqueness of $\exth^+$ in $(0,\infty)$ is proved in \cite[Lemma 2.3]{DemMon10}. Together with \cite[Lemma 12.27(ii)-(iii)]{georgii2011gibbs} and the fact that $\exth=0$ is not a fixed point if $B>0$, we conclude (1a) and (1b). Case (2), including (2a) and (2b), follows from Case (1) using the   spin-flip symmetry. For Case (3), the fact that $\theta^\sharp = 0$ is a fixed point of $\Gamma$ follows easily on inspection of the map \eqref{eqn:Ising-cavity-map}. The existence of $\exth^-$ and $\exth^+$ in the non-uniqueness regime follows from \cite[Lemma 12.27(iii)]{georgii2011gibbs}. The value of the critical threshold $\betacrit$ when $B=0$ is a consequence of the computations on the uniqueness threshold in \cite[(12.28) and (12.30)]{georgii2011gibbs}. The fact that $\betacrit$ is symmetric on $B$ is due to \cite[Lemma 12.27]{georgii2011gibbs}.

In the antiferromagnetic regime $\beta<0$, since the Ising cavity map $\Gamma$ is a decreasing function, there is always a unique fixed point $\exth^\sharp$. In particular, since $\Gamma(0)=B$, we see that when $B>0$, the unique fixed point is in the interval $(0,\infty)$; When $B<0$, the unique fixed point is in the interval $(-\infty,0)$; When $B=0$, the fact that $\exth^\sharp=0$ is the unique fixed point follows easily on inspection of $\Gamma.$ 
\end{proof}

Below is a summary of all $\TIS$ Ising measures on $T^\k$ in different regimes; also see Figure \ref{fig:TIS-Ising-characterization}.  They will appear in the main result.

\begin{definition}[Characterization of all finite temperature $\TIS$ Ising measures on $T^\k$]\label{def:Ising-on-tree}
    Given $\k \in \mathbb N, \k \geq 2,$ $0\neq \beta\in \R$, $B\in \R$, and $\Gamma = \Gamma (\k, \beta, B)$. In the ferromagnetic regime $\beta>0$, we let
    \begin{enumerate}
    \item[(1a)] 
         $\ising_\k^*(\beta,B)$ denote the  Ising measure on $T^\k$ corresponding to the fixed point $\exth^*$ when $(\beta, B)$ lie in the uniqueness regime \eqref{regime:unique}; 
          \item[(1b)] $\ising_{\k}^{\sharp}(\beta,B)$ denote the 
        Ising measure 
        on $T^\k$ corresponding to the fixed point $ \exth^\sharp$ of $ \Gamma$ 
        when 
        $\beta > \betacrit$; 
        \item[(1c)] $\ising_\k^+(\beta,B)$ as the  Ising measure on $T^\k$  corresponding to the fixed point $\exth^+$ of $\Gamma$ when $B > 0$ (positive external field) or $\beta>\beta_*$ and $B= 0$ (zero external field in the non-uniqueness regime) or $\beta\geq \beta_*$ and $B<0$ (negative external field in the non-uniqueness regime); 
\item[(1d)] 
$\ising_\k^-(\beta,B)$ 
        as the Ising measure on $T^\k$ corresponding to the fixed point $\exth^-$ of $\Gamma$ when $B < 0$ (negative external field) or $\beta>\beta_*$ and $B= 0$ (zero external field in the non-uniqueness regime) or $\beta\geq \beta_*$ and $B>0$ (positive external field in the non-uniqueness regime);
        \end{enumerate}
        In the antiferromagnetic regime $\beta<0$, we let 
        \begin{itemize}
            \item[(2)] $\ising_\k^\sharp(\beta,B)$ denote the Ising measure on $T^\k$ corresponding to the unique fixed point $\exth^\sharp$.
        \end{itemize}
\end{definition}

\begin{remark} 
\label{rem-Isingmodels}
It is clear from Lemma \ref{lem:cavitymap} and Definition \ref{def:Ising-on-tree} that  
in the uniqueness regime \eqref{regime:unique}, 
$\ising_\k^*(\beta,B)=\ising_\k^+(\beta,B)$ when $B>0$, and $\ising_\k^*(\beta,B)=\ising_\k^-(\beta,B)$ when $B<0$, and also that 
when $0 < \beta \leq  \betacrit(\k,0),$ 
$\ising_\k^\sharp(\beta, 0)  = \ising_\k^*(\beta, 0)$.
\end{remark}

\begin{figure}[h]
    \centering
    \includegraphics[width=0.6\linewidth]{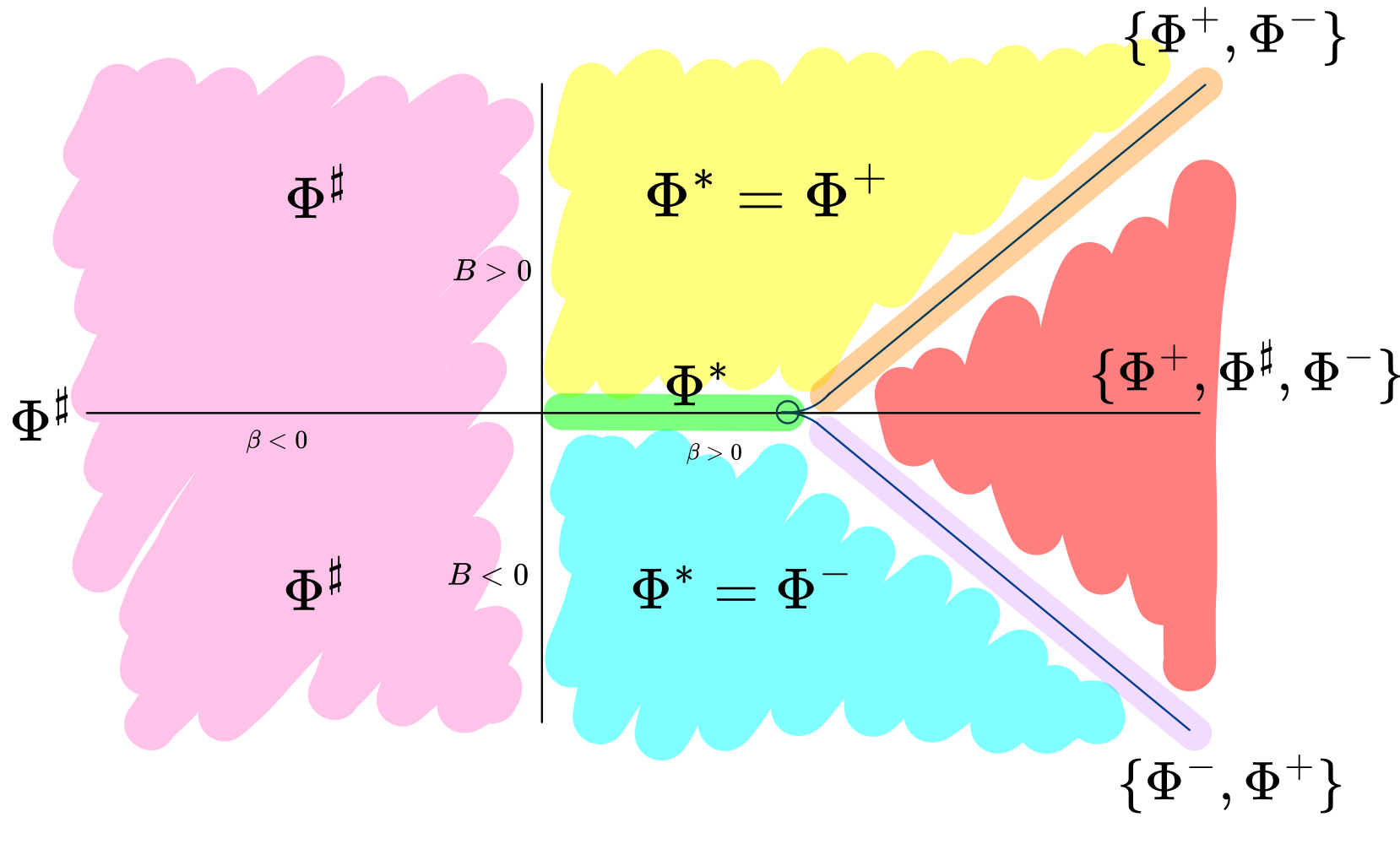}
    \caption{All $\TIS$ Ising measures in different parameter regimes $(\beta,B)\in \R\times \R$.}
    \label{fig:TIS-Ising-characterization}
\end{figure}

\begin{remark}
    In the ferromagnetic regime $\beta > 0$,  $\ising_\k^+$ is also called the Ising measure on $T^\k$ ``with plus boundary condition'' since it is the monotone  decreasing limit,  as $r \rightarrow \infty,$ of the finite volume Ising measures on $T^\k_r$, the depth-$r$ subtree of $T^\k$ defined in Section \ref{subsubsec:topology}, with boundary condition $\omega_{\partial \Lambda_r}\equiv 1$ (see, e.g., \cite[(2.5)]{MonMosSly12}).     
    Analogously, the Ising measure $\ising_\k^-$ is also called the Ising measure ``with minus boundary condition''. 
\end{remark}

\begin{remark}
    Note that the Ising measures defined above are elements in $\P\cpr{\X^{T^\k}}$ and will be identified as elements in $\P\cpr{\TkaXinf}$, according to Remark \ref{rmk: labeling scheme}.
\end{remark}

In order to state the main results in a unified fashion,  
we will go beyond Definition \ref{def:Ising-on-tree} and also consider 
Ising measures on $T^\k$ with ``freezing'' temperature, that is, when  $|\beta|=\infty$. 
In what follows let $T^{\k,+}$ (respectively, $T^{\k,-}$)  denote the infinite $\k$-regular tree with all marks being 
$1$ (respectively, $-1$).  Also, let $T^{\k,\pm}$ be the infinite $\k$-regular tree with root mark $1$ and alternating vertex marks, that is, each vertex $v\in T^\k$ has mark $1$ if $d(v,o)$ is even and mark $-1$ if $d(v,o)$ is odd. 
Likewise, let $T^{k,\mp}$ be the infinite $\k$-regular tree with root mark $-1$ and alternating vertex marks.

\begin{definition}[Freezing temperature Ising measures on $T^\k$] \label{def:freezing-ising} 
         For any $B\in \R$, the ferromagnetic Ising measure on $T^\k$ with freezing temperature and 
         plus/minus boundary condition are defined as follows: 
    \begin{equation}
        \ising_\k^+(\infty,B)=\ising_\k^+(\infty)\coloneqq \delta_{T^{\k,+}},\quad \ising_\k^-(\infty,B)=\ising_\k^-(\infty)\coloneqq \delta_{T^{\k,-}}.\label{eqn:freezing-ferro}
    \end{equation}
    
         In the antiferromagnetic regime, for any $B\in \R$, extend the definition of the unique $\TIS$ Ising measure $\ising_\k^\sharp$ to the following antiferromagnetic Ising measure on $T^\k$ with freezing temperature: 
    \begin{equation}
        \ising_\k^{\sharp}(-\infty,B)=\ising_\k^{\sharp}(-\infty)\coloneqq \frac{1}{2}\delta_{T^{\k,\pm}}+\frac{1}{2}\delta_{T^{k,\mp}}.\label{eqn:freezing-anti}
    \end{equation}
\end{definition}

Having introduced candidate conditional limits,  we can  state our first conditional limit theorem, given atypical consensus. 

\subsubsection{The conditional limit of the component empirical measure for two-spin systems}
\label{subs:Isingresult}

For $\delta > 0,$ define  $c_\delta$ to be $\min\{c+\delta,  \k\}$ if $c > \smref$ and $\max\{c - \delta, -\k\}$ if $c < \smref.$ Recall that the mark distribution is $\nu=\Ber(\prone)$, for some $p \in (0,1)$.
We now state our first main result, whose 
proof  is  relegated to Section \ref{sec:Ising-phase-transition-proof}.
\begin{theorem}[Gibbs conditioning principles and phase transitions for two-spin systems]\label{thm:Ising-phase-transition}
Fix $0\neq c\in [-\k,\k],$ and recall the definition of $\smref$ from \eqref{eqn:smref}. Then the following statements hold:

\begin{enumerate} 
    \item If $\cc \in [-\k, \smref)$, then there exists a unique $\beta = \beta (c,p) \in [-\infty,0)$ such that 
    \begin{align*}
    \lim_{\delta \downarrow 0} \lim_{n \to \infty} \Prob(U_n \in \cdot  \mid \Conset_n(c_\delta)) =\delta_{ \ising_\k^\sharp(\beta,B)},
    \end{align*}
    where $B=\log \cpr{p/(1-p)},$ and $\ising_\k^\sharp(\beta, B)$ is as specified in Definition \ref{def:Ising-on-tree}(2) and \eqref{eqn:freezing-anti}. Moreover, $\beta(-\k,p)=-\infty$ and the unique $\beta$ is such that the consensus of $\ising_\k^\sharp(\beta,0)$ is $\cc$.
    \item If $c \in (\smref, \k]$ and $p\neq 1/2$,  then the following is true:
    \begin{enumerate}
        \item If $p>1/2$, then there exists a unique $\beta = \beta (c,p) \in (0,\infty]$ such that
        \begin{align*}
        \lim_{\delta \downarrow 0} \lim_{n \to \infty} \Prob(U_n \in \cdot  \mid \Conset_n(c_\delta)) =\delta_{ \ising_\k^+(\beta,B)},
        \end{align*}
        where $B=\log \cpr{p/(1-p)}>0$ and $\ising_\k^+(\beta, B)$ is as specified in Definition \ref{def:Ising-on-tree}(1c) and \eqref{eqn:freezing-ferro}. Moreover, $\beta(\k,p)=\infty$ and the unique $\beta$ is such that the consensus of $\ising_\k^+(\beta,B)$ is $\cc$;
        \item If $p<1/2$, then there exists a unique $\beta = \beta (c,p) \in (0,\infty]$ such that
        \begin{align*}
        \lim_{\delta \downarrow 0} \lim_{n \to \infty} \Prob(U_n \in \cdot  \mid \Conset_n(c_\delta)) =\delta_{ \ising_\k^-(\beta,B)},
        \end{align*}
        where $B=\log \cpr{p/(1-p)}<0$ and $\ising_\k^-(\beta, B)$ is as specified in Definition \ref{def:Ising-on-tree}(1d) and \eqref{eqn:freezing-ferro}. Moreover, $\beta(\k,p)=\infty$ and the unique $\beta$ is such that the consensus of $\ising_\k^-(\beta,B)$ is $\cc$;
    \end{enumerate}
    \item If $c \in (\smref, \k]$  and $p=1/2$, in which case $\smref=0$,  then the following is true:
    \begin{enumerate}
        \item If, in addition, $c\in (0, \k/(\k-1)]$, then $U_n$ given $\Conset_n(c)$ converges to the unique Ising measure $\ising_\k^*(\tanh^{-1}(\cc/\k),0)$ as specified in Definition \ref{def:Ising-on-tree}(1a). In particular, the consensus of $\ising_\k^*(\tanh^{-1}(c/\k),0)$ is $\cc$;
        \item If $c \in (\k/(\k-1),\k]  $, then there exists a unique $\beta = \beta(c,0) \in (\betacrit(\k,0),\infty]$ such that  $U_n$ given $\Conset_n(c)$ asymptotically lies in $\{\ising_\k^+\cpr{\beta,0},\ising_\k^-\cpr{\beta,0}\}$, where $\ising_\k^+(\beta,0)$ and $\ising_\k^-(\beta,0)$ are as specified in Definition \ref{def:Ising-on-tree}(1c)(1d) and \eqref{eqn:freezing-ferro}. In other words,  for any $\varepsilon > 0,$ we have
    \begin{align*}
        \lim_{\delta \downarrow 0} \limsup_{n \to \infty} \frac{1}{n} \log \Prob(U_n \notin B_\varepsilon(\pr{\ising_\k^+\cpr{\beta,0},\ising_\k^-\cpr{\beta,0}}) \mid \Conset_n(c_\delta)) < 0,
    \end{align*}
    where $B_\varepsilon(\{\ising_\k^+\cpr{\beta,0},\ising_\k^-\cpr{\beta,0}\})$ denotes an $\varepsilon$-neighborhood of $\{\ising_\k^+\cpr{\beta,0},\ising_\k^-\cpr{\beta,0}\}$ in $\P(\TkaXinf).$
    Furthermore, $\beta(\k,0)=\infty$ and the unique $\beta$ is such that the consensus of $\ising_\k^+(\beta,0)$ is $\cc$. 
        
    \end{enumerate}
\end{enumerate}

\end{theorem}

\begin{remark}
\label{rmk:external-field-p}
Note that in all cases, the temperature parameter $\beta$ depends on both $\cc$ and $p$, whereas the external field $B=B(p)$ is always the same function on $p$. Moreover, the specific external field $B(p)=\log(p/(1-p))$ can be viewed as merely representing the original product  measure of the marks in the following way: on any finite graph $G$, when $B=B(p),$ 
\begin{align}
    \ising_G(\beta,B)(\sigma)&\coloneqq \frac{1}{Z_G}\exp\pr{\beta\sum_{\pr{u,v}\in E(G)}\sigma_u\sigma_v+B\sum_{v\in G}\sigma_v},\quad \forall \sigma\in \pr{1,-1}^{|G|}\label{eqn:ising-finite-graph}\\
    &\propto \exp\pr{\beta\sum_{\pr{u,v}\in E(G)}\sigma_u\sigma_v} \prod_{v\in G}\nu(\sigma_v),\nonumber
\end{align}
For this reason, we will also refer to the mark distribution $\nu$ as the reference measure.
\end{remark}

The first assertion of Theorem \ref{thm:Ising-phase-transition}  
is a conditional limit theorem for $\{U_n\}$, which states that when the consensus is smaller  than typical, the sequence $\{U_n\}$ conditioned on the atypical consensus converges to the unique $\TIS$ antiferromagnetic Ising measure with a specific external field and inverse temperature. It is important to note that there are multiple antiferromagnetic Ising measures on $T^\k$ whenever $|\beta|$ is large enough (depending on $B$), with  two prominent examples being the period-$2$ extremal measures, a.k.a. the ``chessboard phases''. However, since any (local weak) limit of a sequence of finite graphs  has to be ``unimodular'' (this property is called ``sofic'' in the literature), and unimodular measures are translation-invariant, the two chessboard phases cannot be obtained as any (local weak) limit of finite graph sequences. This explains the translation-invariant property of the conditional limit. On the other hand, the fact that the conditional limit is splitting is related to the fact that we are studying atypical events involving only pairwise interactions, as will be clear from the proof.  

The second assertion of Theorem \ref{thm:Ising-phase-transition} considers the setting where consensus is atypically high, and the reference measure $\nu$ is non-uniform (equivalently, $p\neq 1/2$). In this case, the sequence $\{U_n\}$ conditioned on an atypical consensus converges to the ferromagnetic Ising measure with a specific inverse temperature and external field with a boundary condition that has the same bias as  the reference measure $\nu$, specifically, plus boundary condition if $p > 1/2$  and minus boundary condition if $p < 1/2$. It is worth noting that there are multiple $\TIS$ Ising measures on $T^\k$ for  when  $\beta$ is large enough given $B$. Nonetheless, the conditional limit is still unique since the asymmetry of the reference measure $\nu$ makes one of them more likely to occur.

The above asymmetry is not present when the true mark distribution $\nu$ is uniform, which leads to an interesting phase transition   in the nature of the  conditional limit   of $\{U_n\}$  
  with respect to the atypical consensus value $\cc$.  
  As specified in (3a), there is still a unique conditional limit if (and only if) the consensus is positive but not too large, that is, $0=\smref< \cc\leq \kappa/(\kappa-1)$. In particular, this  limit is symmetric with respect to  spin flips,  and in this case, the optimal way to exceed the  consensus value $\cc$ is by only increasing the probability of agreements between neighboring spins.
    Note that the threshold $\k/(\k-1) = \tanh\cpr{\betacrit(\k,0)}$  is the 
    maximal consensus that can be achieved by any ferromagnetic Ising measure on $T^\k$ with no external field in the uniqueness regime. 
    Thus, as shown in (3b),  a consensus  greater than $\kappa/(\kappa-1)$  can  be achieved by a ferromagnetic Ising measure only in the non-uniqueness regime. Moreover, among the multiple Ising measures in the non-uniqueness regime, including $\ising_\k^-,\ising_\k^\sharp$ and $\ising_\k^+$, Theorem \ref{thm:Ising-phase-transition} suggests that the most probable way of achieving this atypical consensus $c$ is by breaking the symmetry between the $1$ and $-1$ spins and selecting $\ising_\k^-$ and $\ising_\k^+$. In fact, these two Ising measures appear equally likely (due to the spin-flip symmetry of the reference measure $\nu$, the definition of the consensus function $m$, and the Ising measures with no external field) and it is not possible to determine which one is the conditional limit solely based on the atypical consensus value $c$.

\subsubsection{Connecting back to the conditional behavior of  the joint  mark distribution}\label{subs:connecting-back-to-finite-graphs}

While Theorem \ref{thm:Ising-phase-transition} provides  a characterization of $\{U_n\}$ under an atypical consensus in terms of various Ising measures on the infinite regular tree, we are also interested in the conditional structure of the joint distribution of the marks $\{X_i\}_{i\in [n]}$. In this section, we try to fill this gap by deducing a reasonable model 
that is consistent with the conditional limits given an atypical empirical consensus. 

We propose the following {\em posterior model} on $n$-vertex graphs:

\begin{protocol}
    Sample a random $\k$-regular graph $G_n$ on $n$ vertices and assign the $\{1,-1\}$ spins 
    $X^n=\{X_i\}_{i \in [n]}$ to each vertex according to the Ising measure on (the random graph) $G_n$ with 
    external field $B= B(p) = \log(p/(1-p))$ and 
        the unique temperature parameter $\beta = \beta (c,p) \in \R$ specified in Theorem \ref{thm:Ising-phase-transition}.   In other words,  $\Law(X^n\mid G_n)=\ising_{G_n}(\beta,B)$, where $\ising_{G_n}$ is the Ising measure on $G_n$ defined in \eqref{eqn:ising-finite-graph}.
\end{protocol}

We now show that the Posterior Model 
can be rigorously justified in several regimes, including when $\cc > \smref$,  with the most subtle case being in the setting $p=1/2$ where there is a phase transition,  addressed in 
1(d) below. Recall the various Ising models specified in Definition \ref{def:Ising-on-tree}.

\begin{lemma}\label{lemma:posterior-consistent}
    Fix $p \in (0,1)$ and $\cc\in (-\k,\k)$ such that $\cc\neq \smref$. Let $(G_n,X^n)$ be sampled from the Posterior Model 
    with associated $B = B(p)$  and $\beta=\beta(\cc,p)$  as specified therein. 
    Then the local weak limit of the sequence $\{(G_n,X^n)\}_{n\in \N}$  is consistent with the asymptotic conditional limit of $U_n$, given  the event $\Conset_n(c)$, in the following cases:
    \begin{enumerate}
        \item[(1a)] If $\cc>\smref$ and $p>1/2$, then $(G_n,X^n)$ converges locally weakly to $\ising_\k^+(\beta,B)$ a.s.;
        \item[(1b)] If $\cc>\smref$ and $p<1/2$, then $(G_n,X^n)$ converges locally weakly to $\ising_\k^-(\beta,B)$ a.s.;
        \item[(1c)] If $\cc\in (0,\k/(\k-1)]$ and $p=1/2$, then $(G_n,X^n)$ converges locally weakly to $\ising_\k^*(\beta,B)$ a.s.;
        \item[(1d)] If $\cc\in (\k/(\k-1),\k]$ and $p=1/2$, then $(G_n,X^n)$ asymptotically lies in $\{\ising_\k^+(\beta,0),\ising_\k^-(\beta,0)\}$ with high probability
        , that is, for any $\varepsilon>0$,
        \begin{equation*}
            \lim_{n\rightarrow \infty}\bbP\cpr{U(G_n,X^n)\in B_\varepsilon\cpr{\pr{\ising_\k^+(\beta,0),\ising_\k^-(\beta,0)}}}=1.
        \end{equation*}
        \item[(2)] If $\cc<\smref$ and there exists a unique Ising $(T^\k,\beta,B)$ measure, then $(G_n,X^n)$ converges locally weakly to $\ising_\k^\sharp(\beta,B)$ a.s.. 
    \end{enumerate}
\end{lemma}
\begin{proof}
    Note that, by \cite[Proposition 2.5]{DemMon10b}, random $\k$-regular graphs $G_n$ converges locally weakly to $T^\k$ on a set of full measure $\Omega\setminus \mathcal{N}$.
    Below, we justify the stated convergence of $(G_n(\omega),X^n)$, whose law is given by $\ising_{G_n(\omega)}(\beta,B)$, for every $\omega\in \Omega\setminus \mathcal{N}$ in each of the cases:

    {\em Case (1a):} It follows from  \cite{DemMon10} (see also  the introduction of \cite{MonMosSly12})  that since $\beta>0$ and $B>0$, $\ising_{G_n(\omega)}(\beta,B)$ converges locally weakly to $\ising_\k^+(\beta,B)$.
    
    {\em Case (1b):} This follows from (1a) and the spin-flip symmetry.

    {\em Case (1c):} In this case, Theorem \ref{thm:Ising-phase-transition}(3a) implies that $\beta(\cc,p)=\tanh^{-1}(\cc/\k)$ and $B=0$. By \cite[Theorem 12.31]{georgii2011gibbs}, there is a unique Ising $(T^\k,\beta,0)$ measure $\ising_\k^*(\beta,0)$. Therefore, \cite[Proposition 2.14]{LacRamWu23} implies that $\ising_{G_n(\omega)}(\beta,0)$ converges locally weakly to $\ising_\k^*(\beta,0)$. 

    {\em Case (1d):} In this case, Theorem \ref{thm:Ising-phase-transition}(3b) implies that $\beta(\cc,p)> \betacrit(\k,0)$. Note that \cite[Theorem 2.4]{MonMosSly12} implies that $\ising_{G_n(\omega)}=\ising_{G_n(\omega)}(\beta,0)$ converges locally in probability to $\frac{1}{2}\ising_\k^+(\beta,0)+\frac{1}{2}\ising_\k^-(\beta,0)$. At first, this may appear  to suggest that the $(G_n,X^n)$ limit does not align  with the conditional limit set for $U_n$  specified in Theorem \ref{thm:Ising-phase-transition}(3b). However, \cite[Theorem 2.4]{MonMosSly12} also provides a  refinement of the convergence result, which shows that whenever the graph sequence $G_n(\omega)$ further satisfies a certain edge expander property (which holds with high probability for random $\k$-regular graphs as $n$ grows to infinity), the conditional measure $\ising_{G_n(\omega),+}\coloneqq \ising_{G_n(\omega)}(\cdot \mid \sum_{v\in G_n(\omega)}X_v>0)$ converges to $\ising_\k^+(\beta,0)$ locally in probability. In particular, under $\ising_{G_n(\omega)}$, the two events $\{\sum_{v\in G_n(\omega)}X_v>0\}$ and $\{\sum_{v\in G_n(\omega)}X_v<0\}$ occur equally likely with probability close to $1/2$. As a result, with probability close to $1/2$ (when $G_n(\omega)$ satisfies a certain edge expander property and $\sum_{v\in G_n(\omega)}X_v>0$), $\ising_{G_n(\omega)}$ can be approximated by $\ising_\k^+(\beta,0)$, and with probability close to $1/2$ (when $G_n(\omega)$ satisfies a certain edge expander property and $\sum_{v\in G_n(\omega)}X_v<0$), $\ising_{G_n(\omega)}$ can be approximated by $\ising_\k^-(\beta,0)$. 
This completes the justification of Case (1d).

    {\em Case (2):} In this case, \cite[Proposition 2.14]{LacRamWu23} and the assumption that there is a unique Ising $(T^\k,\beta,B)$ measure imply that $\ising_{G_n(\omega)}(\beta,B)$ converges locally weakly to the unique Ising measure $\ising_\k^\sharp(\beta,B)$.
\end{proof}

\begin{remark}
In the remaining regime not covered by Lemma \ref{lemma:posterior-consistent}, 
Theorem \ref{thm:Ising-phase-transition}(1)
shows that the conditional limit of 
$U_n$ given  the event $\Conset_n(c)$, is still concentrated on the 
unique $\TIS$ Ising measure $\ising_\k^\sharp$. on $T^\k$.
   Despite this,  the description on the $n$-vertex graph in this regime is more complicated because the antiferromagnetic Ising measure on random regular graphs is known to exhibit a ``replica-symmetric-breaking (RSB) phase transition'' at a  
     certain inverse temperature value $\beta_{RSB} < 0$ that is lower than the uniqueness threshold (e.g., see \cite{CojLoiMezSor22}, \cite{ZdeBoe10} and references therein.) 
     For instance, when $B=0$, it is shown in \cite{CojLoiMezSor22} that $\beta_{RSB}=-\tanh^{-1}(1/\sqrt{\k-1})$, which is smaller than the Gibbs uniqueness threshold  
     $\beta=-\tanh^{-1}(1/(\k-1))$ (in the absence of an external field). For a reason similar to that given to explain   Theorem \ref{thm:Ising-phase-transition}(1) in Section \ref{subs:Isingresult}, we believe that by the sofic nature of limits, the Posterior Model should still be the right answer in the regime $\beta \in (\beta_{RSB}, 0)$. However, in the regime   $\beta < \beta_{RSB}$ (also referred to as  the condensation regime),  the global structure of the graph becomes dominant and  the question of picking the most likely model on the $n$-vertex graph conditional on such a low consensus is less straightforward. 
\end{remark}

\subsection{General Spin Systems}\label{subsec:result-general-spin}

We now describe some results on the conditional  behavior of multi-spin systems conditioned on a more general class of rare events beyond atypical consensus.  
Throughout, $\X$ is an arbitrary finite set, and $\nu \in \P(\X)$ is the mark distribution, which we assume without loss of generality to satisfy  $\nu (x) > 0$ for all $x \in {\mathcal X}$. Note that we shall identify $\nu\in \P(\X)$ with its probability mass function (p.m.f.) depending on the context.

\subsubsection{Rare Events of Interest and Conditional Limit Laws}

The class of atypical or rare events we consider is described in terms of a {\em symmetric}
function $h:\X^2\to \R$ that we refer to as the {\em  edge potential}. 
Define the corresponding  {\em (local) $h$-consensus}  $S_h:\TkappaX{1}\to \R$ 
to be
\begin{align}
\label{def-interact}
    S_h = S_h(\tau,X)\coloneqq\sum_{v =1 }^\kappa  h(X_o,X_v), \quad (\tau,X)\in \TkappaX{1}. 
\end{align}
Then for any measure $\mu\in 
\P (\TkappaX{1})$, 
 the {\em  $h$-consensus} under $\mu$ is equal to 
 \begin{align}
\label{eqn;avg-h-local}
    \E_{\mu}\sbrac{S_h} = \E_{\mu}\sbrac{\sum_{v=1}^{\k} h(X_o,X_v)}=\k\E_{\pi_\mu}\sbrac{h(X_o,X_1)},
\end{align}
where the last equality holds 
due to the exchangeability of the leaf marks under $\mu$ (see Remark \ref{rmk: labeling scheme}), and $\pi_\mu$ denotes the edge marginal, that is, the joint law of $(X_o,X_1)$ under $\mu.$  
Hence,  \eqref{eqn;avg-h-local} shows that  $\E_{\mu}\sbrac{S_h}$ depends on $\mu$ only through its edge marginal.    
When the measure $\mu$ is the random empirical measure $L_n$, we refer to the random variable $\E_{L_n}\sbrac{S_h}$ as the {\em empirical $h$-consensus}.  Note that by \eqref{remark:true-law} and the fact that $\TkappaX{1}$ is finite, as $n \rightarrow \infty$, 
the empirical $h$-consensus converges in probability 
to a deterministic limit: 
\begin{align}
     \label{eqn:ceta}
\E_{L_n}  \sbrac{S_h} \rightarrow 
 \E_{\eta_1}\sbrac{S_h} := \cc_{h, \texttt{ref}}. 
\end{align}
We refer to the limit 
$\cc_{h, \texttt{ref}}$ as the {\em typical  $h$-consensus}.  When $h(x,y) = xy$, we recover the previous consensus functional \eqref{def-spech}.

We are interested in the behavior of $U_n$, conditioned on seeing an atypical empirical $h$-consensus, as made precise below.

\begin{definition}[Atypical empirical $h$-consensus]
\label{def-rareevent}
Given an edge potential $h$ and $\cc\in \R$ with $ \cc \neq \sref$, 
let the event $\Conset_n(h,\cc)$ be defined by 
\begin{align}
    \Conset_n(h,c)\coloneqq
    \begin{cases}
       \{ \E_{L_n} [S_h] \geq \cc\} \quad \text{ if } \cc > \sref,\\
       \{ \E_{L_n} [S_h] \leq \cc\} \quad \text{ if } \cc < \sref. 
    \end{cases}
    \label{eqn:atypical-emp-h-interaction}
\end{align}
and let the collection of measures $\Conset(h,\cc)$ be defined by
\begin{align}
    \Conset(h,c)\coloneqq
    \begin{cases}
       \{ \mu\in \P\cpr{\TkappaX{1}}:\E_{\mu} [S_h] \geq \cc\} \quad \text{ if } \cc > \sref,\\
       \{ \mu\in \P\cpr{\TkappaX{1}}:\E_{\mu} [S_h] \leq \cc\} \quad \text{ if } \cc < \sref. 
    \end{cases}
    \label{eqn:atypical-h-interaction}
\end{align}
\end{definition}
Also, in what follows, 
for any $c \neq \sref$ and  $\delta > 0$, 
set 
\begin{align}
\label{sdelta}
c_\delta := 
\begin{cases}
     c + \delta  \quad \text{ if } c > \sref,  \\
 c - \delta \quad \text{ if } c < \sref.  
 \end{cases}
 \end{align}

We are interested in  conditional 
limit laws of the following kind.

\begin{definition}
 \label{def-condlim} 
 Fix an edge potential $h$, and $c \in \R$, $c \neq \sref.$ Let $\Conset_n(h,c)$ and  $(c_\delta)_{\delta > 0}$ be  as in Definition \ref{def-rareevent} and  \eqref{sdelta}, respectively. 
\begin{enumerate}
\item Given $\rho^*\in \P(\Gstar[\X])$, we say  that {\em   $U_n$  given $\Conset_n(h,c)$ converges to $\rho^*$ } if  
   \begin{align}
\lim_{\delta\downarrow0}\lim_{n\to \infty}\mathbb{P}(U_n\in \cdot\mid \Conset_n(h, c_\delta))=\delta_{\rho^*}(\cdot).
\end{align}
    \item 
For any $\mathcal{U}\subset \P(\Gstar[\X]),$ we say that {\em  $U_n$ given $\Conset_n(h,c)$ lies in $\mathcal{U}$ asymptotically} if for every open set $U\supset \mathcal{U}$, the following holds:
    \begin{align}
        \lim_{\delta\downarrow 0}\limsup_{n\to \infty}\frac{1}{n}\log\mathbb{P}\left(U_n\notin U\ \vert\ 
        \Conset_n(h,c_\delta) \right)<0.
    \end{align} 
\end{enumerate}
\end{definition}

\subsubsection{Results for General Spin Systems and Rare events}\label{subs:TISGM}
 
In the general setting we will show that, 
similar to the case of atypical consensus, the conditional limits under an atypical $h$-consensus are again Gibbs measures on trees. However, not all Gibbs measures on trees  appear as  conditional limits. Below, we define the subclass of Gibbs measures on trees that appear as  candidate conditional limit points.

\begin{definition}[Translation-invariant splitting Gibbs measure on infinite trees]
\label{def:TISGM}
    Given an infinite (locally finite) tree $T$, a positive symmetric function $\psi:\X\times \X\rightarrow (0,\infty)$ and a positive function $\bar\psi:\X\rightarrow (0,\infty)$, we say that a measure $\pmb\rho\in \P(\X^{T})$ is a splitting Gibbs measure on $T$ with state space $\X$ and specification $(\psi,\bar\psi)$ if there exists a collection of probability measures $(\ell_v)_{v\in T}$ in $\P(\X)$, called the ``boundary law'' or ``entrance law'', such that for each $r\in \N$ and $x\in \X^{T_r}$,
    \begin{align}\label{eqn:marginal-splitting-Gibbs}
        \pmb\rho_{T_r}(x)=\frac{1}{Z_{T_r}}\prod_{\pr{u,v}\in E(T_r)}\psi(x_u,x_v)\prod_{v\in T_{r-1}}\bar\psi(x_v)\prod_{v\in T_r\setminus T_{r-1}}\ell_v(x_v), 
    \end{align}
    where $T_r$ is the depth-$r$ subtree of $T$ defined in Section \ref{subsubsec:topology}, $\pmb\rho_{T_r}$ is the marginal of $\pmb\rho$ on $T_r$ and $Z_{T_r}$ is the normalizing constant. 
    When $T=T^\k$ for some $\k\geq 2$, we say that a splitting Gibbs measure on $T^\k$ is translation-invariant if the corresponding boundary law is homogeneous, that is, $\ell_v=\ell$,  for all $v\in T^\k$ for some $\ell\in \P(\X)$. Denote the collection of translation-invariant splitting ($\TIS$) Gibbs measures on $T^\k$ with state space $\X$ and specification $(\psi,\bar\psi)$ by $\TISGM^{\psi,\bar\psi}_{T^\k}$.
\end{definition}

\begin{remark}
    While our terminology agrees with that used in  \cite{Roz13,KulRozKha14},  
    in \cite[Chapter 12]{georgii2011gibbs} (and \cite{Pre74,Spi75,Zac83,Zac85}), 
    splitting Gibbs measures on trees are also referred to as {\em Markov chains on trees},   $\TIS$ Gibbs measures on trees are also called {\em completely homogeneous Markov chains on trees} 
    and $T^\k$ is also referred to as {\em the Cayley tree of degree $\k-1$}. 
     Characterizations of splitting Gibbs measures for various specifications have been extensively studied,  for example, see \cite{georgii2011gibbs,Roz13,Spi75} for Ising measures and  \cite{KulRozKha14,Roz21} for Potts measures. 
\end{remark}

\begin{remark}
    In order to guarantee  consistency of  marginals, that is, $(\pmb \rho_{T_{r+s}})_r=\pmb\rho_{T_r}$ for every $r,s\in \N$, the boundary laws $(\ell_v)_{v\in T}$ must satisfy a certain consistency property, defined in \cite[Definition 12.10]{georgii2011gibbs}, which is used as the definition of {\em boundary laws} therein. For the purpose of this paper, we only state this consistency property for homogeneous boundary laws when $T=T^\k$ in the following lemma
    and postpone its proof to Appendix \ref{sec:pf-of-TISGM-consistency}.
     
\end{remark}

\begin{lemma}[Characterizing edge marginals via boundary laws] \label{lemma:TISGM-to-cavity}
    Given any $\pmb\rho\in \TISGM^{\psi,\bar\psi}_{T^{\k}}$, its boundary law $\ell\in \P(\X)$ is a fixed point of the recursion $\BP:\P(\X)\rightarrow \P(\X)$ defined by
    \begin{align}\label{eqn:cavity}
        \BP\ell(x)=\frac{1}{\zeta}\bar\psi(x)\cpr{\sum_{z\in \X}\psi(x,z)\ell(z)}^{\k-1},\quad \forall x\in \X,
    \end{align}
    where $\zeta$ is the normalizing constant. In fact, there is a one-one correspondence between the set $\TISGM^{\psi,\bar\psi}_{T^\k}$ and the set of fixed points $\Delta^*\coloneqq\pr{\ell\in \P(\X): \BP\ell=\ell}$. 
    Moreover, for any $\ell\in \Delta^*$, the edge marginal $\pi_{\pmb \rho_1}$ of the TISGM $\pmb \rho$ corresponding to $\ell$ satisfies  
    \begin{align}\label{eqn:edge-marginal-cavity}
        \pi_{\pmb\rho_1}(x,z)\propto \ell(x)\psi(x,z)\ell(z),\quad \forall x,z\in \X.
    \end{align}
\end{lemma}

\begin{remark}\label{remark:TISGM-is-non-degenerate}
    Due to Lemma \ref{lemma:TISGM-to-cavity}, any boundary law $\ell$ must have full support on $\X$ since the specifications $\psi$ and $\bar\psi$ are both positive; otherwise \eqref{eqn:cavity} cannot be satisfied. Moreover, due to \eqref{eqn:edge-marginal-cavity}, the edge marginal (and thus the vertex marginal) of any $\pmb\rho \in \TISGM^{\psi,\bar\psi}_{T^\k}$ must have full support on $\X^2$ (and $\X$, respectively).
\end{remark}

\begin{remark}
    Once again, note that the splitting Gibbs measures defined above are elements of  $\P\cpr{\X^{T^\k}}$ and can also be viewed as elements of $\P\cpr{\TkaXinf}$, according to Remark \ref{rmk: labeling scheme}.
\end{remark}

We now state our second main result, which characterizes  conditional limits of $U_n$ under atypical $h$-consensus. First, set 
\[ h_{\min}\coloneqq \min_{x,z\in \X}h(x,z) \quad \mbox{  and } \quad h_{\max}\coloneqq \max_{x,z\in \X}h(x,z). 
\]
Also, let $\Psym{\X}$ denote the subset of symmetric measures in $\P(\X^2)$, that is,
\begin{align}  
\label{def-pps}
\Psym{\X} := \{ \pi \in \P(\X^2): \pi(x, y) = \pi(y,x), \quad  \forall x,y \in \X\}, 
\end{align} and  for $c \in \R,$ define \begin{equation}\label{eqn:Medge}
        \BB_h(c)\coloneqq \pr{\pi\in \Psym{\X}:\k\bbE_{\pi}\sqpr{h(X_o,X_1)}=c}. 
    \end{equation}
Recall the definition of the map $\mathsf{Range}$ in \eqref{eqn:range}.

\begin{theorem}[Conditional limits are  $\TIS$ Gibbs measures] \label{thm:Gibbs-deg-and-nondeg}
    Fix $\nu\in \P(\X)$. Let $h$ be an edge potential and fix $c\in (\k h_{\min},\k h_{\max})\setminus \mathsf{Range}(\k h)$  such that $c\neq \sref$. Then $U_n$ given $\Conset_n(h,\cc)$ lies in $\Gibbsdeg(\cc)$ asymptotically (in the sense of Definition \ref{def-condlim}), where
    \begin{align}\label{eqn:TIGSMs-deg-and-nondeg}
        \Gibbsdeg(\cc)\coloneqq \bigcup_{\X^{\downarrow}\subseteq \X}\bigcup_{\beta\in \R}\TISGM^{\exp (\beta h^\downarrow),\nu^\downarrow}_{T^\k}\bigcap 
        \left\{ \rho \in \P (\Tstar^\kappa[\X]): \pi_{\rho_1} \in \BB_h(c)\right\},
    \end{align}  
   $h^\downarrow\coloneqq h\vert_{\X^\downarrow\times \X^\downarrow}:\X^\downarrow\times \X^\downarrow\rightarrow \R$ and $\nu^\downarrow=\nu\cpr{\cdot|\X^\downarrow}\in \P(\X^\downarrow)$, identified with its p.m.f. 
   \end{theorem}

 To prove this result, we use the LDP for the component empirical measure for random $\k$-regular graph marked with i.i.d.\,vertex marks (see Theorem \ref{thm:ldp-Un}) together with the Gibbs conditioning principle for large deviations in Proposition \ref{prop:gibbs cond},  to reduce the problem of identifying the conditional limits to the analysis of a certain constrained optimization problem involving the rate function, on the space of probability measures on marked trees.  We then simplify this problem (in Section \ref{sec:simplification}) and study the simplified optimization problem (in Sections \ref{subs:opt1}--\ref{subs:opt2}) to conclude the specific form of the conditional limits in \eqref{eqn:TIGSMs-deg-and-nondeg}; see Section \ref{subs:pf-of-Gibbs-deg-nondeg} for the proof of this theorem.

\begin{remark}
    The assumption $c\notin  \mathsf{Range}(\k h)$ in Theorem \ref{thm:Gibbs-deg-and-nondeg} is a technical condition that ensures that conditional limits are $\TIS$ Gibbs measures with {\it positive temperature}, that is, with $\beta\in \R$. We believe that this assumption can be removed by enlarging the set  $\Gibbsdeg(\cc)$ to include  $\TIS$ Gibbs measures on $T^\k$ with {\it freezing temperature}, that is, with $\beta=\infty$ or $\beta = -\infty$, just as was done in Theorem \ref{thm:Ising-phase-transition}. However, since these values of $\cc$ correspond to  boundary cases, we choose not to  include them in the main statement; they could be analyzed separately as in the proof of Theorem \ref{thm:Ising-phase-transition}.
\end{remark}

\begin{remark}\label{remark:def-of-degenerate}
   Note that the set $\Gibbsdeg(\cc)$ in \eqref{eqn:TIGSMs-deg-and-nondeg} is comprised of $\TIS$ Gibbs measures on various state spaces whose $h$-consensus is $\cc$. This includes  those whose 
  state space is the original spin space $\X$ and lie in $\TISGM^{\exp(\beta h),\nu}_{T^\k}$, which we refer to as {\em non-degenerate} $\TIS$ Gibbs measures, 
  as well as those that have  state space $\X^\downarrow \subsetneq \X$, and lie  
 in $\TISGM^{\exp(\beta h^\downarrow),\nu^\downarrow}_{T^\k}$, which we refer to as {\em degenerate}  $\TIS$ Gibbs measures. 
  By Remark \ref{remark:TISGM-is-non-degenerate}, note the 
     non-degenerate measures are precisely those whose root  marginal  has full support, and the      degenerate ones are those whose root marginal is supported on 
     a strict subset $\X^\downarrow\subsetneq \X$. 
     We will use the terms degenerate/non-degenerate exclusively for this purpose.
\end{remark}

While Theorem \ref{thm:Gibbs-deg-and-nondeg} provides a very general result, it will be clear from the proof that  specific knowledge of the edge potential $h$ or mark distribution $\nu$
 could be used to identify a smaller subset of $\Gibbsdeg(\cc)$ that contains all conditional limits. 
Indeed, 
in the  special case  when $\X = \{1,-1\}$ and $h(x,y) = m(x,y) = xy$ is the consensus functional, Theorem \ref{thm:Ising-phase-transition} provides such a  refinement, and identifies the precise conditional limit in all regimes.
In particular, it shows that for every constraint value $\cc \in (-\k, \k)$,  conditional limits are always non-degenerate Gibbs measures.   This naturally raises the question of whether a more general refinement in this direction is possible, leading to the following:\\

\noindent 
{\bf Open Question. } 
For which spin systems $\X$, mark distributions $\nu$ and edge potentials $h$ are conditional limits always non-degenerate?  \\

\subsubsection{Non-degeneracy of Conditional Limits}
\label{subsub:degencondlim}

We show in Section \ref{subs:opt2} that degenerate conditional limits can arise for constraint values $c \in (\k h_{\min}, \k h_{\max})$  even in simple two-spin systems. 
This suggests that 
the resolution of the open problem (stated in the last section) is non-trivial.   Nevertheless, our next result makes partial progress towards the open question by identifying a broad subset of two-spin systems for which  conditional limits are always non-degenerate. 
 
\begin{theorem}[Non-degeneracy for two-spin systems with uniform mark distribution]  \label{thm:Gibbs-component-two-spin-uniform} 
    Fix $|\X|=2$,  $\nu=\mathsf{Unif}(\X)$, an edge potential $h$, and let $c\in (\k h_{\min},\k h_{\max})$ be such that $\cc\neq\sref$. Then $U_n$ given $\Conset_n(h,\cc)$ lies in $\Gibbsnondeg(\cc)$ asymptotically, where
    \begin{align}\label{eqn:Gibbs-limit-splitting}
        \Gibbsnondeg(\cc)\coloneqq\bigcup_{\beta\in \R}\TISGM^{\exp\pr{\beta h},\nu}_{T^\k}\bigcap 
         \left\{ \rho \in \P (\Tstar^\kappa[\X]): \pi_{\rho_1} \in \BB_h(c)\right\}.
    \end{align}
\end{theorem}

The proof of this theorem is given in Section \ref{sec:verification}, as a corollary to a more general result (see Proposition \ref{prop:constrained-min-interior}) that shows that degenerate conditional limits do not arise for a class of multi-spin systems that satisfy a certain property 
 (see \eqref{eqn:cond} of Proposition \ref{prop:constrained-min-interior} in Section \ref{sec:sufficient-condition-for-nondegeneracy}).  
We believe that this class should include all multi-spin systems  with a uniform mark distribution, and in fact have the following conjecture.

\begin{conjecture}[Non-degeneracy for multi-spin systems with uniform mark distribution]
\label{conj:Gibbs-component}
    Fix any finite set $\X$ and edge potential $h$. If 
        $\nu= \mathsf{Unif}(\X)$, then for every  $c\in (\k h_{\min},\k h_{\max})$,  $U_n$ given $\Conset_n(h,\cc)$ lies asymptotically in the set $\Gibbsnondeg(\cc)$ defined in \eqref{eqn:Gibbs-limit-splitting}. 
  \end{conjecture}

\section{From conditional limits to optimization problems}
\label{sec:cond-to-opt}
In this section, we reduce the analysis of conditional limit theorems to  certain constrained optimization problems using the theory of large deviations and Gibbs conditioning principles.  The main result of this section is given in Section \ref{subs:finalred} (see Proposition \ref{prop:optsimp} therein).  Sections \ref{sec:ldp-graph}-\ref{sec:simplification} 
contain several preliminary  results that are required for the proof of Proposition \ref{prop:optsimp}, which is given in Section \ref{subs:pf-of-prop-optsimp}.

\subsection{Reduction to an optimization problem over edge measures}
\label{subs:finalred}

\subsubsection{Unimodular extension}\label{sec:unimodular-extension}

To define the optimization problem, we first need to define a certain natural extension of probability measures on $\TkappaX{1}$ to probability measures on $\TkaXinf$. Let $\mu \in \P(\TkappaX{1})$ be such that $\pi_\mu$ is symmetric. 
First, given any $\tilde{\tau} \in  \Tstarone^{\k-1}[\X]$ 
and $x' \in \X$, we let $\tilde{\tau} \oplus x'$ denote the tree in $\Tstarone^\k[\X]$
obtained by adding to $\tilde{\tau}$  one extra leaf (i.e., vertex attached to the root) with the mark $x'$.  Note that  the map $\tilde{\tau} \mapsto \tilde{\tau} \oplus x'$ in fact maps $\Tstarone^{\k -1}[\X]$ to $\Tstarone^{\k}[\X]$ for every $\k \geq 1$. 
Also,   define  
\begin{align*}
    E_1(x',x)(\tau) \coloneqq\{v \sim o : (\tau_v, \tau_o) = (x',x)\}, \quad x,x'\in \X, \tau \in \Tstarone^{\k}[\X], 
\end{align*}
to be equal to the number of leaves of $\tau$ with mark $x'$ when the root mark is $x$, and set it equal to zero otherwise.  Then define   the conditional measure on stars,  
$\widehat{P}_\mu(x, x'),$  by 
\begin{align}
\widehat{P}_\mu(x, x')(\TBdef) = \frac{\indf_{\{\TBdef_o  = x\}}}{\k \pi_\mu(x,x')} \mu (\TBdef \oplus x') E_1(x', x)(\TBdef \oplus x'), \, \quad  \TBdef \in \Tstarone^{\k-1}[\X].
\label{eqn:PU}
\end{align} 
It can be verified that $\widehat{P}_\mu(x,x') \in \P(\Tstarone^{\k-1} [\X])$ is the conditional law (under $\mu$) of the rooted subtree obtained by removing a randomly chosen  leaf  conditioned on the root mark being $x$ and the mark on the leaf being $x'$ (see, e.g., \cite[Lemma A.2]{RamYas23}).   
Recall the true law $\eta$ in  Remark \ref{remark:true-law} and the definition of the depth-$h$ marginal in Definition \ref{def:rho_h}. 

Set  $\mu_1=\mu$ and for $h \geq 2,$ iteratively define $\mu_h\in \P(\TkappaX{h})$ as follows:
\begin{align}
\label{muh}
    \frac{d\mu_h}{d\eta_h}(\tau) := \frac{d\mu_{h-1}}{d\eta_{h-1}} (\tau_{h-1}) \prod_{v: \, d_{\tau}(o,v) = h-1} \frac{d\widehat{P}_\mu( X_v, X_{a(v)} )}{d\widehat{P}_{\eta_1}(X_v,X_{a(v)})} ( \tau(v \setminus a(v))_1), \quad \tau=(\tau,X) \in \TkappaX{h},
\end{align}
where $a(v)$ denotes the parent of $v$ (as defined in Section \ref{subsubsec:topology}) and $\tau(v \setminus a(v))_1$ denotes the depth-1 tree consisting of $v$ and all its neighbors expect $a(v)$, with root at $v.$ It can be shown that $\mu_h$ is indeed a probability measure on $\TkappaX{h}$. Note that $\widehat{P}_{\eta_1}(x,x')$ does not depend on $x'$.  
Also, observe that $\{\mu_h;h\in \N\}$ forms a consistent family of probability measures on marked rooted trees. Thus by the Kolmogorov extension theorem, the following is well defined.

\begin{definition}
\label{defn:UniExt}
Given $\mu \in \P(\TkappaX{1})$, the unimodular extension of $\mu$, denoted 
$\UGW_1(\mu)$,  is defined to be the unique  probability measure in $\P(\TkaXinf)$  such that for every $h \in \mathbb{N},$ 
its marginal on $\TkappaX{h}$ coincides with the measure 
$\mu_h$ defined in \eqref{muh}.  
 \end{definition}
 
Note that the terminology ``extension'' is justified by the fact that 
$(\UGW_1(\mu))_1 = \mu.$

\subsubsection{The edge optimization problem}

Recall the definition of 
$\Psym{\X}$ from \eqref{def-pps}, and 
 let $J_\k^\nu:\Psym{\X}\rightarrow [0,\infty)$ be defined by 
\begin{align}
J_\k^\nu(\pi) \coloneqq H(\piz\|\nu)+\frac{\kappa}{2}H(\pi\|\piz\otimes \piz),\quad \forall \pi\in \Psym{\X}.\label{eqn:J-nu-k}
\end{align}
Note that $J^\nu_\k$ is well defined since $\X$ is a finite set and relative entropies are nonnegative. In particular, 
\begin{equation}\label{eqn:unique-minimizer-of-J}
    J^\nu_\k(\pi)=0 \quad \text{ if and only if }\quad  \pi=\nu\otimes \nu.
\end{equation}
Moreover, given an ``edge measure" $\pi\in \Psym{\X}$, define 
$\mup \in \P \cpr{\TkappaX{1}}$ by 
\begin{align}
 \frac{d\mup}{d\eta_1}(\tau) := \frac{d\piz}{d\nu}(X_o)\prod_{v =1}^\k   \frac{d\pic}{d\eta_1^{1|o}} (X_v\mid X_o), \quad  \tau=(\tau,X)\in \TkappaX{1}. \label{eqn:1MRF}
\end{align}
Note that $\eta_1^{1|o}(x|x')=\nu(x)$ for any $x,x'\in \X$.

We first establish a connection between $\mu^{(\pi)}$ and $\TIS$ Gibbs measures.

\begin{lemma}[$\TIS$ Gibbs measures as unimodular extensions]\label{lemma:unimodular-extension-of-splitting-Gibbs-measure}
For any $\TIS$ Gibbs measure $\pmb{\rho}$ on $T^\k$ with some given specification    
we have 
\begin{equation}\label{eqn:Gibbs-MRF-1}
    \pmb{\rho}_1=\mu^{(\pi_{\pmb\rho_1})},
\end{equation}
with $\mu^{(\pi_{\pmb\rho_1})}$ as defined in \eqref{eqn:1MRF}, and 
\begin{equation}\label{eqn:Gibbs-uni-ext}
    \pmb{\rho}=\UGW_1(\pmb{\rho}_1),
\end{equation}
where $\UGW_1$ is the map given by Definition \ref{defn:UniExt}.
\end{lemma}
 The proof is provided in Appendix \ref{sec:proof-unimodular-extension-of-splitting-Gibbs-measure}.

We now present the main result of this section.

\begin{proposition}[Reduction to an optimization over edge measures]
\label{prop:optsimp}
For all $c\neq  \sref,$  define
\begin{align}
\label{opt:pi}
    \RRR_{edge} (\cc) := \min_{\pi\in \BB_{h}(\cc) }J^\nu_\k(\pi)\quad \text{ and } \quad \optset_{edge}(\cc)\coloneqq \argmin_{\pi\in \BB_h(\cc)}J^\nu_\k(\pi)
\end{align}
with $\BB_{h}(\cc)$ being the constraint set defined in \eqref{eqn:Medge}, 
and let
\begin{equation}\label{eqn:opti-pi-to-opti-rho}
    \optset(\cc):=\pr{\UGW_1(\mu^{(\pi)}):\pi\in \optset_{edge}(\cc)}.
\end{equation}
Then for any open set $\openset \subset\P(\TkaXinf)$ with $\optset (c) \subset \openset$, the following conditional limits hold:
\begin{align}
    &\lim_{\delta\to 0}\limsup_{n\to \infty}\frac{1}{n}\log\mathbb{P}\left(U_n\notin \openset \ \vert \ \{\mathbb{E}_{L_n}\sqpr{S_h}\geq c-\delta\right)<0.\label{eq:general cond limit}\\
    &\lim_{\delta\to 0}\limsup_{n\to \infty}\frac{1}{n}\log\mathbb{P}\left(U_n\notin \openset \ \vert \ \{\mathbb{E}_{L_n}\sqpr{S_h}\leq c+\delta\right)<0.\label{eq:general cond limit other way}
\end{align}
\end{proposition}

We refer to the optimization problem $\RRR_{edge}(\cc)$ in \eqref{opt:pi} as the {\em edge optimization problem}. The proof of this result is given in Section \ref{subs:reduction}.    It relies on a LDP for random regular graphs with i.i.d. vertex marks that is stated in Section \ref{sec:ldp-graph}, and a related Gibbs conditioning principle stated in Section \ref{subs-Gibbscond}. The latter establishes   \eqref{eq:general cond limit}-\eqref{eq:general cond limit other way} with $\optset(\cc)$ expressed as the minimizer of a certain optimization problem over measures on marked $\kappa$-regular trees, see \eqref{eqn:Iinf-optimization}. In Section \ref{sec:simplification}, we simplify the latter optimization problem to show it has the representation given in Proposition \ref{prop:optsimp}.

\subsection{LDPs for the neighborhood and component empirical measure sequence}
\label{sec:ldp-graph}

We now state a LDP associated with the sequence of $\P(\TkappaX{1})$-valued neighborhood empirical measures  $\{L_n\}$ defined in Section \ref{sec:empirical-measures}. Recall the ``true law'' $\eta_1 \in \P(\TkappaX{1})$  defined in  Remark \ref{remark:true-law}, and the relative entropy functional from \eqref{eqn:relative-entropy}.  Also, recall that $\pi_\mu \in \P(\X^2)$ denotes the edge marginal under  $\mu \in \P(\TkappaX{1}),$ that is, the law of $(X_o,X_1)$ when $(\tau, X)$ has law $\mu.$ Define  $I^\k_1: \P(\TkappaX{1})\to [0,\infty]$ by  

\begin{align}
     I^\k_1(\mu) :=\begin{dcases}
         H(\mu\|\eta_1)-\frac{\kappa}{2}H(\pi_{\mu}\|\pi_{\eta_1})  &\text{if $\pi_\mu$ is symmetric,}\\ \infty, &\text{otherwise,}
     \end{dcases}
      \label{def-Ikk}
\end{align}
where $H$ is the relative entropy defined in \eqref{eqn:relative-entropy}.
We now claim that $H(\pi_\mu\|\pi_{\eta_1})<\infty$ and hence, 
$I^\k_1$ is well defined. To see why the claim holds,  note that since $\nu(x) > 0$ for all $x \in \X$ and $\pi_{\eta_1} = \nu \otimes \nu,$ it follows that $\pi_\mu \ll \pi_{\eta_1}$ for all $\mu$. Together with the fact that  $\pi_\mu$ is supported on the finite set $\X^2$, this implies the claim. The following result is proved in \cite[Theorem 3.2] {RamYas23}.  

 \begin{theorem}(LDP for the neighborhood  empirical measure sequence  \cite[Theorem 3.2]{RamYas23}).
 \label{thm:ldp-ln}
     The sequence $\{L_n\}$ satisfies an LDP on $\P(\TkappaX{1})$ with rate function $I^\k_1.$ In other words, $I^\k_1: \P(\TkappaX{1}) \rightarrow [0,\infty]$ is a lower semicontinuous function such that the following holds: 
     \begin{enumerate}
         \item (Compactness of level sets). For any $M < \infty,$ $\{\mu:I^\k_1(\mu) \leq M\}$ is a compact subset of $\P(\TkappaX{1})$;
         \item (LDP lower bound). For any open set $\mathcal{O} \subset \P(\TkappaX{1}),$ we have
         \begin{align*}
             \liminf_{n \to \infty} \frac{1}{n} \log \mathbb{P}(L_n \in \mathcal{O}) \geq -\inf_{\mu \in \mathcal{O}} I^\k_1(\mu);
         \end{align*}
         \item (LDP upper bound). For any closed set $\mathcal{F} \subset \P(\TkappaX{1}),$ we have
         \begin{align*}
             \limsup_{n \to \infty} \frac{1}{n} \log \mathbb{P}(L_n \in \mathcal{F}) \leq -\inf_{\mu \in \mathcal{F}} I^\k_1(\mu).
         \end{align*}
     \end{enumerate}
 \end{theorem}

\begin{remark}[Form of the rate function]
    \label{rem:ldpproof}
Versions of this LDP result were  
established earlier in \cite{BalPerRei26, BalOlietal22} using the colored configuration model introduced in \cite{BorCap15} and its marked version in \cite{DelAna19}.   
 However, the form of the rate functions obtained therein involved many terms that were expressed in terms of some combinatorial quantities. 
 The more succinct form of the rate function given in \eqref{def-Ikk} is crucial for our analysis in this paper.
\end{remark}

\subsubsection{LDP for the component empirical measure sequence}

Next, we state the LDP for $\{U_n\}.$ Recall the unimodular extension defined in Section \ref{sec:unimodular-extension}.
\begin{theorem}[{(LDP for the component empirical measure sequence  \cite[Theorem 4.13]{RamYas23})}]
 \label{thm:ldp-Un}
     The sequence $\{U_n\}$ satisfies an LDP on $\P(\Gstar[\X])$ with rate function $I^\k.$ Moreover, $I^\k(\rho)=\infty$ if $\rho\notin \P(\TkaXinf)$. Furthermore, for any $\mu\in \P(\TkappaX{1})$ such that $\pi_\mu$ is symmetric, we have $I^\k_1(\mu) < \infty$ and
         \begin{align}\label{eqn:component-to-neighborhood}
         I^\k_1(\mu)=\min\pr{I^\k(\rho):\rho_1=\mu}, \quad \text{ and }\quad \UGW_1(\mu)=\argmin\pr{I^\k(\rho):\rho_1=\mu}.
     \end{align}
 \end{theorem} 
An explicit form for the rate function $I^\k$ is given in \cite[Section 4]{RamYas23}. However,  since this exact form of $I^\k$ is not used in our analysis, 
for simplicity in the last theorem we only stated the properties of $I^\k$  relevant to our analysis.

\subsection{Gibbs conditioning principle}
\label{subs-Gibbscond}

In this section, we state a Gibbs conditioning principle associated with the sequence $\{U_n\}$.  
Let 
 $\A_1(c)$ be the subset of $\Conset(h,c)$ for which $\pi_{\mu} \text{ is symmetric}$, or more explicitly, define   
 \begin{align}
     \A_1(c) \coloneqq
    \begin{cases}
       \{ \mu\in \P\cpr{\TkappaX{1}}:\E_{\mu} [S_h] \geq \cc \text{ and } \pi_{\mu} \text{ is symmetric}\}  \quad \text{ if } \cc > \sref,\\
       \{ \mu\in \P\cpr{\TkappaX{1}}:\E_{\mu} [S_h] \leq \cc \text{ and } \pi_{\mu} \text{ is symmetric}\} \quad \text{ if } \cc < \sref. 
    \end{cases}
    \label{def:A1}
\end{align}
Also, define the associated set  
\begin{align}
\label{def:A}
    \A(c) & \coloneqq \pr{\rho\in \P(\TkaXinf): \rho_1\in \A_1(\cc)}.
\end{align}
We now state the main result of this section.

\begin{proposition}[Gibbs conditioning principle] \label{prop:gibbs cond}
    Let $c\in \R$ and let $\optset(c)  \subset\P(\TkaXinf)$ be the set of minimizers of the constrained optimization problem 
    \begin{align}
 \RRR (c) :=   \min_{\rho \in {\mathcal A} (c)} I^\k (\rho), 
   \label{eqn:Iinf-optimization}
   \end{align}
   where $I^\k$ is the large deviation rate function of $\{U_n\}$ from Theorem \ref{thm:ldp-Un}. 
Then for any open set $\openset \subset\P(\TkaXinf)$ with $\optset (c) \subset \openset$, the conditional limits \eqref{eq:general cond limit}-\eqref{eq:general cond limit other way} hold. In other words, if $c > \sref$ (resp. $c < \sref$)  
then $\{U_n\}$ asymptotically lies in $\optset(c)$ given $\{\mathbb{E}_{L_n}\sqpr{S_h}\geq c-\delta\}$ (resp. $\{\mathbb{E}_{L_n}\sqpr{S_h}\leq c-\delta\}$).
\end{proposition}
\begin{proof}
      Our setting satisfies the assumptions required for the usual Gibbs conditioning principle for large deviations; see, for example, \cite[Theorem 7.1]{Leonard2010EntropicProjections} and \cite[Proposition 6]{KimRam18}. Indeed, the following holds:
\begin{enumerate}
    \item By Theorem \ref{thm:ldp-Un}, the component empirical measure $\{U_n\}$ satisfies the LDP on $\P(\Gstar[\X])$ with a rate function $I^\k$ whose domain lies in  $\P(\TkaXinf).$
    \item The mapping $\P(\TkaXinf) \ni \rho\mapsto \mathbb{E}_{\rho_1}\sqpr{S_h}\in \R$ is linear and continuous. This is immediate from \eqref{eqn;avg-h-local}.

    \item The constraint set $\mathcal{A}=\mathcal{A}(\cc)$ satisfies 
    \begin{enumerate}
        \item $\inf_{\rho\in \A}I^\k(\rho)=\inf_{\mu\in \A_1}I^\k_1(\mu)<\infty$ by Theorem \ref{thm:ldp-Un};
        \item $\A$ is closed: if $\rho^{(k)} \to \rho$ then $\rho^{(k)}_1\rightarrow \rho_1$, thus the bounded convergence theorem implies  $\Exp_{\rho^{(k)}_1}[S_h]=\Exp_{\rho^{(k)}_1}[S_h] \to \Exp_{\rho}[S_h]$. Consequently, if $\rho^{(k)} \in \A$ for all $k,$ then $\rho \in \A.$
        \item Let $\delta>0$ and define $\A_\delta\coloneqq\{\rho\in \P(\TkaXinf)\colon \mathbb{E}_{\rho_1}[S_h]\geq \cc-\delta\text{ and }\pi_{\rho_1}\text{ is symmetric}\}.$ Then $\A_\delta$ is closed in $\P\cpr{\TkaXinf}$ and one has:\\
        (i) $\A=\cap_{\delta>0} \A_\delta. $ Indeed, if $\rho \in \A$ then $\pi_{\rho_1}$ is symmetric and $\Exp_{\rho_1}[S_h] \geq c$, which implies $\Exp_{\rho_1}[S_h] \geq c - \delta$ for all $\delta >0.$ Consequently $\rho \in \A_\delta$ for all $\delta > 0$ and hence $\rho \in \cap_{\delta > 0} \A_\delta.$ Conversely, if $\rho \in \A_\delta$ for all $\delta > 0$ then $\Exp_{\rho_1}[S_h] \geq c - \delta$ for all $\delta >  0.$ Letting $\delta \to 0$ we obtain $\Exp_{\rho_1}[S_h] \geq c$ and it follows that $\rho \in \A.$

        (ii) For any $\delta > 0$ we have $\mathbb{P}(U_n\in \A_\delta)=\mathbb{P}\cpr{\bbE_{L_n}\sqpr{S_h}\geq \cc-\delta}>0$ for large enough $n$. Indeed, pick any $\rho \in \A_\delta$, then since $\pi_{\rho_1}$ is symmetric, by \cite[Lemma 3.5]{BalPerRei26} with depth $k=2$ and the fact that the marginalization map $\rho_2\mapsto \rho_1$ is continuous, there exists a sequence of marked $\k$-regular graphs $ \{(F_n, x^n)\}$  such that their corresponding neighborhood empirical measures $\ell_n = \frac{1}{n} \sum_{v \in F_n}\delta_{\cl_v(F_n, x^n)}$ converges to $\rho_1$ in $\P(\TkappaX{1})$. 
        Therefore, for any fixed $n\in \N$ that is large enough, we have $\ell_n\in \A_\delta$ and $\mathbb{P}\cpr{\bbE_{L_n}\sqpr{S_h}\geq \cc-\delta}\geq \bbP\cpr{L_n=\ell_n}>0$.

        \item $\A$ is contained in the interior of $\A_\delta$ for all $\delta>0.$ Indeed, if $\rho \in \A$ and $\rho_k \to \rho$ then the  bounded convergence theorem yields $\Exp_{\rho^{(k)}_1}[S_h] \to \Exp_{\rho_1}[S_h],$ and in particular, for any $\delta > 0$ there exists $K_1 < \infty$ such that for all  $k \geq K_1$, $\Exp_{\rho^{(k)}_1} \geq \Exp_{\rho_1}[S_h] - \delta \geq c - \delta.$ This implies $\rho$ lies in the interior of $\A_\delta$ for any $\delta > 0.$
    \end{enumerate}
\end{enumerate}
An application of \cite[Theorem 7.1]{Leonard2010EntropicProjections} then implies that \eqref{eq:general cond limit} holds.   Carrying out the same argument but replacing $h$ and $\cc$ with $-h$ and $-\cc$, respectively, we similarly obtain  \eqref{eq:general cond limit other way}.
\end{proof}

\subsection{Simplification of the optimization problem}
\label{sec:simplification}

We now simplify the optimization problem for $\RRR_1 (c)$ given in \eqref{eqn:Iinf-optimization}. For simplicity, we treat only the case $c>\sref$; the case $c<\sref$ follows analogously.

\subsubsection{Reduction to an optimization over edge marginals}

First, note that the set $\A_1(c)$ defined in 
\eqref{def:A1} can be rewritten in terms of the set  $\BB_h(c)$ defined in 
\eqref{eqn:Medge} as follows: 
\begin{equation}
    \label{eq:A1alt}
\A_1(c) = \pr{ \mu \in \P(\TkaXinf): \pi_\mu \in \bigcup_{c' \geq c} \BB_{h}(\cc')}. 
\end{equation}
Also, using the relations in \eqref{eqn:component-to-neighborhood}, note that \eqref{eqn:Iinf-optimization} reduces to 
\begin{align}
    R(c) = \min_{\mu \in \A_1(c)} I^\k_1(\mu). \label{eqn:I1_optimization}
\end{align}
Next, use \eqref{eq:A1alt}, substitute the definition of $I_1^{\k}$  from \eqref{def-Ikk} into \eqref{eqn:I1_optimization} and the fact that $\pi_{\eta_1} = \nu \otimes \nu$ by Remark \ref{remark:true-law} to obtain 
\begin{align}
\label{opt:red1}
\RRR (\cc)  &= 
\min_{\cc' \geq \cc} \min_{\pi \in \BB_{h}(\cc')}  \min_{\mu \in \P(\TkaXinf): \pi_\mu = \pi}  I^\k_1 (\mu) \\
\label{opt:red2}
 &= \min_{\cc' \geq \cc} \min_{\pi \in {\BB}_{h} (\cc')} 
\left[ \min_{\mu \in  \P(\TkaXinf): \pi_\mu = \pi} H(\mu \|\eta_1) - \frac{\kappa}{2} H(\pi||\nu \otimes \nu) \right]. 
\end{align} 

Note that the minimization problem within  brackets in \eqref{opt:red2} is convex  with linear constraints.  The  next lemma identifies its explicit   solution. We postpone its proof to Appendix  \ref{section:kappa reg}.

\begin{lemma}
    \label{lemma:1MRF}
For any $\pi \in \Psym{\X}$, 
the measure $\mup$  defined in \eqref{eqn:1MRF} is the unique minimizer in 
the set 
$\{ \mu \in \P(\TkaXinf): \pi_\mu = \pi\}$ 
of the functional $H(\cdot\| \eta_1)$ and hence, of 
 $I^{\k}_1 (\cdot)$.  Furthermore, we have   
 \begin{equation}
 \label{eqn:J=I-mu-pi}
 \min_{\mu \in \P(\TkaXinf): \pi_\mu = \pi}  I^\k_1 (\mu) =  I^\k_1 (\mup) =  J^\nu_{\k} (\pi), \quad \forall \pi \in \Psym{\X}, \end{equation}
  where $J^\nu_{\k}$ is as defined in \eqref{eqn:J-nu-k}. 
\end{lemma}

\subsubsection{A further reduction} 
\label{subs:reduction}

Combining  \eqref{opt:red1} with Lemma \ref{lemma:1MRF},  it immediately follows that 
\begin{equation}
    \RRR (\cc) =  \min_{\cc' \geq \cc} \min_{\pi\in \BB_{h} (\cc')} J_\k^\nu(\pi).\label{eqn:J_opti_c>=c*}
\end{equation} 
We now provide a further reduction to the optimization problem in \eqref{eqn:J_opti_c>=c*}, by showing that the minimal value on the right-hand side is necessarily attained at $c' = c.$ This is  intuitive as it simply says that  the cheapest way to achieve the desired atypicality is by using the least deviation possible.  Nonetheless, the proof is rather technical and indirect, so we  postpone its proof to Appendix \ref{sec:c=c'}.

\begin{lemma}\label{lemma:optimizing_c=c'}
    For $\cc > \sref$, $\qquad$
    \begin{align}
        \min_{\cc' \geq \cc} \min_{\pi \in \BB_{h}(\cc')} J_\k^\nu(\pi)= \min_{\pi \in \BB_{h}(\cc)} J_\k^\nu(\pi). \label{eqn:=>1}
    \end{align}
    Likewise, for $\cc < \sref$, 
    \begin{align}
    \min_{\cc' \leq \cc} \min_{\pi \in \BB_{h}(\cc')} J_\k^\nu(\pi)= \min_{\pi \in \BB_{h}(\cc)} J_\k^\nu(\pi); \label{eqn:=>2}
    \end{align}
\end{lemma}

\subsubsection{Proof of Proposition \ref{prop:optsimp}} \label{subs:pf-of-prop-optsimp}
 
The result follows from the Gibbs conditioning principle (Proposition \ref{prop:gibbs cond}), together with the identity \eqref{opt:red2}, Lemma \ref{lemma:1MRF} and Lemma \ref{lemma:optimizing_c=c'}. 
\qed

\section{Characterization of minimizers}\label{sec:characterization_of_the_optimizers}

In view of Proposition \ref{prop:optsimp}, in order to identify conditional limits, it suffices to analyze the set of minimizing  measures ${\mathcal M}_{edge} (\cc)$ of the non-convex  optimization problem $\RRR_{edge}(\cc)$ specified in \eqref{opt:pi}.  In this section, we  characterize the larger set of  {\em local minimizers,}  including those that lie in the interior of the simplex ${\mathcal P}(\X^2)$ or equivalently, 
have  full support on $\X^2,$ as well as those that lie on the boundary of the simplex.  In Section \ref{subs:opt1}, we connect interior local minimizers  with non-degenerate $\TIS$ Gibbs measures using {\em the cavity equation} associated with a factor model on infinite regular trees. In Section \ref{subs:existence-of-boundary-minimizers}, we provide  two examples to show  that local minimizers  could  lie on the boundary and could even be the unique global minimizer, in which case  they   correspond to $\TIS$ Gibbs measures that is degenerate (namely, have a vertex marginal that is supported on a strict subset $\X^\downarrow\subsetneq\X$).  This  explains the need to include   degenerate $\TIS$ Gibbs measures in the set of potential conditional limits $\Gibbsdeg$ in Theorem \ref{thm:Gibbs-deg-and-nondeg}. In Section \ref{subs:boundary-local-minimizer-to-degenerate-TISGM}, we show that every boundary local minimizer  corresponds to some degenerate $\TIS$ Gibbs measure. We combine these results in Section \ref{subs:pf-of-Gibbs-deg-nondeg} to prove Theorem \ref{thm:Gibbs-deg-and-nondeg}.

We begin with clear definitions of some terminologies. Recall that the objective function $J_\k^\nu$ and constraint set $\BB_h(\cc)$ of the optimization problem $\RRR_{edge}(\cc)$  are as defined in  
\eqref{eqn:J-nu-k} and \eqref{eqn:Medge}, respectively.

\begin{definition}[Local minimizers of the edge measure optimization problem]\label{def:local-minimizer}
    An element $\pi_*\in \BB_h(\cc)$ is said to be a {\em local minimizer} of the edge optimization problem $\RRR_{edge}(\cc)$ if there exists a relatively open set $\mathcal{U}$ in $ \BB_h(\cc)$ containing $\pi_*$ such that $J^\nu_\k(\pi_*)\leq J^\nu_\k(\pi)$ for any $\pi\in \mathcal{U}.$
\end{definition}

\begin{definition}
    Given any countable discrete set $\mathcal{Z}$ and $\mu\in \P(\mathcal{Z})$, we say that $\mu$ lies {\em in the interior} (of $\P(\mathcal{Z})$) if $\mu(z)>0$ for all $z\in \mathcal{Z}$. Otherwise, we say that $\mu$ lies {\em on the boundary} (of $\P(\mathcal{Z})$). The collection of all probability measures in the interior and the collection of all probability measures on the boundary are denoted, respectively,  by
\begin{align}
    \Ppos(\mathcal{Z})&\coloneqq \pr{\mu\in \P(\mathcal{Z}):\mu(z)>0 \, \forall z\in \mathcal{Z}}\label{eqn:Ppos};\\
    \Pzero(\mathcal{Z})&\coloneqq \pr{\mu\in \P(\mathcal{Z}):\exists z_*\in \mathcal{Z}\text{ s.t. }\mu(z_*)=0}\label{eqn:Pzero}.
\end{align}
We also define analogous subsets of symmetric measures as follows:
\begin{align}
        \Psympos{\X}&\coloneqq \pr{\pi\in \Psym{\X}: \pi(x,z)>0 \, \forall x,z\in \X}; \label{eqn:Psympos}\\
        \Psymzero{\X}&\coloneqq \pr{\pi\in \Psym{\X}: \exists \, x,z\in \X \text{ such that } \pi(x,z)=0}.\label{eqn:Psymzero} 
\end{align}
\end{definition}

We refer to (local) minimizers that lie in the interior/boundary of $\P(\X^2)$ as {\em interior/boundary minimizers}.

\subsection{From interior local minimizers to non-degenerate $\TIS$ Gibbs measures.}
\label{subs:opt1}

\begin{proposition}[Interior local minimizers are marginals of $\TIS$ Gibbs measures] 
\label{prop:int_local_min_to_Gibbs}
    Fix $\nu\in \P(\X)$. Let $h$ be any edge potential and fix $c\in (\k h_{\min},\k h_{\max})$. Then for every  interior local minimizer $\pi_*$ of the edge optimization problem \eqref{opt:pi}, there exist $\beta\in \R$ and $\pmb{\rho}\in\TISGM_{T^\k}^{\exp{\pr{\beta h}},\nu}$ such that $\pi_{\pmb\rho_1}=\pi_*$. 
\end{proposition}  
\begin{proof}
        First, observe from  
        \eqref{eqn:J-nu-k} and the assumption $\nu \in {\mathcal P}^+ (\X)$,  that the  objective function $J^\nu_\k:\Psym{\X}\rightarrow \R$ of the edge optimization problem \eqref{opt:pi} can be rewritten as
         \begin{equation*}
            J^\nu_\k(\pi)=(\k-1)H(\pi^o)-\frac{\k}{2}H(\pi)-\sum_{x\in \X}\pi^o(x)\log \nu(x),\quad \forall \pi\in \Psym{\X}. 
        \end{equation*}
       Let $\pi_*$ be an interior local minimizer of the problem \eqref{opt:pi}.  
        We now verify the conditions required to invoke the Karush-Khun-Tucker (KKT) Theorem (see \cite[Theorem 21.1]{ChoZak13}) 
        for the constrained optimization problem \eqref{opt:pi}. First, observe that $J^\nu_\k$ 
        is smooth in a neighborhood of $\pi_*$ since $\pi_*\in \Psympos{\X}$, 
        and  $H,$ $\pi^0$ and the map $\pi \to H(\pi^o)$   are all smooth on $\Psympos{\X}$. Next, we check that $\pi_*$ is a {\em regular point}, that is, the gradients of active (i.e., constraints satisfied with equality) constraint functions  are linearly independent. Since $\pi_*\in \Psympos{\X}$, the active constraints are 
        \begin{equation*}
            \sum_{x,z\in \X}\pi(x,z)=1,\quad \sum_{x,z\in \X}\pi(x,z)h(x,z)=\cc\quad  \text{and}\quad\pi(x,z)-\pi(z,x)=0,\quad \forall x< z\in \X,
        \end{equation*}
        whose gradients can be represented by $|\X|\times |\X|$ matrices $[1]_{x,z\in \X}$, $[h(x,z)]_{x,z\in \X}$ and $[\xi_{x'z'}(x,z)]_{x,z\in \X}$ for every $x'< z'\in \X$, respectively, where
        \begin{align*}
            \xi_{x'z'}(x,z)\coloneqq \begin{dcases}
                1 & \text{ if }x=x',z=z';\\
                -1 & \text{ if }x=z',z=x';\\
                0 & \text{ otherwise.}
            \end{dcases}
        \end{align*}To show that they are linearly independent, let $\alpha,\lambda\in \R$ and $\{\gamma_{x'z'}\}_{x'<z'\in \X}\subset \R$ be such that
        \begin{equation*}
            \alpha [1]+\lambda[h]+\sum_{x'< z'\in \X}\gamma_{x'z'}[\xi_{x'z'}]=[0].
        \end{equation*}
        Then, together with the fact that $[h]^T=[h]$, since $h$ is symmetric, and the fact that $[\xi_{x'z'}]+[\xi_{x'z'}]^T=[0]$ by the definition of $\xi_{x'z'}$, it follows that $2\alpha[1]+2\lambda[h]=[0]$. However, since $h$ is not a constant function (otherwise $(\k h_{\min},\k h_{\max})=\emptyset$), we see that $\alpha=\lambda=0$ and thus $\gamma_{x'z'}=0$ for all $x'<z'\in \X$.
        Therefore, the KKT theorem implies the existence of $\beta\in \R$ such that $\pi_*$ is a critical point of the map 
     \begin{equation}
     \Psympos{\X} \ni \pi  \mapsto J_\k^\nu(\pi)-\beta\k\inprod{h,\pi}.  
     \label{eqn:first-order-crit}
    \end{equation}
 An application of  \cite[Proposition 1.7]{DemMonSlySun14} with $d$, $\Delta$, $\pmb{h}$, $\bar{h}$, $h$, $\bar{\psi}$, 
 $\psi$ and the map $\Phi$ 
     therein corresponding to 
     $\k$, $\Psym{\X},$ $\pi$, 
     $\pi^0$, $\ell$, $\nu$, the function $\exp\pr{\beta h}$, and the map  $\pi\mapsto-J_\k^\nu(\pi)+\beta \k\inprod{h,\pi}$, respectively, shows that any critical point $\pi_*$ of \eqref{eqn:first-order-crit} can be represented as
    \begin{equation*}
        \pi_*(x,z)\propto \ell(x)\exp{\pr{\beta h(x,z)}}\ell(z),\quad \forall x,z\in \X,
    \end{equation*}
    where $\ell\in \P(\X)$ is a fixed point of the recursion $\BP$ in \eqref{eqn:cavity} with $(\psi,\bar\psi)=(\exp{\pr{\beta h}},\nu)$. Hence, by Lemma \ref{lemma:TISGM-to-cavity} 
    $\pi_*$ is the edge marginal of the $\TIS$ Gibbs measure on $T^\k$ with specification $(\exp\pr{\beta h},\nu)$ and (homogeneous) boundary law $\ell$.
\end{proof}

\begin{remark}
    \label{rem-Lagrange} 
    The value(s) of $\beta$ in the last proposition correspond to  Lagrange multipliers in the  optimization problem, as in \eqref{eqn:first-order-crit}.    They can be  determined for specific models using  further  knowledge of the structure of the edge potential $h$. For example, this is carried out for two-spin systems and the consensus edge potential in Section \ref{sec:Ising-phase-transition-proof}. 
\end{remark}

\begin{remark} 
Combining Proposition \ref{prop:int_local_min_to_Gibbs}, \eqref{eqn:opti-pi-to-opti-rho} in Proposition \ref{prop:optsimp} and Lemma \ref{lemma:unimodular-extension-of-splitting-Gibbs-measure}, we see that every interior local minimizer $\pi_*$ of the edge optimization problem $\RRR_{edge}(\cc)$ corresponds to a non-degenerate $\TIS$ Gibbs measure $\pmb\rho$ and furthermore that if $\pi_*$ were  the unique (global) minimizer of $\RRR_{edge}(\cc)$, 
that is, if ${\mathcal M}_{edge} (\cc) = \{\pi_*\}$, then the corresponding $\pmb\rho$ would be the conditional limit of $\{U_n\}$. 
\end{remark}

However, in general $\RRR_{edge}(\cc)$ might not have a unique (global) minimizer and,  even when it does, this minimizer need not  lie in the interior, as illustrated by Example 
\ref{example:boundary-global-minimizer} in 
 Section \ref{subs:existence-of-boundary-minimizers}.

\subsection{Existence of boundary minimizers and the connection with degenerate $\TIS$ Gibbs measures} \label{subs:opt2}

We now provide two examples to illustrate the existence of boundary (local and global) minimizers.

\subsubsection{Existence of boundary minimizers and degenerate conditional limits}
\label{subs:existence-of-boundary-minimizers} 

\begin{example}[Boundary local minimizers exist even for uniform mark distributions]\label{example:boundary-local-minimizers}
Let $\X=[q]$ for $q\geq 2$, $\nu=\mathsf{Unif}(\X)$ and $\k\geq 3$. Consider the edge potential $h$ defined by
\begin{align*}
    h(x,z)\coloneqq \begin{dcases}
        1 & \text{if }(x,z)=(q,q);\\
        -1 & \text{if }x=q \text{ or }z=q \text{ and }(x,z)\neq (q,q);\\
        0.5 &\text{if }x,z\in [q-1],
    \end{dcases}
\end{align*}
with constraint value $c=0.5\k\in \cpr{\k h_{\min},\k h_{\max}} = (-\k, \k)$. Then, as proved in Section \ref{subs:pf-boundary-local-minimizers}, 
\begin{align*}
    \pi_*(x,z)\coloneqq \begin{dcases}
        0 & \text{if }x=q \text{ or }z=q;\\
        (q-1)^{-2} &\text{if }x,z\in [q-1]; 
    \end{dcases}
\end{align*}
is a boundary local minimizer of the edge optimization problem in \eqref{opt:pi}.  
\end{example}

\begin{example}[Boundary global minimizers exist even for two-spin systems]\label{example:boundary-global-minimizer} Let $\X=\pr{1,-1}$, $\nu(1)=1/3$, $\nu(-1)=2/3$ and $\k=5$. Consider the edge potential defined by
\begin{align}
    h(x,z)\coloneqq \begin{dcases}
        7 & \text{ if }(x,z)=(1,1);\\
        -5 & \text{ if }(x,z)=(1,-1) \text{ or }(x,z)=(-1,1);\\
        4 & \text{ if }(x,z)=(-1,-1); 
    \end{dcases}\label{eqn:boundary-global-h}
\end{align}
with constraint value $\cc=\k\cdot h(-1,-1)=20\in (\k h_{\min},\k h_{\max})$. Then, as shown in Section \ref{subs:pf-boundary-global-minimizers},
\begin{align*} \pi_*\coloneqq\delta_{(-1,-1)}
\end{align*}
is the unique (global) minimizer of the edge optimization problem $\RRR_{edge}(\cc)$. 

Moreover,  using \eqref{eqn:opti-pi-to-opti-rho} of Proposition \ref{prop:int_local_min_to_Gibbs}, it follows that  $\optset(\cc)=\pr{\UGW_1(\mu^{(\delta_{(-1,-1)})})=\delta_{T^{\k,-}}}$. Together with Proposition \ref{prop:gibbs cond}, we conclude that $\{U_n\}$ given $\Conset_n(c)$ converges to $\delta_{T^{\k,-}}$, which is a degenerate $\TIS$ Gibbs measure with $\X^\downarrow=\pr{-1}$. 
\end{example}
\begin{remark}[A new type of degeneracy] \label{remark:degenerate-not-freezing}
The degeneracy that appears in Example \ref{example:boundary-global-minimizer} is different from the one that arises in 
 Theorem \ref{thm:Ising-phase-transition}(2) when $p\neq 1/2$, $h$ is the consensus function $m$ and the constraint value $\cc=\k=\k m_{\max}$ lies on the boundary of $(\k m_{\min},\k m_{\max})$).  In latter case, the degenerate conditional limit can be described by an Ising measure with freezing temperature $\beta=\infty$. In contrast, the degenerate conditional limit in Example 
\ref{example:boundary-global-minimizer} 
  does not correspond to any $\TIS$ Gibbs measures with state space $\X$ and specification $(\exp\{\beta h\},\nu)$ even if we allow freezing temperatures ($\beta=\infty$ or $-\infty$). This is because Gibbs measures with freezing temperatures would have edge marginals supported on $\argmin_{x,z\in \X}h(x,z)$ if $\beta=-\infty$, and on $\argmax_{x,z\in \X}h(x,z)$ if $\beta=\infty$, but the constraint value $\cc=\k\inprod{h,\pi_*}\in (\k h_{\min},\k h_{\max})$.
\end{remark}

\subsubsection{From boundary local minimizers to degenerate $\TIS$ Gibbs measures}\label{subs:boundary-local-minimizer-to-degenerate-TISGM}
Example \ref{example:boundary-global-minimizer} illustrates the emergence of degenerate $\TIS$ Gibbs measures as potential conditional limits.  
In this section, we provide a more systematic connection between boundary local minimizers of the edge optimization problem and the degenerate $\TIS$ Gibbs measures in $\Gibbsdeg$ of Theorem \ref{thm:Gibbs-deg-and-nondeg}.

We start by  computing the ``derivative'' of $J_\k^\nu$ in $\Psympos{\X}$.  Recall from the definition of $J^\nu_\k$ in \eqref{eqn:J-nu-k} that $J_\k^\nu$ is expressed as the sum of various relative entropies, which  blow up at the rate $\varepsilon\log \varepsilon$ near the boundary, so the correct scaling here should be $|\varepsilon\log \varepsilon|$.

\begin{lemma}\label{lemma:eloge-derivative}
Fix $\nu\in \P(\X)$ and $\k\geq 2$. Given $\pi\in \Psym{\X}$, for any $\xi\in \Psym{\X}$, define
\begin{equation*}
    \pi_\varepsilon=\pi_\varepsilon[\xi]\coloneqq (1-\varepsilon)\pi+\varepsilon\xi,\quad \forall \varepsilon\in [0,1].
\end{equation*}
Then
\begin{equation}
    \lim_{\varepsilon\rightarrow 0^+}\frac{1}{|\varepsilon\log \varepsilon|}\cpr{J_\k^\nu(\pi_\varepsilon)-J_\k^\nu(\pi)}=(\k-1)\sum_{x\in \X}\xi^o(x)\indf_{\pr{\pi^o(x)=0}}-\frac{\k}{2}\sum_{x,z\in \X}\xi(x,z)\indf_{\pr{\pi(x,z)=0}}.
\end{equation}
\end{lemma}
\begin{proof}
    This is easily  shown by direct computation 
    (e.g., see the  proof of \cite[Proposition 3.2]{DemMonSun13}). 
\end{proof}

The next lemma shows that we can exclude the cases  where there is not even a single interior probability measure in the constraint set $\BB_h(\cc)$, by adding a minimal assumption on the constraint value $\cc$.

\begin{lemma}\label{lemma:existence-of-degenerate-set}
    Fix any mark distribution $\nu \in \P(\X)$, edge potential $h$ and $\k\geq 2$. Then $\cc\in \cpr{\k h_{\min},\k h_{\max}}$ if and only if $ \Psympos{\X}\cap \BB_h(\cc)\neq \emptyset$. Moreover, if $\cc\in \cpr{\k h_{\min},\k h_{\max}}\setminus \mathsf{Range}(\k h)$, then for any $\X^\downarrow \subseteq \X$ such that $\Psym{\cpr{\X^\downarrow}}\cap \BB_h(\cc)\neq \emptyset$, we have $\Psympos{\cpr{\X^\downarrow}}\cap \BB_h(\cc)\neq \emptyset$.
\end{lemma}
\begin{proof}
    We begin with the ``only if'' direction of the first assertion of the lemma.  Partition $\X^2$ into the sets $\calH_>(\cc),\calH_<(\cc)$ and $\calH_=(\cc)$  defined by
    \begin{align*}
\calH_\gtreqqless(\cc)\coloneqq \pr{(x,z)\in \X^2:\k h(x,z)\gtreqqless \cc}.
    \end{align*}
    Then the fact that $\cc\in \cpr{\k h_{\min},\k h_{\max}}$ implies that $\calH_>(\cc)\neq \emptyset$ and $\calH_<(\cc)\neq \emptyset.$ Let $\xi^+$ and $\xi^-$ be the uniform measure on $\calH_>(\cc)$ and $\calH_<(\cc)$ with mean $\inprod{\k h,\xi^+}=\cc^+ > c$ and $\inprod{\k h,\xi^-}=\cc^- < c$, respectively. Note that both $\xi^+$ and $\xi^-$ lie in $\Psym{\X}$ since $h$ is symmetric. Moreover, since $\cc^+>\cc>\cc^-$, there exists  $\alpha\in (0,1)$ such that $\alpha \cc^++(1-\alpha) \cc^-=\cc$. Then $\xi\coloneqq \alpha \xi^++(1-\alpha)\xi^-$ 
    satisfies $\inprod{\k h,\xi}=\cc$. If $\calH_=(\cc)=\emptyset$, then $\xi$ belongs to $ \Psympos{\X}\cap \BB_h(\cc)$. Otherwise, the average of $\xi$ and the uniform measure on $\calH_=(\cc)$ would lie in $\Psympos{\X}\cap \BB_h(\cc)$. In either case,  $\Psympos{\X}\cap \BB_h(\cc) \neq \emptyset$, and the ``only if'' direction of the first assertion of the lemma follows.

    For the ``if'' direction of the first assertion, note that if $\cc\notin (\k h_{\min},\k h_{\max})$, then we even have $\P\cpr{\X^2}\cap \BB_h(\cc)=\emptyset$. 
    Indeed, if $\cc=\k h_{\max}$, then for any $\xi\in \Psympos{\X}$, we have $\inprod{\k h,\xi}<\k h_{\max}$, thus $ \Psympos{\X}\cap \BB_h(\cc)=\emptyset$, contradiction.

    We now turn to the proof of the second assertion. Assume  $\cc\in \cpr{\k h_{\min},\k h_{\max}}\setminus \mathsf{Range}(\k h)$, and  $\X^\downarrow\subseteq \X$ is such that 
    $\Psym{\cpr{\X^\downarrow}}
    \cap \BB_h(\cc)\neq \emptyset$. Then it follows that \begin{equation*}
        \cc\in [\k \cdot\min_{x,z\in \X^\downarrow}h(x,z),\k\cdot\max_{x,z\in \X^\downarrow}h(x,z)]\setminus \mathsf{Range}(\k h)\subseteq (\k\cdot \min_{x,z\in \X^\downarrow}h(x,z),\k \cdot\max_{x,z\in \X^\downarrow}h(x,z)).
    \end{equation*}
    Consequently, from the first assertion of this lemma, it follows that $\Psympos{\cpr{\X^\downarrow}}\cap \BB_h(\cc)\neq \emptyset$. This completes the proof of the lemma.  
\end{proof}

The next lemma shows that for any spin system 
the marginal of any boundary local minimizer must assign zero mass to at least one of the spin values.

\begin{lemma}[Boundary local minimizers must be degenerate] \label{lemma:pos_marginal_cannot_be_locmin}
    Let $\nu\in \P(\X)$ and $\k\geq 2$. Fix  an 
edge potential $h$ and $\cc\in \R$ such that $\cc\in (\k h_{\min},\k h_{\max})$.  If $\pi_*$ is a boundary local minimizer of the edge optimization problem $\RRR_{edge}(\cc)$ in \eqref{opt:pi}, then $\pi_*^o\in \Pzero(\X)$. Moreover, if we further assume that $\cc\in (\k h_{\min},\k h_{\max})\setminus \mathsf{Range}(\k h)$, then for any boundary local minimizer $\pi_*$ of the edge optimization problem $\RRR_{edge}(\cc)$ in \eqref{opt:pi}, there exists a unique nonempty set $\X^\downarrow \subsetneq \X$ such that
    \begin{align*}
        \pi_*\in \Psympos{\cpr{\X^\downarrow}}.
    \end{align*}
\end{lemma}
\begin{proof} 
Fix a boundary local minimizer $\pi_*$ of the edge optimization problem and $\cc\in (\k h_{\min},\k h_{\max})$. We argue by contradiction to prove the first assertion.  
    Suppose, to the contrary, that  $\pi_*^o(x)>0$ for all $x\in \X$. 
    By Lemma \ref{lemma:existence-of-degenerate-set} and the fact that $\cc\in (\k h_{\min},\k h_{\max})$, there exists  $\xi\in \BB_h(\cc)\cap \Psympos{\X}$. Then set $\pi_\varepsilon\coloneqq (1-\varepsilon)\pi_*+\varepsilon\xi$, and note that  $\pi_\varepsilon \in \BB_h(\cc)\cap \Psympos{\X}$ and 
    by Lemma \ref{lemma:eloge-derivative}, we have
    \begin{align*}
        \lim_{\varepsilon\rightarrow 0^+}\frac{1}{|\varepsilon\log \varepsilon|}\cpr{J_\k^\nu(\pi_\varepsilon)-J_\k^\nu(\pi_*)}&=(\k-1)\sum_{x\in [q]}\xi^o(x)\indf_{\pr{\pi^o_*(x)=0}}-\frac{\k}{2}\sum_{x,z\in [q]}\xi(x,z)\indf_{\pr{\pi_*(x,z)=0}}\nonumber\\
        &=-\frac{\k}{2}\sum_{x,z\in [q]}\xi(x,z)\indf_{\pr{\pi_*(x,z)=0}}\nonumber\\
        &<0,
    \end{align*}
    where in the second line we used our assumption that $\pi_*^o (x)>0$ for all $x \in \X$  and in the last line, we used $\xi\in \Psympos{\X}$ and $\pi_*\in  \Psymzero{\X}$. This contradicts the supposition that $\pi_*$ is a local minimizer of the edge optimization problem, and the first assertion is proved. 
    
    The first assertion implies that if 
    $\X^{\downarrow} := \supp \pi_*^o$, then 
     $\X^{\downarrow} \subsetneq \X$, and 
    clearly $\pi_*(x,z)=0$ if $x\notin \X^\downarrow$ or $z\notin \X^\downarrow$   and hence,  $\pi_* \in 
    \Psym{\cpr{\X^\downarrow}}
    \cap \BB_h(\cc)$. 
    Thus, if  $\cc\in (\k h_{\min},\k h_{\max})\setminus \mathsf{Range}(\k h)$,  Lemma \ref{lemma:existence-of-degenerate-set} implies that $\pi_*\in \Psympos{\cpr{\X^\downarrow}}$.  This proves the second assertion. 
\end{proof}

\subsection{Proof of Theorem \ref{thm:Gibbs-deg-and-nondeg}}
\label{subs:pf-of-Gibbs-deg-nondeg}

Fix an arbitrary edge potential $h$, mark distribution $\nu \in \P(\X)$ and 
$ \cc \in (\k h_{\min}, \k h_{\max}) \setminus  \textsf{Range} (\k h)$  such that $\cc \neq \sref$. 
 By Definition \ref{def-condlim} and the Gibbs conditioning principle in Proposition \ref{prop:gibbs cond}, it suffices to show that  $\Gibbsdeg(\cc)$ contains the set 
 $\optset(\cc)$ of minimizers of the  minimization problem for $\RRR(c)$ in \eqref{eqn:Iinf-optimization}.
 Applying Proposition \ref{prop:optsimp}  
 we see that 
  \[ \optset (\cc) = \{\UGW_1 (\mup ): \pi \in \calM_{edge}(c)\}. \]
 Also, Lemma \ref{lemma:pos_marginal_cannot_be_locmin} shows that for every $\pi_* \in \calM_{edge}(c)$, there exists $\X^\downarrow \subseteq \X$ such that 
  $\pi_*\in \Psympos{\X^\downarrow}\cap \BB_h(\cc)$. Together with the identity $J^{\nu^\downarrow}_\k(\pi)=J^{\nu}_\k(\pi)+\log \nu(\X^\downarrow)$ for any $\pi\in \Psym{\X^\downarrow}$ from 
  \eqref{eqn:J-nu-k}, and 
  the   fact that $\pi_*$ is a (global) minimizer of the map $\Psym{\X}\cap \BB_h(\cc)\ni \pi\mapsto J^\nu_\k(\pi)$, it follows that $\pi_*$ is necessarily also a (global) minimizer of the restricted map $\Psym{\X^\downarrow}\cap \BB_{h^\downarrow}(\cc)\ni \pi\mapsto J^{\nu^\downarrow}_\k(\pi)$ with edge potential $h^\downarrow \coloneqq h\mid_{\X^\downarrow\times \X^\downarrow}$ and reference measure $\nu^\downarrow=\nu(\cdot\mid \X^\downarrow)$ 
  Together with Proposition 
\ref{prop:int_local_min_to_Gibbs} this implies that 
\begin{align*}
    \calM_{edge}(c) \subseteq 
\bigcup_{\X^\downarrow \subseteq \X} \bigcup_{\beta \in \R} \pr{ \pi_{\rho_1}:  \rho \in \TISGM^{\exp\pr{\beta h^\downarrow},\nu^\downarrow}_{T^\k}} \bigcap \BB_h(c).
\end{align*}  
On the other hand, Lemma 
\ref{lemma:unimodular-extension-of-splitting-Gibbs-measure} 
 shows that every $\rho \in \TISGM^{\exp\pr{\beta h^\downarrow},\nu^\downarrow}_{T^\k}$ 
 satisfies $\rho = \UGW_1 (\mu^{(\pi_{\rho_1})})$. 
Together with   the last two displays and the definition of $\Gibbsdeg(\cc)$ 
 in \eqref{eqn:TIGSMs-deg-and-nondeg}, 
 this proves the theorem. 

\section{Sufficient conditions for non-degeneracy of conditional limits}\label{sec:sufficient-condition-for-nondegeneracy}

 In Section \ref{subs:condition1} 
 we introduce a sufficient condition 
for spin systems to have non-degenerate  conditional limits. 
This condition is verified  for 
two-spin systems with uniform mark distribution 
in Section \ref{subs:suffcond}.  Then, as shown in Section \ref{sec:verification}, the proof of Theorem \ref{thm:Gibbs-component-two-spin-uniform} 
follows as an 
immediate corollary. 

\subsection{A sufficient condition for non-degeneracy}
\label{subs:condition1}

\begin{proposition}
[Conditional limits when (global) minimizers lie  in the interior]\label{prop:constrained-min-interior}
Suppose the mark space $\X$, mark distribution $\nu \in \P (\X)$, edge potential $h$, and constraint value $c \neq \sref$ satisfy  the following condition: 
\begin{equation}
\label{eqn:cond}
  c \in (\k h_{\min}, \k h_{\max}) \quad \mbox{ and } \quad      \optset_{edge}(\cc)\subseteq \Psympos{\X}, 
    \end{equation} 
    where 
$\Psympos{\X}$ is the collection of interior symmetric measures on 
$\X^2$ introduced in \eqref{eqn:Psympos}, and 
${\mathcal M}_{edge} (c)$ is the set of minimizers of the optimization problem 
$\RRR_{edge}(c)$ in \eqref{opt:pi}.  Then 
 the sequence $\{U_n\}$ given $\Conset_n(h,\cc)$ lies asymptotically in the set $\Gibbsnondeg(\cc)$  defined in 
\eqref{eqn:Gibbs-limit-splitting}. 
    \end{proposition}

\begin{proof} 
Let $\optset(\cc)$ be the set of minimizers of the  minimization problem  for $\RRR (\cc)$ in  \eqref{eqn:Iinf-optimization}. By Definition \ref{def-condlim} and the Gibbs conditioning principle in Proposition \ref{prop:gibbs cond},  it suffices to show that $\optset(\cc)$ is contained in $\Gibbsnondeg(\cc)$.  To this end, first apply \eqref{eqn:opti-pi-to-opti-rho} of Proposition \ref{prop:optsimp}, next the assumption that $\optset_{edge} (\cc) \subseteq \Psympos{\X}$, the fact that $\optset_{edge} (\cc) \subseteq \BB_h(\cc)$ and  
  Proposition  \ref{prop:int_local_min_to_Gibbs}, and finally invoke Lemma \ref{lemma:unimodular-extension-of-splitting-Gibbs-measure}, 
  to obtain 
    \begin{align*}\calM (c) &= \{\UGW_1 (\mup ): \pi \in \calM_{edge}(c)\} \\
    & \subseteq 
     \left\{\UGW_1 (\mu^{(\pi_{\rho_1})} ): \rho \in \bigcup_{\beta \in \R} \TISGM^{\exp\pr{\beta h},\nu}_{T^\k}\right\} \bigcap 
     \{ \rho \in \P(\TkaXinf): \pi_{\rho_1} \in \BB_h(\cc)\}\\
&     =  \bigcup_{\beta \in \R}\TISGM^{\exp\pr{\beta h},\nu} _{T^\k} \bigcap 
     \{ \rho \in \P(\TkaXinf): \pi_{\rho_1} \in \BB_h(\cc)\}.
    \end{align*} 
Since the right-hand side of the last display coincides with $\Gibbsnondeg(\cc)$, this proves the proposition. 
   \end{proof}

\begin{remark} 
\label{rem:conj} 
In view of the last proposition,  
to show Conjecture \ref{conj:Gibbs-component} it suffices to show that 
multi-spin systems with uniform mark distributions  satisfy 
\eqref{eqn:cond} of Proposition \ref{prop:constrained-min-interior}, or equivalently that any (global) in $\optset_{edge}$ lies in the interior, for all $h$ and valid $c$.   
\end{remark}

In Proposition \ref{prop:two-spin-local-not-global} below, we verify  
\eqref{eqn:cond} of Proposition \ref{prop:constrained-min-interior} 
for two-spin systems with uniform mark distribution. 
It is natural to ask if this restriction to a uniform mark distribution is necessary. 
Theorem \ref{thm:Ising-phase-transition}(1) shows that for certain two-spin systems with $h=m$ being the consensus functional, 
this condition is satisfied  even when the mark distribution is not uniform. 
However, as is evident from Example \ref{example:boundary-global-minimizer}, if one wants the condition to be satisfied by  all $h$ and admissible $c$, then indeed the 
restriction to a uniform mark distribution may be necessary.

\subsection{Verification of the conjecture for certain two-spin models}
\label{subs:suffcond}

\begin{proposition}[Case when  minimizers lie in the interior] \label{prop:two-spin-local-not-global}
    Let $|\X|=2$ and $\nu=\mathsf{Unif}(\X)$. For any edge potential $h$ and $\cc\in (\k h_{\min},\k h_{\max})$, the set of  
    minimizers 
    $\optset_{edge}(\cc)$ defined in \eqref{opt:pi} 
    satisfies 
    \begin{equation}\label{eqn:constrained-min-interior}
        \optset_{edge}(\cc)\subseteq \Psympos{\X}.
    \end{equation}
    In particular, Conjecture \ref{conj:Gibbs-component} is valid when $|\X| = 2.$
\end{proposition}

\begin{remark}
    Note that the property \eqref{eqn:constrained-min-interior} that global minimizers lie in the interior  need not hold when $\cc$ lies 
    on the {\it boundary} of the interval $(\k h_{\min},\k h_{\max})$. For example, take $\X=\pr{1,-1}$ and edge potential $h$ such that $h(1,1)>h(x,z)$ for all $(x,z)\neq (1,1)$. If $\cc=\k h_{\max}=\k h(1,1)$, then $\BB_h(\cc)=\pr{\delta_{(1,1)}}$. Hence, $\delta_{(1,1)}$ is the unique minimizer, but $\delta_{(1,1)}\notin \Psympos{\X}$.
\end{remark}

As shown in Example \ref{example:boundary-local-minimizers}, 
even for the sub-class 
of models considered in Proposition \ref{prop:two-spin-local-not-global}, local minimizers need not lie in the interior. 
 In Section \ref{subsub:two-spin-local}, we obtain a characterization of   local minimizers that is then used in Section \ref{sec:two-spin-local-global-uniform} to show 
 that 
boundary local minimizers can never be  global minimizers, thus proving 
 Proposition \ref{prop:two-spin-local-not-global}.  
The arguments 
 make use of a reparameterization of the optimization problem for two-spin systems  introduced in Section \ref{subs:reparam}.
\subsubsection{Reparameterization of the edge optimization problem for two-spin models}
\label{subs:reparam}

When  $|\X| = 2$,  
the elements of $\Psym{\X}$  are symmetric probability measures on a four-dimensional space and hence, have only two degrees of freedom. 
For concreteness, 
assume  henceforth that $\X =  \{1,-1\}$, and   reparameterize 
$\Psym{\X}$ 
as  follows:

\begin{equation}
\label{eqn:pist}
  \pi=\pi[s,t]\in \Psym{\X}, \qquad (s,t) \in \Delta := \{[s',t']
 \in [0,1]^2: s' \geq t', s'+t' \leq 1  \},
 \end{equation}
 with   
\begin{equation}\label{eqn:two-spin-parametrization}
    \pi(1,1)=s-t,\quad \pi(1,-1)=\pi(-1,1)=t,\quad \pi(-1,-1)=1-s-t.
\end{equation}
Note that under this parametrization, we have $\pi^o(1)=s$, where $\pi^o$ is the marginal law of $\pi$, defined in Section \ref{sec:notation}. 
Any functional on  or subset of $\Psym{\X}$ can then be expressed as a functional on  or subset of $\Delta$.  In particular,  for the  objective function from  \eqref{eqn:J-nu-k}, we set (using the same  notation for the reparameterized version) 
\begin{align}\label{eqn:J-in-st}
    J_\k^\nu(s,t)\coloneqq J_\k^\nu(\pi[s,t]),\quad \forall (s,t)\in \Delta, 
\end{align}
which takes the explicit form   
\begin{align}
    J^\nu_\k(s,t)= \cpr{s\log\cpr{\frac{\nu(-1)}{\nu(1)}}-\log\cpr{2\nu(-1)}}+J^{(2)}_\k(s,t),\quad \forall (s,t)\in \Delta\label{eqn:Jnu-two-spin}, 
\end{align}
where $J^{(2)}_\k(s,t)\coloneqq J^{\mathsf{Unif(\X)}}_\k(\pi[s,t])$ is equal to 
\begin{align}
    J^{(2)}_\k(s,t)& = \log2 - (\k-1)\cpr{ s \log s + (1-s) \log (1-s)} \nonumber \\
& \qquad + \frac{\k}{2} \cpr{(s-t)\log (s-t) + 2t \log t + (1-s-t) \log(1-s-t)},\quad \forall (s,t)\in \Delta. \label{eqn:Jst}
\end{align}
For future purposes, define the interior of $\Delta$ by
\begin{equation}
    int(\Delta)\coloneqq \pr{(s',t')\in (0,1)^2:s'>t',s'+t'<1}\label{eqn:Delta-interior}.
\end{equation}
Then we have
\begin{align}
    \partial_sJ_\k^{(2)}(s,t)&=(1-\k)\log\cpr{\frac{s}{1-s}}+\frac{\k}{2}\log\cpr{\frac{s-t}{1-s-t}},\quad \forall (s,t)\in int(\Delta)\label{eqn:Jst-ds}\\
    \partial_tJ_\k^{(2)}(s,t)&=-\frac{\k}{2}\log\cpr{\frac{(s-t)(1-s-t)}{t^2}},\quad \forall (s,t)\in int(\Delta).\label{eqn:Jst-dt}
\end{align}

Lastly,  note that the constraint set 
 $\BB_h(\cc)$ can be 
  reparameterized as  
 \begin{equation}
 \label{eqn:Bhst}
 \{  (s,t) \in \Delta: (h(1,1)-h(-1,-1))s+h(-1,-1)-\frac{\cc}{\k}=(h(1,1)+h(-1,-1)-2h(1,-1))t\}.
\end{equation}
Therefore, if $h(1,1)+h(-1,-1)-2h(1,-1)\neq 0$, we have
\begin{equation}
    \BB_h(\cc)=\pr{(s,t)\in \Delta: t=w(h)s}, \quad w(h)\coloneqq \dfrac{h(1,1)-h(-1,-1)}{h(1,1)+h(-1,-1)-2h(1,-1)},\label{eqn:two-spin-slope}
\end{equation}
where $w(h)$ is called the \textit{slope of the constraint segment}.

\subsubsection{Characterization of boundary local minimizers for two-spin systems}
\label{subsub:two-spin-local}

\begin{lemma}[Characterization of boundary local minimizers] \label{lemma:two-spin-existence-boundary-locmin}
    Let $\X=\{1,-1\}$ and $\nu=\mathsf{Unif}(\X)$. Given an edge potential $h$ and $\cc\in \cpr{\k h_{\min},\k h_{\max}}$, 
    there exists a boundary local minimizer $\pi_*$ of the edge optimization problem $\RRR_{edge}(\cc)$ in \eqref{opt:pi} if and only if \[ h(1,1)+h(-1,-1)-2h(1,-1)\neq 0,  \]
and the slope of the constraint segment $w=w(h)$, defined in \eqref{eqn:two-spin-slope}, and the value $c$ are  such that either one of the  following properties holds:
    \begin{align}\label{eqn:two-spin-condition-for-boundary-locmin}
        \begin{dcases}
            -\frac{\k-2}{\k}<w<0 &\text{and}\quad  c=\k h(1,1);\\
            0<w< \frac{\k-2}{\k}& \text{and}\quad  c=\k h(-1,-1).
        \end{dcases}
    \end{align}
    Moreover, when \eqref{eqn:two-spin-condition-for-boundary-locmin} is satisfied, then $\pi_*$ is unique and given explicitly by
    \begin{align}
        \pi_*=\begin{dcases}
            \delta_{(1,1)}& \text{if }-\frac{\k-2}{\k}<w<0;\\
            \delta_{(-1,-1)}& \text{if }0<w< \frac{\k-2}{\k}.
        \end{dcases}
    \end{align}
\end{lemma}

\begin{proof}
Fix $\cc\in \cpr{\k h_{\min},\k h_{\max}}$.  
Then by  Lemma \ref{lemma:existence-of-degenerate-set}, 
$\BB_h(\cc)\cap \Psympos{\X}\neq \emptyset$  and by Lemma \ref{lemma:pos_marginal_cannot_be_locmin}, any boundary local minimizer $\pi_*$ of the edge optimization problem is such that $\pi_*^o\in \Pzero(\X)$. Since $\X=\pr{1,-1}$, the only candidate boundary local minimizers are $\delta_{(1,1)}$ and $\delta_{(-1,-1)}$.  
We focus on the case $\pi_* = \delta_{(-1,-1)}$, as the other case can be analyzed exactly analogously by symmetry. 
 Using the parametrization in \eqref{eqn:two-spin-parametrization}, we have $\pi_* = \pi [0,0]$. Moreover,   $(0,0)$ lies in the (reparameterized) constraint set $\BB_h(c)$ given in \eqref{eqn:Bhst} 
  if and only if $h(-1,-1) = \k c$, in which case the constraint set reduces to 
\begin{equation*}
   \{ (s,t) \in \Delta: (h(1,1)-h(-1,-1))s=(h(1,1)+h(-1,-1)-2h(1,-1))t \}. 
\end{equation*}
If $h(1,1) + h(-1,-1)  - 2h(-1,1) = 0$, then the constraint set becomes $\{s=0\} \cap \Delta = \{0,0\}$.   
 By \eqref{eqn:pist},  this corresponds to $\BB_h(c) = \{\delta_{\{-1,-1\}}\}$,  which  contradicts the 
property $\BB_h(\cc)\cap \Psympos{\X}\neq \emptyset$. 
Thus, we must have $h(1,1) + h(-1,-1)  - 2h(-1,1) \neq 0$ and in this case, 
 the  constraint set takes the form  
$\{(s,t) \in \Delta: t = ws\},$ where $w=w(h)$ due to \eqref{eqn:two-spin-slope}. The property   $\BB_h(\cc)\cap \Psympos{\X}\neq \emptyset$ then implies that  we must also have $0 <  w < 1$. 

Next apply Lemma \ref{lemma:eloge-derivative} with 
$\pi = \pi_* = \delta_{(-1,-1)}$ and the measure $\xi := \pi[1,\frac{w}{1+w}]$, 
that is, with 
 \[  \xi(1,1) = \frac{1-w}{1+w},  \qquad 
 \xi(-1,1) = \xi(1,-1) = \frac{w}{1+w}, \qquad \xi(-1,-1) = 0, \]
and $\pi_\varepsilon = (1-\varepsilon) \pi_* + \varepsilon \xi$, we have
\begin{align*}
    \lim_{\varepsilon\rightarrow 0^+}\frac{1}{|\varepsilon\log\varepsilon|}\cpr{J^{(2)}_\k(\pi_\varepsilon)-J^{(2)}_\k(\pi_*)}&=(\k-1)\sum_{x\in \X}\xi^o(x)\indf_{\pr{\pi^o_*(x)=0}}-\frac{\k}{2}\sum_{x,z\in \X}\xi(x,z)\indf_{\pr{\pi_*(x,z)=0}}\\
    &=(\k-1)\frac{1}{1+w}-\frac{\k}{2}
    \\
    &=\frac{1}{2(1+w)}\cpr{\k-2-\k w}. 
\end{align*} 
The latter quantity is strictly positive 
if and only if $-1<w<\frac{\k-2}{\k}$, and is zero if and only if $w=\frac{\k-2}{\k}$. Moreover, since $w \in (0,1),$ this implies that $\pi_* = \delta_{(-1,-1)}$ 
is a local minimizer if $0<w<\frac{\k-2}{\k}$ and $c=\k h(-1,-1)$; and is not a local minimizer if $w>\frac{\k-2}{\k}$. The case when  $w=\frac{\k-2}{\k}$ cannot be determined from the above computation, yet we can directly compute the derivative of $J^{(2)}_\k(s,\frac{\k-2}{\k}s)$ at $s=0$, that is, using  the mean value theorem, \eqref{eqn:Jst-ds} and \eqref{eqn:Jst-dt}, we have
\begin{align}
    \lim_{s\rightarrow 0^+}\frac{1}{s}&\cpr{J^{(2)}_\k\cpr{s,\frac{\k-2}{\k}s}-J^{(2)}_\k(0,0)}\nonumber\\
    &= \lim_{s\rightarrow 0^+}
    \left[ \frac{d}{ds} J_\k^{(2)} \left(s, \frac{\k-2}{\k}s \right)\right]  \\
    &= \lim_{s\rightarrow 0^+} 
    \left[ \partial_s J_\k^{(2)} \left( s, \frac{\k-2}{\k}s\right) + \frac{\k-2}{\k}\partial_t J_\k^{(2)} \left( s, \frac{\k-2}{\k}s\right) \right]
 \\
    &=\log\cpr{\frac{2}{\k}}+(\k-2)\log\cpr{\frac{\k-2}{\k}},\label{eqn:J2-crit}
\end{align}
which is strictly negative. 
This shows that when $w=\frac{\k-2}{\k}$ and $c=\k h(-1,-1)$, $\pi_*=\delta_{\cpr{-1,-1}}$ is not a local minimizer. 
Note that the derivative computed in \eqref{eqn:J2-crit} for $J^{(\nu)}_\k = J^{(2)}_\k = J^{\mathsf{Unif}(\X)}_\k$ might not be negative if $\nu\neq \mathsf{Unif}(\X)$.
In conclusion, we have shown that $\pi_* = \delta_{(-1,-1)}$ 
is a local minimizer if and only if $0<w<\frac{\k-2}{\k}$ and $c=\k h(-1,-1)$. The argument for the case $\pi_* = \delta_{(1,1)}$ is exactly analogous, and thus omitted. 
\end{proof}

\begin{remark}\label{remark:pf-non-pos-line}
    From the above proof, we see that when $\nu=\mathsf{Unif(\pr{1,-1})}$,  $\pi[0,0] = \delta_{\{-1,-1\}}$ is a local minimizer  if and only if  
    \begin{equation}
        \label{eqn:pf-non-pos-line}
    0 < w < \frac{\k -2}{\k} \quad \mbox{ and } \quad \BB_h(\cc) = \{(s,t) \in \Delta: t  = w s\}, \end{equation}
    and   an analogous argument shows that $\delta_{(1,1)}$ is a local minimizer if and only if 
    \begin{equation} 
    \label{eqn:pf-non-pos-line2}
    -\frac{\k-2}{\k}<w< 0 \quad  \mbox{ and } \quad \BB_h(\cc) = \{(s,t) \in \Delta:  t=w(s-1)
\}. \end{equation}
\end{remark}

\subsubsection{Proof of Proposition \ref{prop:two-spin-local-not-global}}\label{sec:two-spin-local-global-uniform}
The main idea behind the proof is to show that if there exists a boundary local minimizer of $J^{(2)}_\k$ in $\BB_h(\cc)$, then there must also exist an element in $\BB_h(\cc)\cap \Psympos{\X}$ that has a strictly smaller value, and hence the boundary minimizer cannot be a global minimizer.  
By Remark \ref{remark:pf-non-pos-line}, it follows that boundary local minimizers exist if and only if the slope of the constraint segment $w = w(h)$, defined in \eqref{eqn:two-spin-slope}, satisfies 
    $|w| < \frac{\kappa-2}{\kappa}, 
    w \neq 0.$
    Without loss of generality, assume that $w \in (0, \frac{\k-2}{\k})$. 
Then, using  the parameterizations for      $\pi=\pi[s,t]$, $J^{(2)}_\k$ and $\BB_h(c)$  in Section \ref{subs:reparam} and invoking   
    Lemma \ref{lemma:two-spin-existence-boundary-locmin} 
    and Remark \ref{remark:pf-non-pos-line},  it follows that  $(s,t) = (0,0)$ is a strict local minimizer of $J^{(2)}_\k(s,t)$, and 
    $ \BB_h(\cc)$ is given by \eqref{eqn:pf-non-pos-line},  which we call the constraint segment. Also, consider the curve 
    \[ {\mathcal C} := \{(s,t) \in \Delta: \partial_t J_\k^{(2)}(s,t) = 0 \} = 
    \{ (s,t) \in \Delta: t = s(1-s) \},  
    \]
    where the equality holds by \eqref{eqn:Jst-dt}. 

Then we show that the following two properties hold: 
\begin{enumerate}
    \item[(i)]
 the value of $J_\k^{(2)}(s,t)$ on the curve ${\mathcal C}$ attains strict maxima at $(0,0)$ and $(1,0)$; 
 \item [(ii)] the constraint segment  intersects 
 ${\mathcal C}$ at some point $(s^\dagger,t^\dagger)$, with $s^\dagger\notin \pr{0,1}$. 
 \end{enumerate}
 Together,  these two properties imply that 
 \begin{equation*}
    \min\pr{J_\k^{(2)}(s,t): t=ws, (s,t) \in \Delta} \leq J_\k^{(2)}(s^\dagger,t^\dagger)<J_\k^{(2)}(0,0), 
\end{equation*}
which shows that 
 $(0,0)$ is not a global minimizer.

It only remains to prove properties (i) and (ii).  We start with (i).  
Taking the derivative along the curve ${\mathcal C}$,  and using \eqref{eqn:Jst-ds} and the fact that $\partial_t J_\k^{(2)}(s,s(1-s))=0$ in the second line, we see that for all $s \in (0,1),$   
\begin{align}
    \frac{d}{ds}J_\k^{(2)}(s,s(1-s))
    =&\partial_s J_\k^{(2)}(s,s(1-s))+(1-2s)\partial_t J_\k^{(2)}(s,s(1-s))\nonumber\\
    =&(1-\k)\log\cpr{\frac{s}{1-s}}+\frac{\k}{2}\log\cpr{\frac{s-s(1-s)}{1-s-s(1-s)}}\nonumber\\
    =&\log\cpr{\frac{s}{1-s}}, \label{eqn:J-ds-s-1-s}
    \end{align}
   whose sign coincides with that of $s-\frac{1}{2}.$
This implies that the value of $J_\k^{(2)}(s,t)$ along the curve ${\mathcal C}$  attains a maximum at either $(0,0)$ or $(1,0)$. 

Next, we turn to property (ii). This is straightforward since the curve ${\mathcal C}$ is quadratic, contains the points $(0,0)$ and $(1,0)$ and has derivative equal to $1$ at $s=0$. Therefore, any constraint segment passing through $(0,0)$ with slope $0<w<(\k-2)/\k$ has to intersect with ${\mathcal C}$ again in the interior. 
\qed

\subsection{Verification  of 
Conjecture \ref{conj:Gibbs-component} for two-spin systems}
\label{sec:verification}

We now obtain Theorem \ref{thm:Gibbs-component-two-spin-uniform} as an easy consequence of the results proved thus far. 

\begin{proof}[Proof of  Theorem \ref{thm:Gibbs-component-two-spin-uniform}] 
    This follows from the Gibbs conditioning principle in Proposition \ref{prop:gibbs cond}, the simplification of the minimization problem from \eqref{eqn:Iinf-optimization} to \eqref{opt:pi} in Proposition \ref{prop:optsimp}, the fact  that (global) minimizers of \eqref{opt:pi} lie in the interior by Proposition \ref{prop:two-spin-local-not-global}, the characterization of interior local minimizers $\pi_*\in \optset_{edge}(\cc)$ as marginals of  non-degenerate $\TIS$ Gibbs measures in Proposition  \ref{prop:int_local_min_to_Gibbs}, the corresponding characterization of $\optset(\cc)$ in \eqref{eqn:opti-pi-to-opti-rho}, and the relation between $\TIS$ Gibbs measures and their edge marginals established in Lemma \ref{lemma:unimodular-extension-of-splitting-Gibbs-measure}.
\end{proof}

\section{Two-spin systems with consensus as the edge potential}\label{sec:Ising-phase-transition-proof}

In this section, we prove Theorem \ref{thm:Ising-phase-transition} by characterizing the set 
$\optset_{edge}(\cc)$ of global minimizers of the edge optimization problem 
$\RRR_{edge}(\cc)$  in 
\eqref{opt:pi} when 
$\X=\pr{1,-1}$ and  $h(x,y)=\spech(x,y)=xy$. 
We split the proof of Theorem \ref{thm:Ising-phase-transition} into two cases, the case when $\cc$ is an  extreme point of the convex hull of the range of $m$,  considered in Section \ref{subs:Isingextreme}, and the case when $\cc$ lies in the interior, considered in 
Section \ref{subs:Isinginterior}.  
Throughout we will use  the  parameterizations for two-spin systems of edge measures $\pi=\pi[s,t]$  and the objective function $J^{\nu}_\k(s,t): \Delta\rightarrow \R,$ as specified in  \eqref{eqn:pist}-\eqref{eqn:two-spin-parametrization} and \eqref{eqn:J-in-st}-\eqref{eqn:Jst},  respectively. Note that under this parametrization,
\begin{equation}
    \bbE_{\pi[s,t]}\sqpr{X_o}=2s-1,\quad \forall (s,t)\in \Delta.\label{eqn:vertex-marginal-of-pist}
\end{equation}
Also, note that the parameterized version \eqref{eqn:Bhst} of $\BB_h(\cc)$,  when $h = m$, reduces to 
    \begin{equation}
       \BB_m(\cc)=\pr{\cpr{s,t} \in \Delta: s\in \sqpr{t(c),1-t(c)},  
        t=t(c)},
        \quad 
        t(c)\coloneqq \frac{1}{4}\cpr{1-\frac{c}{\k}}.   \label{eqnising-constraint-set}
    \end{equation}

\subsection{The case when $\cc \in \{-\k, \k\}$}
\label{subs:Isingextreme}

Note that the general result of Theorem \ref{thm:Gibbs-component-two-spin-uniform} does not cover the case when 
$\cc$  is an extreme point of the range of the (generalized) consensus functional.  We now address this case in the 
two-spin setting. 

\begin{proof}[Proof of Theorem \ref{thm:Ising-phase-transition} when $c \in \{-\k, \k\}$] 
      First, fix $\cc = -\k$, and  $\pi \in  \BB_m(-\k)$.  Then by the definitions of $\BB_m$ and $m$ in \eqref{eqn:Medge} and \eqref{def-spech}, respectively, and the fact that $\pi \in \Psym{\X}$,  we have 
        $-\k=\k\bbE_{\pi}[X_oX_1]= \k(\pi(1,1) + \pi(-1,-1) - 2 \pi (1,-1)$ $=
        \k\cpr{1-4\pi(1,-1)},$ 
    which implies $\pi(1,-1)=\pi(-1,1)=1/2$,  which then implies that $\pi(1,1)=\pi(-1,-1)=0$. Therefore, $\BB_m(\cc)$ contains just the one element,         $\pi^{\texttt{alt}}\coloneqq\frac{1}{2}\delta_{(-1,1)}+\frac{1}{2}\delta_{(1,-1)}$. Hence, the same is true of ${\mathcal M}_{edge}(c)$, 
    and  Proposition \ref{prop:optsimp} implies that $\optset(\cc)=\{\UGW_1(\mu^{(\pi^{\texttt{alt}})})\}$. By \eqref{eqn:1MRF}, we have 
    \begin{equation*}
        \mu^{(\pi^{\texttt{alt}})}=\frac{1}{2}\delta_{T^{\k,\pm}_1}+\frac{1}{2}\delta_{T^{\k,\mp}_1},
    \end{equation*}
    where $T^{\k,\pm}$ and $T^{\k,\mp}$ are the infinite $\k$-regular tree with alternating vertex marks (defined prior to Definition \ref{def:freezing-ising}), 
    and the subscript $1$ denotes their depth-$1$ subtree (defined in Section \ref{subsubsec:topology}).
    To identify its unimodular extension, note that the conditional measure on stars $(\widehat{P}_{\mu^{(\pi^{\texttt{alt}}}}(x, x'))_{x,x'\in \X}$ from \eqref{eqn:PU} 
    are given by 
    \begin{align*}
        \widehat{P}_{\mu^{(\pi^{\texttt{alt}})}}(1, -1)=\delta_{T^{\k-1,\pm}_1} \quad \text{ and } \quad 
        \widehat{P}_{\mu^{(\pi^{\texttt{alt}})}}(-1, 1)=\delta_{T^{\k-1,\mp}_1},
    \end{align*}
    and so, recalling the definition of $\ising_\k^{\sharp}(-\infty)$ in \eqref{eqn:freezing-anti}, we have  
    \begin{equation*}
\UGW_1(\mu^{(\pi^{\texttt{alt}})})=\frac{1}{2}\delta_{T^{\k,\pm}}+\frac{1}{2}\delta_{T^{\k,\mp}}=\ising_\k^\sharp(-\infty). 
    \end{equation*}
    
    Proposition \ref{prop:gibbs cond} then  implies that  ${U_n}$ given $\Conset_n(\cc)$ converges to $\ising_\k^{\sharp}(-\infty)$.

    Next, suppose $\cc=\k$. Then  $t(\cc) = 0$ by \eqref{eqnising-constraint-set}, and using \eqref{eqn:Jnu-two-spin} and \eqref{eqn:Jst}, it follows that 
    \begin{align*}
        J^\nu_\k(s,0)=\pr{-s\log\nu(1)-(1-s)\log\nu(-1)+\cpr{\frac{\k}{2}-1}H([s,1-s])}.
    \end{align*}
    Using the concavity of the map $s\mapsto H([s,1-s])=-s\log s-(1-s)\log(1-s)$,  it is easy to verify that 
    $[0,1] \ni s \to J^\nu_\k(s,0)$
     has two local minima, achieved at $s=0$ and $s=1$.

          We now consider three sub-cases. If $\nu(1)>\nu(-1)$, then the global minimum is attained at $s=1$, in which case   $\optset_{edge}(\cc)=\{\delta_{(1,1)}\}$ and hence, again by \eqref{eqn:opti-pi-to-opti-rho}, 
     $\optset(\cc)=\{\UGW_1(\mu^{(\delta_{(1,1)})})\}$. Then, arguing as above, by \eqref{eqn:1MRF} we have  $\mu^{(\delta_{(1,1)})}=\delta_{T^{\k,+}_1},$
    where $T^{\k,+}_1$ is the depth $1$-subtree of the infinite $\k$-regular tree with all vertex marks being $1$, 
    and hence, by Definition \ref{defn:UniExt},  $\UGW_1(\mu^{(\delta_{(1,1)})})=\delta_{T^{\k,+}}=\ising_\k^+(\infty),$
    where $\ising_\k^{+}(\infty)$ is as specified in \eqref{eqn:freezing-ferro}. 
     Next, if $\nu(-1)>\nu(1)$, then one can analogously argue that $\optset(\cc)=\{\delta_{T^{\k,-}}=\ising_\k^{-}(\infty)\}$, where $\Phi_\k^-(\infty)$ is as specified in \eqref{eqn:freezing-ferro}. 
          Lastly, in the uniform case when  $\nu(1)=\nu(-1)$,  clearly $J^\nu_\k(1,0) = J^\nu_\k(0,0)$ and so $\optset_{edge}(\cc)=\{\delta_{(1,1)},\delta_{(-1,-1)}\}$. Thus by \eqref{eqn:opti-pi-to-opti-rho} and the facts (argued above) that $ \UGW_1(\delta_{(1,1)})=\ising_k^+(\infty)$ and $\UGW_1(\delta_{(-1,-1)})=\ising_k^-(\infty)$, it follows that $\optset(\cc)=\{\ising_\k^+(\infty),\ising_\k^-(\infty)\}$.  Together with Proposition \ref{prop:gibbs cond}, this completes the proof of Theorem \ref{thm:Ising-phase-transition} in the case when $|c|=\k.$ 
\end{proof}

\subsection{The case when $\cc \in (-\k, \k)$}
\label{subs:Isinginterior}

 Before proceeding with the proof in this setting, we 
 numerically plot the objective function in 
 Figure \ref{fig:two-spin-phase-transition-2x2} to gauge the behavior of global minimizers. First, when $\nu\neq \mathsf{Unif}(\X)$, the plot of the maps $s\mapsto J^\nu_\k(s,t(c))$ for three different values of $\cc$ in Figure \ref{fig:nonuniform-section}  suggest that there is always a unique (global) minimizer even though $s \mapsto J^\nu_\k(s,t(c))$ might not be convex in $s.$ Next, the corresponding plot when $\nu=\mathsf{Unif}(\X)$ illustrated in Figure \ref{fig:uniform-section} suggests that there should be a phase transition for  all $\cc$ sufficiently large, and there are exactly two global minimizers after the phase transition. 

\begin{figure}[htbp]
  \centering

  \begin{subfigure}{0.46\textwidth}
    \centering
    \includegraphics[width=\linewidth,height=5.5cm]{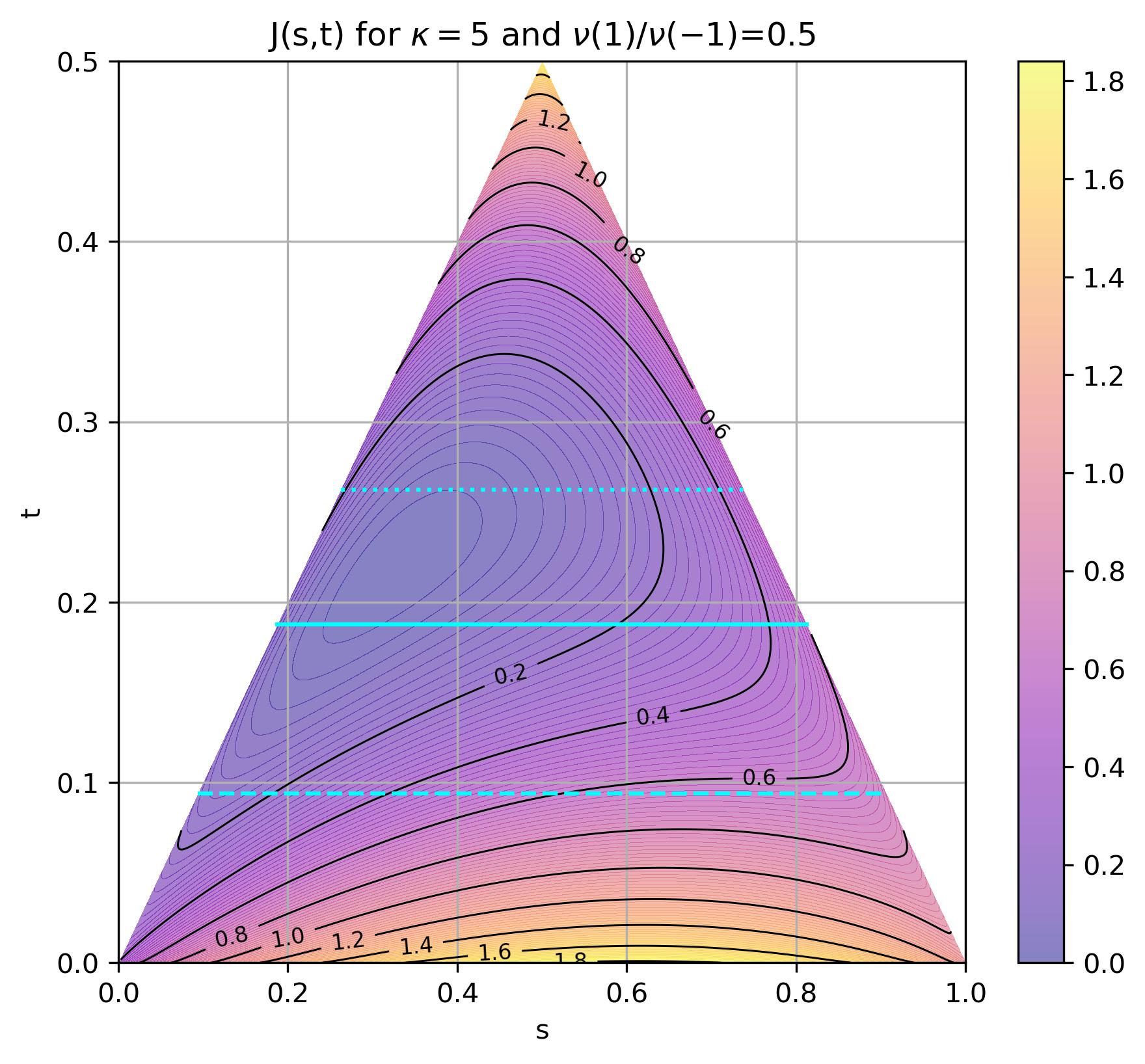}
    \caption{Contour plot for $J^\nu_\k(s,t)$, $p<1/2$.}
    \label{fig:nonuniform-heatmap}
  \end{subfigure}
  \hspace{0.5cm}
    \begin{subfigure}{0.41\textwidth}
    \centering
    \includegraphics[width=\linewidth,height=5.5cm]{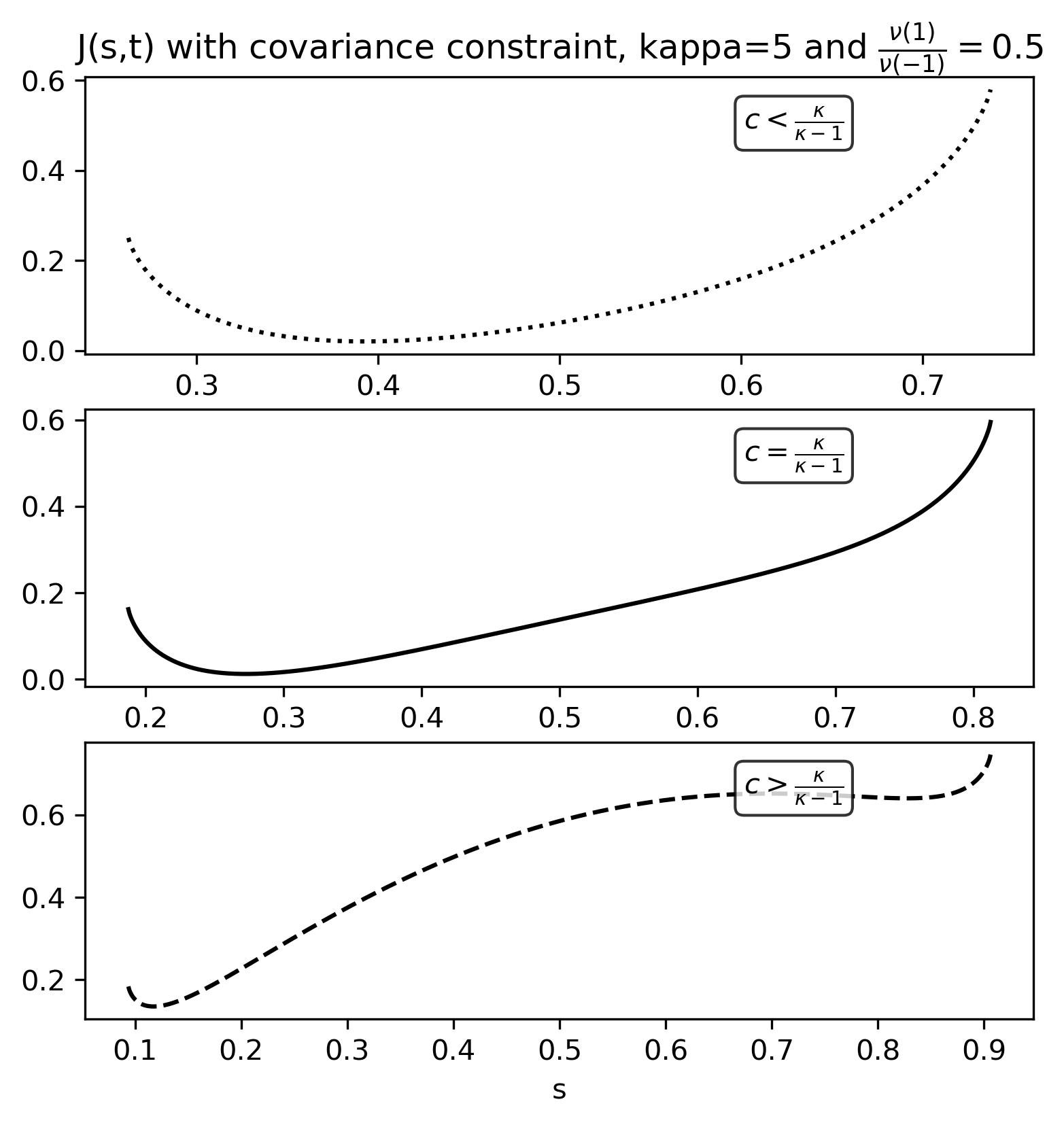}
    \caption{$s\mapsto J^{\nu}_\k(s,t(c))$, $p<1/2$}
    \label{fig:nonuniform-section}
  \end{subfigure}

  \begin{subfigure}{0.46\textwidth}
    \centering
    \includegraphics[width=\linewidth,height=5.5cm]{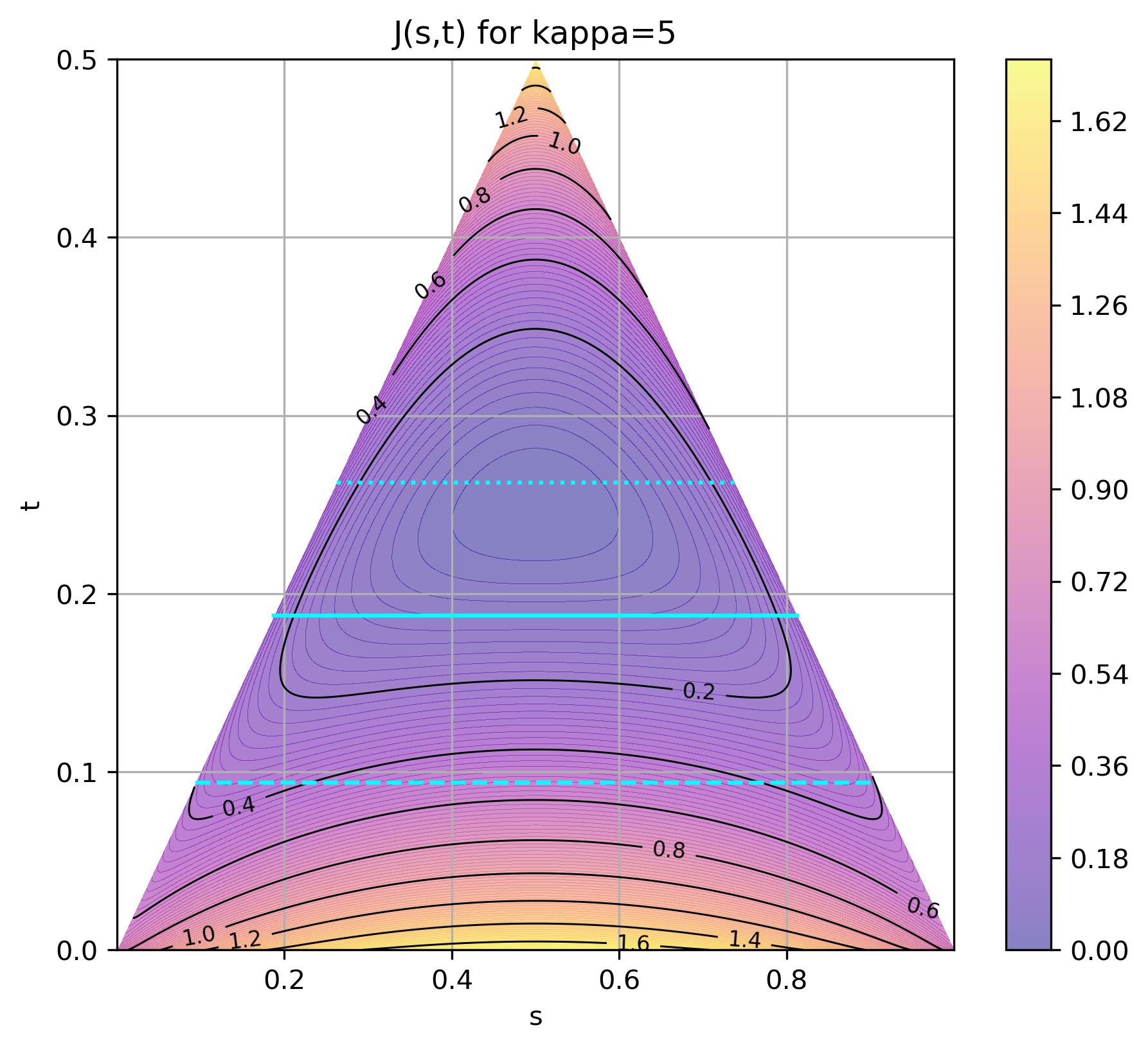}
    \caption{Contour plot for $J^{(2)}_\k(s,t)$}
    \label{fig:uniform-heatmap}
  \end{subfigure}
  \hspace{0.5cm}
  \begin{subfigure}{0.41\textwidth}
    \centering
    \includegraphics[width=\linewidth,height=5.5cm]{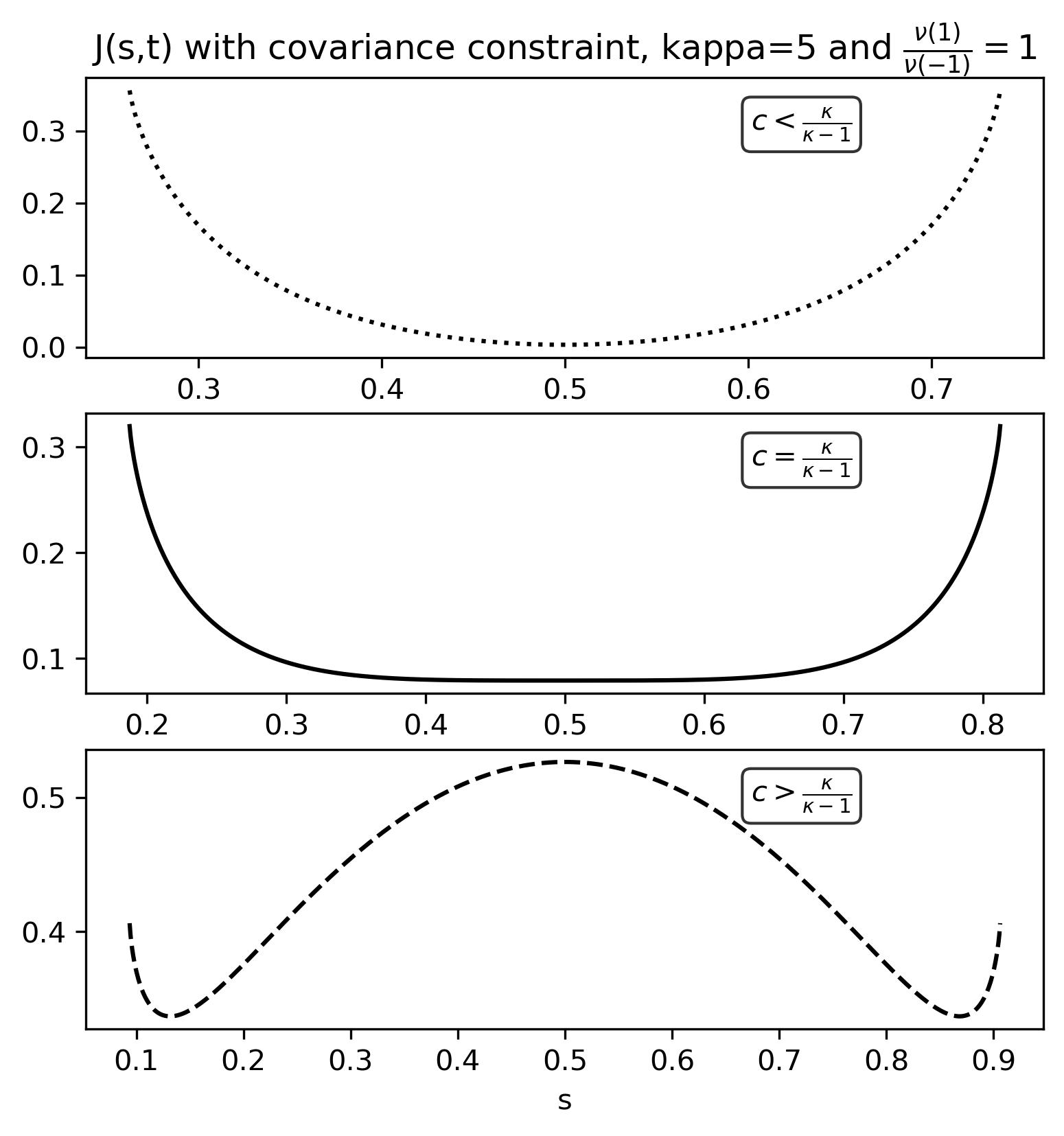}
    \caption{$s\mapsto J^{(2)}_\k(s,t(c))$}
    \label{fig:uniform-section}
  \end{subfigure}

  \caption{Two-spin with consensus functional for $\k=5$. Top: non-uniform mark distribution $\nu$ (no phase transition). Bottom: uniform mark distribution $\nu$ (phase transition).}
  \label{fig:two-spin-phase-transition-2x2}
\end{figure}

From Theorem \ref{thm:Gibbs-deg-and-nondeg}, we see that the conditional limits are $\TIS$ Ising measures (possibly degenerate), but, according to Lemma \ref{lem:cavitymap} and Definition \ref{def:Ising-on-tree}, there could be multiple $\TIS$ Ising measures in certain regimes. The next lemma is then used to identify the exact $\TIS$ Ising measures in each case.

 \begin{lemma}[Root magnetization and consensus under certain Ising measures]\label{lemma:root-mag-consensus-for-ising}
 Let $\ising$ be a $\TIS$ Ising measure $\ising$ on $T^\k$ with temperature parameter $\beta$, external field $B$ and boundary law $\ell\in \P\cpr{\pr{1,-1}}$ corresponding to a fixed point $\theta\in \R$ of the Ising cavity map, as specified in Remark \ref{remark:BP-and-ising-cavity} below. According to Remark \ref{rmk:external-field-p}, let $\nu\in \P(\{1,-1\})$ be such that $\nu(1)=p$ with $B=\log(p/(1-p))$ and define $\smref$ with respect to $\nu$ in each cases. 
 Then the following statements hold:  
 \begin{enumerate}
     \item In the ferromagnetic regime $(\beta>0)$:
     \begin{enumerate}
         \item If $B>0$, then $\ising^+_\k(\beta,B)$ is the unique $\TIS$ Ising measure that has a positive root magnetization. Moreover, the consensus under $\ising^+_\k(\beta,B)$ is greater than typical, that is, 
         \begin{equation}
             \bbE_{\ising^+_\k(\beta,B)}\sqpr{S_m}>\smref.\label{eqn:ferro-plus>sref}
         \end{equation}
         \item If $B<0$, then $\ising^-_\k(\beta,B)$ is the unique $\TIS$ Ising measure that has negative root magnetization. Moreover, the consensus under $\ising^-_\k(\beta,B)$ is greater than typical,  that is, 
         \begin{equation}
             \bbE_{\ising^-_\k(\beta,B)}\sqpr{S_m}>\smref.\label{eqn:ferro-minus>sref}
         \end{equation}
         \item If $B=0$, then in the uniqueness regime $\beta\leq \betacrit(\k,0)$, we have
         \begin{align}
             \bbE_{\ising_\k^*(\beta,0)}[\pmb\sigma_o]=0\quad \text{and}\quad\bbE_{\ising_\k^*(\beta,0)}[S_m]=\k\tanh(\beta)>\smref=0\label{eqn:ferro-root-mag-consensus-unique-no-external}.
         \end{align}
         In the non-uniqueness regime $\beta> \betacrit(\k,0)$, we have
         \begin{align}
             \bbE_{\ising_\k^+(\beta,0)}[\pmb\sigma_o]>\bbE_{\ising_\k^\sharp(\beta,0)}[\pmb\sigma_o]=0>\bbE_{\ising_\k^-(\beta,0)}[\pmb\sigma_o]\label{eqn:ferro-root-mag-nonunique-no-external}
         \end{align}
         and
    \begin{align}
    \bbE_{\ising_\k^+(\beta,0)}[S_m]=\bbE_{\ising_\k^-(\beta,0)}[S_m]>\bbE_{\ising_\k^\sharp(\beta,0)}[S_m]=\k\tanh(\beta)>\smref=0\label{eqn:ferro-consensus-nonunique-no-external}.
    \end{align}
     \end{enumerate}
     \item In the antiferromagnetic regime ($\beta<0$):
     \begin{enumerate}
         \item If $B>0$, then
         \begin{equation}
             \bbE_{\ising_\k^\sharp(\beta,B)}[\pmb\sigma_o]>0\quad \text{and}\quad \bbE_{\ising_\k^\sharp(\beta,B)}[S_m]<\smref.\label{eqn:antiferro-root-mag-consensus-plus}
         \end{equation}
         \item If $B<0$, then
         \begin{equation}
             \bbE_{\ising_\k^\sharp(\beta,B)}[\pmb\sigma_o]<0\quad \text{and}\quad \bbE_{\ising_\k^\sharp(\beta,B)}[S_m]<\smref.\label{eqn:antiferro-root-mag-consensus-minus}
         \end{equation}
         \item If $B=0$, then
         \begin{equation}
             \bbE_{\ising_\k^\sharp(\beta,0)}[\pmb\sigma_o]=0\quad \text{and}\quad \bbE_{\ising_\k^\sharp(\beta,0)}[S_m]=\k\tanh(\beta)<\smref=0.\label{eqn:antiferro-root-mag-consensus-no-external}
         \end{equation}
     \end{enumerate}
 \end{enumerate}
     
 \end{lemma}

The proof of the lemma, given below, will rely on the following two observations. 

\begin{remark}[Relation between the Ising cavity map $\Gamma$ and $\BP$ for Ising measure]\label{remark:BP-and-ising-cavity}
   Fix $\k\geq 2.$ The probability mass function $\ell\in \P(\pr{1,-1})$ is a fixed point of the map $\BP$ in \eqref{eqn:cavity} with $\psi=\exp\{\beta m\}$ and $\bar\psi=\nu$ if and only if $\ell(x)= e^{\exth\cdot \sgn (x)}/(e^\exth+e^{-\exth})$ for $x\in \pr{1,-1}$, where $\exth \in \R$ is a fixed point of the Ising cavity map $\Gamma=\Gamma(\k,\beta,B(p))$, defined in \eqref{eqn:Ising-cavity-map}, where $B(p)=\log(p/(1-p))$ is the external field corresponding to $\nu$ with $\nu(1)=p$ according to Remark \ref{rmk:external-field-p}.
\end{remark}

\begin{remark}[Root magnetization and consensus of $\TIS$ Ising measures]\label{remark:magnetization-and-consensus}
Fix $\k\geq 2$, $\beta\in \R$ and $B\in \R$. Given a $\TIS$ Ising measure $\ising$ on $T^\k$ with temperature parameter $\beta$, external field $B$ and boundary law $\ell\in \P(\pr{1,-1})$, using the characterization of the edge marginal given in \eqref{eqn:edge-marginal-cavity}, together with the correspondence between $\ell\in \P(\pr{1,-1})$ and $\exth\in \R$ as specified in Remark \ref{remark:BP-and-ising-cavity}, we obtain the following forms for root magnetization and consensus of $\ising$:
\begin{align}
    \bbE_{\ising}[\pmb \sigma_o]&=\cpr{e^{-2\beta}+\cosh(2\exth)}^{-1}\sinh(2\exth)\label{eqn:ising-magnetization}\\
    \bbE_{\ising}[S_m]&=\k\cdot\cpr{ \dfrac{e^{2\beta}\cosh(2\exth)-1}{e^{2\beta}\cosh(2\exth)+1}}=\k\cdot \cpr{\dfrac{2\sinh^2(\exth)+(1-e^{-2\beta})}{2\cosh^2(\exth)-(1-e^{-2\beta})}}\label{eqn:ising-covariance}.
 \end{align}
\end{remark}

 \begin{proof}[Proof of Lemma \ref{lemma:root-mag-consensus-for-ising}] 
{\em (1a)} From \eqref{eqn:ising-magnetization}, we see that $\sgn\cpr{\bbE_\ising[\pmb\sigma_o]}=\sgn(\exth)$. Together with Lemma \ref{lem:cavitymap} (1), this proves the uniqueness. Next, we show \eqref{eqn:ferro-plus>sref}. Note that in the ferromagnetic regime, since $1-e^{-2\beta}>0$, by \eqref{eqn:ising-covariance} we have
\begin{align}
    \bbE_{\ising}[S_m]>\k \tanh^2(\exth).\label{eqn:ferro-consensus-greater}
\end{align}
Moreover, since $B>0$, the unique positive fixed point $\exth^+$ of the Ising cavity map $\Gamma$ is such that
\begin{align*}
    \exth^+=B+(\k-1)\tanh^{-1}\cpr{\tanh(\beta)\tanh(\exth^+)}>B,
\end{align*}
together with \eqref{eqn:ferro-consensus-greater} and the monotonicity of $\tanh$, we have
\begin{align*}
    \bbE_{\ising_\k^+(\beta,B)}[S_m]>\k \tanh^2(\exth^+)>\k\tanh^2(B)=\smref.
\end{align*}

{\em (1b)} This follows analogously using Lemma \ref{lem:cavitymap} (2), \eqref{eqn:ferro-consensus-greater} and the monotonicity argument by replacing $\exth^+$ with $\exth^-$.

{\em (1c)} From \eqref{eqn:ising-magnetization} and Lemma \ref{lem:cavitymap} (3), we see that $\exth=0$ and $\bbE_{\ising_\k^*(\beta,0)}[\pmb\sigma_o]=0$ in the uniqueness regime. Moreover, since $\exth=0$ and $\beta>0$, \eqref{eqn:ising-covariance} implies that $\bbE_{\ising_\k^*(\beta,0)}[S_m]=\k\tanh(\beta)>0=\smref$. In the non-uniqueness regime, \eqref{eqn:ising-magnetization} and Lemma \ref{lem:cavitymap} (3) imply \eqref{eqn:ferro-root-mag-nonunique-no-external}. Moreover, since the Ising cavity map with $B=0$ is an odd function, we see that $|\exth^+|=|\exth^-|>0=|\exth^\sharp|$. Applying \eqref{eqn:ising-covariance} and the fact that $\cosh$ is an even function that is strictly increasing in $[0,\infty)$, we obtain \eqref{eqn:ferro-consensus-nonunique-no-external}.

{\em (2a)} By Lemma \ref{lem:cavitymap}, in the antiferromagnetic regime, since $B>0$, there is always a unique fixed point $\exth^\sharp$ and $\sgn(\exth^\sharp)=\sgn(B)>0$. Together with \eqref{eqn:ising-magnetization}, we conclude that $\bbE_{\ising_\k^\sharp(\beta,B)}[\pmb\sigma_o]>0$. Next, note that since $\exth^\sharp>0$ and $B>0$,
\begin{equation*}
    \exth^\sharp=B+(\k-1)\tanh^{-1}\cpr{\tanh(\beta)\tanh(\exth^\sharp)}<B,
\end{equation*}
thus $|\exth^\sharp|<|B|$. Together with \eqref{eqn:ising-covariance}, the fact that $1-e^{-2\beta}<0$ and the fact that $[0,\infty)\ni x\mapsto \tanh^2(x)$ is an increasing function, we have
\begin{align}
    \bbE_{\ising_\k^\sharp(\beta,B)}[S_m]<\k \tanh^2(\exth^\sharp)<\k\tanh^2(B)=\smref.\nonumber
\end{align}

{\em (2b)} This follows analogously but this time $\sgn(\exth^\sharp)=\sgn(B)<0$ so $\bbE_{\ising_\k^\sharp(\beta,B)}[\pmb\sigma_o]<0$. Moreover, $0>\exth^\sharp>B$ so again $|\exth^\sharp|<|B|$. Therefore, the consensus part of (2b) follows accordingly.

{\em (2c)} If $B=0$, then by the antiferromagnetic part of Lemma \ref{lem:cavitymap}, the unique fixed point is $\exth^\sharp=0$, hence \eqref{eqn:ising-magnetization} implies that $\bbE_{\ising_\k^\sharp(\beta,B)}[\pmb\sigma_o]=0$ and \eqref{eqn:ising-covariance} implies that 
\begin{equation*}
    \bbE_{\ising_\k^\sharp(\beta,0)}[S_m]=\k \tanh(\beta)<0=\smref.
\end{equation*}
 \end{proof}

We now turn to the setting where $c$ lies in the interior of the convex hull of the range of $S_m$.

 \begin{proof}[Proof of Theorem \ref{thm:Ising-phase-transition} when $c \in (-\k, \k)$]  We consider two cases based on whether $\nu$ is uniform or not, and then divide further into the  cases when $s \mapsto J^\nu_\k(s, t(c))$  is convex or not.  In each  case, the proof  will proceed via two major steps: we first solve for the (global) minimizer(s) of the edge optimization problem explicitly using the parameterization for two-spin systems,  then we identify the Ising measure(s) corresponding to the minimizer(s). Before we start, note that $\cc\in (-\k,\k)$ implies 
 \begin{equation}
     t(c)=\frac{1}{4}\cpr{1-\frac{\cc}{\k}}\in \cpr{0,\frac{1}{2}}. \label{eqn:tc-in-interior}
 \end{equation}
 {\em Case 1:}  Suppose  $\nu=\mathsf{Unif}(\X)$.  We start by establishing properties of the objective function $J_\k^{(2)}$. 
 First, note that by the definition  \eqref{eqn:Jst} of $J_\k^{(2)}$, it is easy to see that
 \begin{equation}
     s\mapsto J^{(2)}_\k(s,t(c))\text{ is continuous at } s=t(c)\text{ and }s=1-t(c)\label{eqn:J2-conti-at-boundary}
 \end{equation}
 and
 \begin{equation}
     s\mapsto J^{(2)}_\k(s,t(c)) \text{ is smooth  in } (t(c),1-t(c)).\label{eqn:J2-smooth}
 \end{equation}
 In addition, 
 \begin{equation}
     J^{(2)}_\k(s,t)=J^{(2)}_\k(1-s,t),\quad \forall (s,t)\in \Delta. \label{eqn:J2s=J21-s}
 \end{equation}
 Next, \eqref{eqn:Jst-ds} implies that
 \begin{equation}
     \partial_s J_\k^{(2)}(\frac{1}{2},t(c))=0\label{eqn:1/2-crit-pt}
 \end{equation}
 and
    \begin{equation}
        \partial_s J^{(2)}_\k(s,t(c))\rightarrow -\infty \text{ as }s\downarrow t(c)\quad \text{and}\quad \partial_s J^{(2)}_\k(s,t(c))\rightarrow\infty \text{ as }s\uparrow 1-t(c)\label{eqn:J-sym-boundary-derivative}.
    \end{equation}
     Moreover, since taking the derivative in \eqref{eqn:Jst-ds} yields 
    \begin{align*}
        \partial_s^2 J^{(2)}_\k(s,t)
        &=\frac{1-\k}{s(1-s)}+\frac{\k (1-2t)}{2(s-t)(1-s-t)},
    \end{align*}
    we have
    \begin{equation}\label{eqn:Jst-dsds-to-W}
        {\sgn} (\partial_s^2 J^{(2)}_\k(s,t(c))) = {\sgn} (W_c(s)),\quad \forall s\in \cpr{t(c),1-t(c)},
    \end{equation}
    where
    $W_c(s)$ is the following degree-two polynomial with positive leading coefficient:
    \begin{align}
        W_c(s)&\coloneqq 2(1-\k)(s-t(c))(1-s-t(c))+\k s(1-s)(1-2t(c))\nonumber\\
        &=\cpr{\frac{3\k-c-4}{2}}\cpr{s-\frac{1}{2}}^2+\frac{(\k+c)(\k+c-\k c)}{8\k^2},\label{eqn:Wc}
    \end{align}
    where the last equality uses $t(c)=\frac{1}{4}(1-\frac{c}{\k}).$ Note that
    \begin{align}
    \cc < \k \mbox{ and } \k \geq 2  \quad  &\Rightarrow   \quad 
    \frac{3\k-c-4}{2} > 0, \label{eqn:firterm}  \\
    c \leq \frac{\k}{\k-1} \quad & \Leftrightarrow \quad W_c(\frac{1}{2})=\frac{(\k+c)(\k+c-\k c)}{8\k^2} \geq 0.\label{eqn:secterm}
    \end{align} 
    In addition, due to \eqref{eqn:tc-in-interior}, we have
    \begin{equation}
\label{eqn:thirdterm} 
    \cc \in (-\k,\k)   \quad \Rightarrow \quad    
    W_c(t(c))=W_c(1-t(c))>0.
    \end{equation}
    Therefore, when $\cc\leq \k/(\k-1)$, by \eqref{eqn:Jst-dsds-to-W}, \eqref{eqn:firterm} and \eqref{eqn:secterm}, it follows that
      \begin{align}
        \partial_s^2 J^{(2)}_\k(s,t(c))> 0,\quad  \forall s\in (t(c),1-t(c))\setminus \{\frac{1}{2}\}\quad \text{if }\cc\leq \frac{\k}{\k-1} \label{eqn:convexity-J2}
      \end{align}
    On the other hand, when $\cc>\k/(\k-1)$, \eqref{eqn:Jst-dsds-to-W}, \eqref{eqn:firterm}--\eqref{eqn:thirdterm}, and the symmetry of $W_c$ around $s=1/2$ implied by \eqref{eqn:Wc}, together yield the existence of $r$ with $\{\frac{1}{2}-r,\frac{1}{2}+r\}\subset (t(c),1-t(c))$ such that
\begin{equation}
    \sgn(\partial_s^2J^{(2)}_\k(s,t(c)))=\sgn\cpr{(s-\frac{1}{2}+r)(s-\frac{1}{2}-r)},\quad \forall s\in (t(c),1-t(c))\quad \text{ if }c>\frac{\k}{\k-1}\label{eqn:sgn-Jst-dsds}
\end{equation}

{\bf Note: } Properties \eqref{eqn:J2-conti-at-boundary}-- \eqref{eqn:sgn-Jst-dsds} above can and will also be used in {\em Case 2} 
    
    We now further divide Case 1 into further cases, depending on whether or not  $J_\k^{(2)}$ is  convex. \\
    
    {\em Case 1A:} When $\cc\leq \k/(\k-1)$, using \eqref{eqn:1/2-crit-pt}, together with \eqref{eqn:J2-conti-at-boundary},\eqref{eqn:J-sym-boundary-derivative} and \eqref{eqn:convexity-J2}, it follows that $s=\frac{1}{2}$ is the unique global minimizer of the map $[t(c),1-t(c)]\ni s\mapsto J^{(2)}_\k(s,t(c))$. Hence, $\optset_{edge}(\cc)=\pr{\pi_*}$, where $\pi_*=\pi[\frac{1}{2},t(c)]$. Next, \eqref{eqn:opti-pi-to-opti-rho} implies that $\optset(\cc)=\pr{\UGW_1(\mu^{(\pi_*)})}$. Then  Proposition \ref{prop:int_local_min_to_Gibbs}, the fact that $\nu=\mathsf{Unif}(\X)$ with Remark \ref{rmk:external-field-p}, and Lemma \ref{lemma:unimodular-extension-of-splitting-Gibbs-measure} together imply that $\UGW_1(\mu^{(\pi_*)})=\ising$ is a $\TIS$ Ising measure for some $\beta\in \R$ and $B=0$.
    
    We now identify this  Ising measure $\ising$. Note that by \eqref{eqn:vertex-marginal-of-pist}, the root magnetization of $\ising$ is 
    \begin{equation*}
        \bbE_{\UGW_1(\mu^{(\pi_*)})}[\pmb\sigma_o]=\bbE_{\pi[\frac{1}{2},t(c)]}[\bfX_o]=0.
    \end{equation*}
    By (1c) and (2c) of Lemma \ref{lemma:root-mag-consensus-for-ising}, we see that either $\ising = \ising_\k^*(\beta,0)$ in the ferromagnetic uniqueness regime, that is, when  $\beta \in (0,\betacrit]$,  or $\ising = \ising^\sharp_\k(\beta,0)$ otherwise, that is,  when $\beta \in (-\infty, 0) \cup (\betacrit, \infty)$.  Moreover, in all cases, the consensus $\cc$ is equal to  $\k \tanh(\beta)$, which implies that  $\beta=\beta(c) := \tanh^{-1}(\cc/\k)$.  A comparision  with \eqref{eqn:ferro-root-mag-consensus-unique-no-external}, \eqref{eqn:ferro-consensus-nonunique-no-external} and \eqref{eqn:antiferro-root-mag-consensus-no-external} 
        then yields the following implications: \\
        (i) when $0<c\leq \k/(\k-1)$,  $\beta$ lies in the interval $(0,\tanh^{-1}(1/(\k-1)))$, which, according to Definition \ref{def:Ising-on-tree}(1a), is contained in the uniqueness regime and thus $\ising=\ising_\k^*(\tanh^{-1}(\cc/\k,0))$. \\
        (ii) When $c<0$, $\beta(c) <0$, and thus $\ising=\ising^\sharp_\k(\tanh^{-1}(\cc/\k),0)$. \\
        Together with Proposition \ref{prop:gibbs cond}, these observations respectively prove 
        (3a) and the case $p = 1/2$ in (1) of
        Theorem \ref{thm:Ising-phase-transition}.  
     
     {\em Case 1B:} When $c> \k/(\k-1)$, 
    by \eqref{eqn:1/2-crit-pt} and \eqref{eqn:sgn-Jst-dsds}, we see that $\partial_sJ_\k^{(2)}(\frac{1}{2}-r,t(c))>0$ 
    and that $(t(c),1-t(c))\ni s\mapsto J^{(2)}_\k(s,t(c))$ has exactly two inflection points at $\frac{1}{2}-r$ and $\frac{1}{2}+r$. When combined with \eqref{eqn:J-sym-boundary-derivative} and the fact that $(t(c), 1-t(c)) \ni s \to \partial_s J^{(2)}_\k(s,t(c))$ is continuous, due to \eqref{eqn:J2-smooth}, this implies that there exists $s_*\in (t(c),\frac{1}{2}-r)$ such that $s_*$, and thus $1-s_*$ due to \eqref{eqn:J2s=J21-s}, are both critical points of the map $(t(c),1-t(c))\ni s\mapsto J^{(2)}_\k(s,t(c))$. Moreover, since any function with exactly two inflection points can have at most three critical points, we know that $s_*,\frac{1}{2},1-s_*$ are the only three critical points of $J^{(2)}_\k(s,t(c))$. Furthermore, \eqref{eqn:J2-conti-at-boundary} and \eqref{eqn:J-sym-boundary-derivative} imply that any local minimizer must lie in $(t(c),1-t(c))$, \eqref{eqn:sgn-Jst-dsds} implies that $s\mapsto J^{(2)}_\k(s,t(c))$ is strictly concave at $\frac{1}{2}$ and strictly convex at $s_*$ and $1-s_*$, and \eqref{eqn:J2s=J21-s} implies that the two local minimizers have the same value $J_\k^{(2)} (s_*, t(c))=J_\k^{(2)} (1-s_*, t(c))$. Therefore, $s_*$ and $1-s_*$ must correspond to the two global minimizers. Hence, we have $\optset_{edge}(\cc)=\pr{\pi[s_*,t(c)],\pi[1-s_*,t(c)]}$. Next, \eqref{eqn:opti-pi-to-opti-rho} implies that $\optset(\cc)=\pr{\UGW_1\cpr{\mu^{(\pi[s_*,t(c)])}},\UGW_1\cpr{\mu^{(\pi[1-s_*,t(c)])}}}$. Then Proposition \ref{prop:int_local_min_to_Gibbs}, the fact that $\nu=\mathsf{Unif}(\X)$,  Remark \ref{rmk:external-field-p}, and Lemma \ref{lemma:unimodular-extension-of-splitting-Gibbs-measure} together imply that $\UGW_1\cpr{\mu^{(\pi[s_*,t(c)])}}$ and $\UGW_1\cpr{\mu^{(\pi[1-s_*,t(c)])}}$ are certain $\TIS$ Ising measures with some $\beta\in \R$ and $B=0$.

     We now identify the Ising measures corresponding to the two global minimizers in $\optset(\cc)$. By \eqref{eqn:vertex-marginal-of-pist} and the fact that $s_*<\frac{1}{2}$, we know that these two $\TIS$ Ising measures have nonzero root magnetization.  
     By \eqref{eqn:ferro-root-mag-consensus-unique-no-external}, \eqref{eqn:ferro-consensus-nonunique-no-external} and \eqref{eqn:antiferro-root-mag-consensus-no-external} of Lemma 
     \ref{lemma:root-mag-consensus-for-ising}, we see that  $\beta>\tanh^{-1}(1/(\k-1))$ (otherwise there is no $\TIS$ Ising measure with nonzero root magnetization) and their corresponding Ising measures should be $\ising_\k^-(\beta,0)$ and $\ising_\k^+(\beta,0)$, respectively. Together with Proposition \ref{prop:gibbs cond}, this completes the proof of (3b) in Theorem \ref{thm:Ising-phase-transition}.

 {\em Case 2:} 
Suppose $\nu\neq \mathsf{Unif}(\X)$. 
In this case, without loss of generality, we assume that $p = \nu(1) <1/2$. For any $s\in (t(c),1-t(c))$, by \eqref{eqn:tc-in-interior}, we have $(s,t(c))\in int(\Delta)$, where $int(\Delta)$ is the interior of $\Delta$ defined in \eqref{eqn:Delta-interior}. Together with \eqref{eqn:Jnu-two-spin} and \eqref{eqn:Jst-ds}, it follows that
    \begin{align}
        \partial_s J^\nu_\k(s,t(c))&=\log\cpr{\frac{\nu(-1)}{\nu(1)}}+\partial_s J^{(2)}_\k(s,t(c))>\partial_s J^{(2)}_\k(s,t(c)),\label{eqn:J-p<1/2-first-derivative}\\
        \partial_s^2 J^\nu_\k(s,t(c))&=\partial_s^2 J^{(2)}_\k(s,t(c)).\label{eqn:J-p<1/2-second-derivative}
    \end{align}
    We once again split into two further cases depending on $\cc\leq \k/(\k-1)$ or $\cc>\k/(\k-1)$.
    
    {\em Case 2A:} Suppose $c\leq \k/(\k-1)$. Note that \eqref{eqn:J-p<1/2-first-derivative} and \eqref{eqn:1/2-crit-pt} imply that $\partial_sJ^\nu_\k(\frac{1}{2},t(c))>0$, \eqref{eqn:J-p<1/2-first-derivative} and \eqref{eqn:J-sym-boundary-derivative} imply that $\partial_s J^\nu_\k(s,t(c))\rightarrow -\infty$ as $s\downarrow t(c)$. Since $(t(c),\frac{1}{2})\ni s\mapsto \partial_sJ^\nu_\k(s,t(c))$ is continuous due to \eqref{eqn:J-p<1/2-first-derivative} and \eqref{eqn:J2-smooth}, this implies that $s\mapsto J^\nu_\k(s,t(c))$ has a critical point at $s_*\in (t(c),\frac{1}{2})$.
    By \eqref{eqn:J-p<1/2-second-derivative} with \eqref{eqn:convexity-J2}, and \eqref{eqn:Jnu-two-spin} with \eqref{eqn:J2-conti-at-boundary}, it follows that $s_*$ is the unique global minimizer of the map $[t(c),1-t(c)]\ni s\mapsto J^\nu_\k(s,t(c))$. Hence, $\optset_{edge}(\cc)=\pr{\pi_*}$, where $\pi_*=\pi[s_*,t(c)]$. In turn, \eqref{eqn:opti-pi-to-opti-rho} implies that $\optset(\cc)=\pr{\UGW_1\cpr{\mu^{(\pi_*)}}}$, and then Proposition \ref{prop:int_local_min_to_Gibbs}, Lemma \ref{lemma:unimodular-extension-of-splitting-Gibbs-measure} and Remark \ref{rmk:external-field-p} together imply that $\UGW_1\cpr{\mu^{(\pi_*)}}=\ising$ is a $\TIS$ Ising measure with some $\beta\in \R$ and $B=\log\cpr{\frac{\nu(1)}{\nu(-1)}}<0$.

    Next, in order to identify the Ising measure $\ising$ corresponding to the global minimizer $\UGW_1\cpr{\mu^{(\pi_*)}}$, first note that by \eqref{eqn:vertex-marginal-of-pist}, the root magnetization of $\ising$ is given by 
    \begin{equation*}
        \bbE_{\UGW_1\cpr{\mu^{(\pi_*)}}}[\pmb\sigma_o]= \bbE_{\pi[s_*,t(c)]}[\bfX_o]= 2s_*-1<0.
    \end{equation*}
    By (1b) and (2b) of Lemma \ref{lemma:root-mag-consensus-for-ising}, it follows that that $\ising=\ising_\k^-(\beta,B)$ for some $\beta>0$ when $\cc>\smref$, and $\ising=\ising_\k^\sharp(\beta,B)$ for some $\beta<0$ when $\cc<\smref$. In conclusion, $\optset(\cc)=\{\ising_\k^-(\beta,B)\}$ for some $\beta>0$ if $\smref<\cc\leq \k/(\k-1)$ and $\optset(\cc)=\{\ising_\k^\sharp(\beta,B)\}$ for some $\beta<0$ if $-\k<\cc<\smref.$

    {\em Case 2B}: When $c> \k/(\k-1)$, due to \eqref{eqn:J-p<1/2-first-derivative} and \eqref{eqn:J-sym-boundary-derivative}, if we let $s_*$ be the critical point of $s\mapsto J^{(2)}_\k(s,t(c))$ that is in $(t(c),\frac{1}{2})$ (as argued in {\em Case 1B}), then
    \begin{equation}
        \partial_s J^\nu_\k(s,t(c))\rightarrow -\infty \text{ as }s\downarrow t(c )\quad \text{and}\quad  \partial_sJ^\nu_\k(s_*,t(c))=\log\cpr{\frac{\nu(-1)}{\nu(1)}}>0.\label{eqn:Jnuds-downlimit}
    \end{equation}
    Therefore, the continuity of the map $(t(c),1-t(c))\in s\mapsto \partial_s J^\nu_\k(s,t(c))$ implies that there exists a critical point $\tilde{s}_*$ of $s\mapsto J^\nu_\k(s,t(c))$ in $(t(c),s_*)$. In fact, though $\tilde{s}_*$ might not be the unique critical point of $J^\nu_\k(s,t(c))$, it is the unique critical point that lies in $(t(c),\frac{1}{2}]$ because \eqref{eqn:J-p<1/2-second-derivative} and \eqref{eqn:sgn-Jst-dsds} imply $\sgn(\partial_s^2J^{\nu}_\k(s,t(c)))>0$, $\forall s\in (t(c),s_*)$, and \eqref{eqn:J-p<1/2-first-derivative}, the fact that $s_*$ is the unique critical point of $s\mapsto J^{(2)}_\k(s,t(c))$ in $(t(c),\frac{1}{2})$ and the continuity of $(t(c),\frac{1}{2})\ni s\mapsto \partial_s J^{(2)}_\k(s,t(c))$ imply that $\partial_sJ^\nu_\k(s,t(c))>\partial_sJ^{(2)}(s,t(c))\geq 0$, $\forall s\in [s_*,\frac{1}{2}]$.  Next, we argue by contradiction to show that any global minimizer $s'_*$ of $J^\nu_\k(s,t(c))$ must lie in $[t(c),\frac{1}{2}]$. Indeed, if this were not the case, then   by \eqref{eqn:Jnu-two-spin}, the fact that $J_\k^{(2)}(1-s'_*,t(c))=J_\k^{(2)}(s'_*,t(c))$, due to \eqref{eqn:J2s=J21-s}, together with the assumption $\nu(-1)>\nu(1)$ and $s'_*>1/2$, it follows that $J_\k^\nu(1-s'_*,t(c))<J_\k^\nu(s'_*,t(c))$, contradicting the fact that $s'_*$ is a global minimizer of $J^\nu_\k(s,t(c))$. Moreover, note that, by \eqref{eqn:Jnu-two-spin}, \eqref{eqn:J2-conti-at-boundary} and \eqref{eqn:Jnuds-downlimit}, $s=t(c)$ cannot be a global minimizer. This ensures that the critical point $\tilde{s}_*$ is in fact the unique global minimizer of $[t(c),1-t(c)]\ni s\mapsto J_\k^\nu(s,t(c))$. Hence, $\optset_{edge}(\cc)=\pr{\pi_*}$, where $\pi_*=\pi[\tilde s_*,t(c)]$. Next, \eqref{eqn:opti-pi-to-opti-rho} implies that $\optset(\cc)=\pr{\UGW_1\cpr{\mu^{(\pi_*)}}}$. Then Proposition \ref{prop:int_local_min_to_Gibbs}, Lemma \ref{lemma:unimodular-extension-of-splitting-Gibbs-measure} and Remark \ref{rmk:external-field-p} together imply that $\ising = \UGW_1\cpr{\mu^{(\pi_*)}}$ is a $\TIS$ Ising measure with some $\beta\in \R$ and $B=\log\cpr{\frac{\nu(1)}{\nu(-1)}}<0$.
    
    To identify the  Ising measure corresponding to $\ising$, note that by \eqref{eqn:vertex-marginal-of-pist}, the root magnetization of $\ising$ is:
    \begin{equation*}
        \bbE_{\UGW_1\cpr{\mu^{(\pi_*)}}}[\pmb\sigma_o]= \bbE_{\pi[\tilde s_*,t(c)]}[\bfX_o]= 2\tilde s_*-1<0.
    \end{equation*}
    Then (1b) and (2b) of Lemma \ref{lemma:root-mag-consensus-for-ising} imply that $\ising=\ising_\k^-(\beta,B)$ for some $\beta>0$ if $\cc>\smref$, and $\ising=\ising_\k^\sharp(\beta,B)$ for some $\beta<0$ if $\cc<\smref$. In conclusion, $\optset(\cc)=\{\ising_\k^-(\beta,B)\}$ for some $\beta>0$ if $\cc>\max\pr{\smref, \k/(\k-1)}$ and $\optset(\cc)=\{\ising_\k^\sharp(\beta,B)\}$ for some $\beta<0$ if $\smref>\cc>\k/(\k-1).$

    Combining the conclusions in {\em Case 2A} and {\em Case 2B}, we see that $\optset(\cc)=\{\ising_\k^-(\beta,B)\}$ for some $\beta>0$ if $\cc>\smref$ and $\optset(\cc)=\{\ising_\k^\sharp(\beta,B)\}$ for some $\beta<0$ if $\cc<\smref$. Together with Proposition \ref{prop:gibbs cond}, this completes the proof of Theorem \ref{thm:Ising-phase-transition} in the case when $c\in (-\k,\k)$ and $\nu\neq \mathsf{Unif}(\X)$. 
\end{proof}

\appendix

\section{From  minimizing edge marginals to $\TIS$ Gibbs measures}\label{sec:pf-of-TISGM-consistency}

\begin{proof}[Proof of] Lemma \ref{lemma:TISGM-to-cavity}]
    The one-to-one correspondence and the first two assertions of the lemma follow from \cite[Theorem 12.12 and Corollary 12.17]{georgii2011gibbs}. In particular, for any $\ell\in \Delta^*$, the probability distributions on $T^\k_r, r \in \N,$ defined using \eqref{eqn:marginal-splitting-Gibbs} form a consistent family of probability measures. By the Kolmogorov extension theorem, they 
    are the marginals of some $\rho\in \P\cpr{\X^{T^\k}}$. By Definition \ref{def:TISGM}, $\rho$ belongs to $\TISGM^{\psi,\bar\psi}_{T^\k}$.
    Therefore, the depth-$1$ marginal of this  $\TIS$ Gibbs measure $\rho$ is
    \begin{equation*}
        \pmb\rho_1(\tau)=\pmb\rho_{T_1}(\tau)=\frac{1}{Z_{T_1}} \bar{\psi}(x_o)\prod_{i=1}^\k\psi(x_o,x_i)\ell(x_i),\quad \forall \tau=(x_o,x_1,\ldots,x_\k)\in \X^{T^\k_1}. 
    \end{equation*}
    Using this along with the symmetry of $\psi$, interchanging the sum and product in the second line below and using $\ell\in \Delta^*$  in the third line below, we conclude that for any $x,z\in \X$,
    \begin{align*}
        \pi_{\pmb\rho_1}(x,z)&=\frac{1}{Z_{T_1}} \sum_{x_o,x_1,x_2,\ldots,x_\k \in \X}\bar{\psi}(x_o)\prod_{i=1}^\k\psi(x_o,x_i)\ell(x_i)\indf_{\pr{x_1=x,x_o=z}}\\
        &=\frac{1}{Z_{T_1}}\cpr{\psi(x,z)\ell(x)}\cdot\bar{\psi}(z)\cpr{\sum_{y\in \X}\psi(z,y)\ell(y)}^{\k-1}\\
        &=\frac{\zeta}{Z_{T_1}}\ell(x)\psi(x,z)\ell(z),
    \end{align*}
    where $\zeta$ is the normalizing constant in \eqref{eqn:cavity}. This proves  \eqref{eqn:edge-marginal-cavity}.
    
\end{proof}

\section{Proof of the unimodular extension property for $\TIS$ Gibbs measures}\label{sec:proof-unimodular-extension-of-splitting-Gibbs-measure}
In this section, we prove Lemma \ref{lemma:unimodular-extension-of-splitting-Gibbs-measure}. We start with an observation. Recall $\eta \in \P(\TkaXinf)$ from Remark \ref{remark:true-law} and its depth-$h$ marginal $\eta_h \in \P(\TkappaX{h})$ from  Definition \ref{def:rho_h}.

\begin{remark}
     For any $h \geq 1,$ together with Remark \ref{rmk: labeling scheme}, $\eta_h$ is a $\TISGM$ as defined in \eqref{eqn:marginal-splitting-Gibbs}, with $\psi\equiv 1, \bar{\psi}(x) = \nu(x),$  $\ell(x) = \nu(x),$ and the normalizing  constant   $Z_{T_r} = 1.$
    \label{rmk:eta-tis}
\end{remark}
\begin{proof}[Proof of Lemma \ref{lemma:unimodular-extension-of-splitting-Gibbs-measure}]

Fix any specification $(\psi, \bar{\psi})$. Using the definition of $\mu^{(\pi)}$ in \eqref{eqn:1MRF} and viewing both $\mu^{(\pi_{\pmb\rho_1})}$ and $\eta_1$ as measures on $\X^{T^\k_1}$ through the same Ulam-Harris-Neveu labeling (as in Remark \ref{rmk: labeling scheme}), we see that
\begin{equation*}
    \mu^{(\pi_{\pmb\rho_1})}(\tau)=\pi^o_{\pmb\rho_1}(x_o)\cdot \prod_{v=1}^\k \cpr{\dfrac{\pi_{\pmb\rho_1}(x_o,x_v)}{\pi_{\pmb\rho_1}^o(x_o)}},\quad \tau=(x_o,x_1,\ldots,x_\k)\in \X^{T^\k_1}.
\end{equation*}
Then given $\pmb\rho\in \TISGM^{\psi,\bar\psi}_{T^\k}$ with (homogeneous) boundary law $\ell\in \P(\X)$, using \eqref{eqn:edge-marginal-cavity} in the second line and \eqref{eqn:cavity} in the fourth line, we obtain that, for any $\tau=(x_o,x_1,\ldots,x_\k)\in \X^{T^\k_1}$, 
\begin{align*}
    \mu^{(\pi_{\pmb\rho_1})}(\tau)&=\pi_{\pmb\rho_1}(x_o,x_1)\cdot \prod_{v=2}^\k \cpr{\dfrac{\pi_{\pmb\rho_1}(x_o,x_v)}{\pi_{\pmb\rho_1}^0(x_o)}}\\
    &\propto \ell(x_o)\psi(x_o,x_1)\ell(x_1)\prod_{v=2}^\k \cpr{\dfrac{\ell(x_o)\psi(x_o,x_v)\ell(x_v)}{\sum_{z\in \X}\ell(x_o)\psi(x_o,z)\ell(z)}}\\
    &= \ell(x_o)\psi(x_o,x_1)\ell(x_1)\cdot \cpr{\sum_{z\in \X}\psi(x_o,z)\ell(z)}^{1-\k}\cdot\prod_{v=2}^\k \psi(x_o,x_v)\ell(x_v)\\
    &\propto \bar\psi(x_o)\prod_{v=1}^\k \psi(x_o,x_v)\ell(x_v).
\end{align*}
This proves \eqref{eqn:Gibbs-MRF-1}.

Next, let $\rho'\coloneqq \UGW_1(\pmb\rho_1)$, we use induction to show that for every $r\in \N$, $\rho'_r=\pmb\rho_r$. Note that $\rho'_1=\pmb\rho_1$ by definition. Suppose that $\rho'_r=\pmb\rho_r$ for some $r\in \N$, then since $\pmb\rho_1=\mu^{(\pi_{\pmb\rho_1})}$, the conditional measure on stars, defined in \eqref{eqn:PU}, becomes
\begin{align}
    \widehat{P}_{\pmb\rho_1}(z,z')(\tau)=E_1(z',z)(\tau\oplus z')\prod_{v\in N_{\tau}(o)}\pi_{\pmb\rho_1}^{1|o}(x_v|x_o),\quad  \tau\in \Tstarone^{\k-1}[\X], z,z'\in \X,\nonumber
\end{align}
and thus
\begin{align}
    \frac{d\widehat{P}_{\pmb\rho_1}(z,z')}{d\widehat{P}_{\eta_1}(z,z')}(\tau)=\prod_{v\in N_{\tau}(o)}\frac{d\pi_{\pmb\rho_1}^{1|o}}{d\eta_1^{1|o}}(x_v|x_o),\quad  \tau\in \Tstarone^{\k-1}[\X], z,z'\in \X,\label{eqn:update-TSIGM-ext}
\end{align}
Recall the definition of ancestor $a(v)$ from Section \ref{subsubsec:topology}. Using the definition of $\rho'_{r+1}$ in \eqref{muh} in the first line, the induction hypothesis and \eqref{eqn:update-TSIGM-ext} in the second line, the definition of $\pmb\rho_r$ in \eqref{eqn:marginal-splitting-Gibbs},  \eqref{eqn:edge-marginal-cavity}, and Remark \ref{rmk:eta-tis} in the third line, the fact that $\ell$ is a fixed point of the cavity map in \eqref{eqn:cavity} in the fourth line with the fact that the number of children for each vertex (except the root) is $\k-1$, and finally the definition of $\pmb\rho_{r+1}$ from  \eqref{eqn:marginal-splitting-Gibbs} and Remark \ref{rmk:eta-tis} in the fifth line, we obtain
\begin{align*}
    \frac{d\rho'_{r+1}}{d\eta_{r+1}}(\tau_{r+1})&=\frac{d\rho'_r}{d\eta_r}(\tau_r)\prod_{v_r\in \tau_{r}\setminus \tau_{r-1}}\frac{d\widehat{P}_{\pmb\rho_1}(x_{v_r},x_{p(v_r)})}{d\widehat{P}_{\eta_1}(x_{v_r},x_{p(v_r)})}(\tau_{r+1}(v_r\setminus p(v_r))_1)\\
    &=\frac{d\pmb\rho_r}{d\eta_r}(\tau_r)\prod_{v_r\in \tau_r\setminus\tau_{r-1}}\prod_{v_{r+1}\in N_{\tau_{r+1}}(v_r)\cap \tau_{r+1}}\frac{d\pi_{\pmb\rho_1}^{1|o}}{d\eta_1^{1|o}}(x_{v_{r+1}}|x_{v_r})\\
    &=\frac{1}{Z_{T^\k_r}}\prod_{\pr{u,v}\in E(T^\k_r)}\psi(x_u,x_v)\prod_{v_{r-1}\in T^\k_{r-1}} \frac{\bar\psi(x_{v_{r-1}})}{\nu(x_{v_{r-1})}}\prod_{v_r\in T^\k_r\setminus T^\k_{r-1}}\cpr{ \frac{\ell(x_{v_r})}{\nu(x_{v_r})}\prod_{\stackrel{v_{r+1}\in T^\k_{r+1}}{v_{r+1}\sim v_r}}\dfrac{\psi(x_{v_{r}},x_{v_{r+1}})\frac{\ell(x_{v_{r+1}})}{\nu(x_{v_{r+1}})}}{\sum_{z\in \X}\psi(x_{v_r},z)\ell(z)}}\\
    &=\frac{1}{Z'}\prod_{\pr{u,v}\in E(T^\k_r)}\psi(x_u,x_v)\prod_{v_{r-1}\in T^\k_{r-1}}\bar\psi(x_{v_{r-1}})\prod_{v_r\in T^\k_r\setminus T^\k_{r-1}}\cpr{\frac{\bar\psi(x_{v_r})}{\nu(x_{v_r})}\prod_{\stackrel{v_{r+1}\in T^\k_{r+1}}{v_{r+1}\sim v_r}}\psi(x_{v_r},x_{v_{r+1}})\frac{\ell(x_{v_{r+1}})}{\nu(x_{v_{r+1}})}}\\
    &=\frac{d\pmb\rho_{r+1}}{d\eta_{r+1}}(\tau_{r+1}),
\end{align*}
where $Z' = Z_{T_r}/\zeta$ with $\zeta$ given by \eqref{eqn:cavity}.
Finally, note that any measure $\rho\in\P(\Tstarone[\X])$ can be uniquely determined by its marginals $(\rho_r)_{r\in \N}$. This completes the proof.
    
\end{proof}

\section{Proof of the first reduction of the optimization problem}\label{section:kappa reg}

\begin{proof}[Proof of Lemma \ref{lemma:1MRF}]
From the definition of $\mup$ in \eqref{eqn:1MRF}, it is clear that the edge marginal $\pi_{\mup}=\pi$. Next, we show that $\mu^{(\pi)}$ is the unique minimizer of $\mu \mapsto H(\mu\|\eta_1)$ in the set $\{ \mu \in  \P(\TkappaX{1}): \pi_{\mu} = \pi\}$ by showing that for any $\mu\in \P(\TkappaX{1})$ such that $\pi_\mu=\pi$, one has \begin{equation}\label{eqn:relative-entropy-decomposition}
    H(\mu\|\eta_1)=H(\mu\|\mup)+H(\mup\|\eta_1)\geq H(\mup\|\eta_1),
\end{equation}
with a strict inequality unless $\mu=\mup.$
\par 

Let $\mu\in \P(\TkappaX{1})$ be such that $\pi_\mu=\pi$. We first show that $\mu\ll\mup$. Indeed, suppose that $(\tau,X)\in \TkappaX{1}$ and $\mup(\tau,X)=0$. Then  either $\piz(X_o)=0$ or $\pi(X_o,X_v)=0$ for some $v\in N_\tau(o).$ In the first case, $0=\piz(X_o)=\mu_o(X_o)\geq\mu(\tau,X)$. In the second case, $0=\pi(X_o,X_v)=\pi_\mu(X_o,X_v)=\frac{1}{\k}\bbE_\mu\sqpr{\sum_{u=1}^\k \indf_{\pr{\cpr{\bfX_o,\bfX_u}=\cpr{X_o,X_v}}}}\geq\frac{1}{\k}\mu(\tau,X).$
In either case, $\mu(\tau,X)=0$ and so we conclude $\mu\ll\mup$.

Using the fact that $\mu\ll\mup$ and the fact that $\eta_1$ has full support (since $\nu$ has full support), we can write 
\begin{align}
H(\mu\|\eta_1) = H(\mu\|\mup) + \E_\mu \sbrac{\logb{d\mup}{d\eta_1}}.
\label{kappa reg: mu star ineq}
\end{align}
Using the chain rule for relative entropy,   the form of $\mup$ in \eqref{eqn:1MRF},  the definition of $\eta_1$ from Remark \ref{remark:true-law}, and the fact that $\pi_\mu =\pi = \pi_{\mup}$, we obtain  
\begin{align*}
\E_\mu \sbrac{\logb {d\mup}{d\eta_1}} &= \E_{\piz}\sbrac{\logb{d\piz}{d\nu}} + \E_{\mup}\sbrac{\log\brac{{\prod_{v=1}^\kappa\frac{\pic(X_v\mid X_o)}{\nu(X_v)}}}}\\ &=\E_{\piz}\sbrac{\logb{d\piz}{d\nu}}+ \kappa \E_{\pi}\sbrac{\logb{\pic(X_1\mid X_o)}{\nu(X_1)}} \\
&= \E_{\mup} \sbrac{\logb {d\mup}{d\eta_1}}\\
     &= H (\mup \|\eta_1).
\end{align*}
When substituted back into \eqref{kappa reg: mu star ineq},  this yields  \eqref{eqn:relative-entropy-decomposition}.

Next, we show \eqref{eqn:J=I-mu-pi}. Using first the definition of $I^\k_1$ from \eqref{def-Ikk}, next the   definition of $\mup$  from \eqref{eqn:1MRF}, which implies 
$\pi_{\mup}=\pi$, and the definition of 
 $\eta_1$ from  Remark \ref{remark:true-law},  
 which implies $\pi_{\eta_1}=\nu \otimes \nu$,  and finally  the symmetry of $\pi$ in the fourth line, we conclude that
\begin{align*}
    I^\k_1(\mup)&=H(\mup\|\eta_1)-\frac{\k}{2}H(\pi_{\mup}\|\pi_{\eta_1})\\
    &=\bbE_{\mup}\sqpr{\log\dfrac{d\pi^o}{d\nu}(\bfX_o)+\sum_{v\in N_{\pmb\tau}(o)}\log\dfrac{d\pi^{1|o}}{d\nu}(\bfX_v|\bfX_o)}-\frac{\k}{2}H(\pi\|\nu\otimes\nu)\\
    &=H(\pi^o\|\nu)+\frac{\k}{2}\bbE_{\pi}\sqpr{2\log\dfrac{\pi(\bfX_1,\bfX_o)}{\nu(\bfX_1)\pi^o(\bfX_o)}-\log\dfrac{\pi(\bfX_1,\bfX_o)}{\nu(\bfX_1)\nu(\bfX_o)}}\\
    &=H(\pi^o\|\nu)+\frac{\k}{2}\bbE_{\pi}\sqpr{\log\dfrac{\pi(\bfX_1,\bfX_o)}{\pi^o(\bfX_1)\pi^o(\bfX_o)}}\\
    &=H(\pi^o\|\nu)+\frac{\k}{2}H(\pi\|\pi^o\otimes\pi^o).
\end{align*}
By \eqref{eqn:J-nu-k} this is equal to $J_\k^\nu(\pi).$  
    
\end{proof}

\section{Proof that minimizers achieve equality constraint}\label{sec:c=c'}

\begin{proof}[Proof of Lemma \ref{lemma:optimizing_c=c'}] 
Let $\cc > \sref.$  Clearly, we have    
    \[
    \qquad \min_{\cc' \geq \cc} \min_{\pi \in \BB_{h}(\cc')} J_\k^\nu(\pi) \leq  \min_{\pi \in \BB_{h}(\cc)} J_\k^\nu(\pi).\]  
    Suppose the minimum on the left-hand side of the last display is achieved for some  $\tilde{\pi} \in \BB_{h}(\cc')$ with  $\cc'> \cc.$ Since $\Exp_{\tilde{\pi}}[h(X_o, X_1)] = \cc' > \cc,$ there exists a small neighborhood around $\tilde\pi$ that lies completely inside the constraint set. Therefore, $\tilde{\pi}$ must be a local minimum (and thus stationary point) of the map 
    \begin{align}
     \Psym{\X} \ni    \pi \mapsto J_\k^\nu(\pi) \in \R. \label{eqn:pi-map}
    \end{align}
    Applying 
    \cite[Proposition 1.7]{DemMonSlySun14} 
     with $d$, $\Delta$, $\pmb{h}$, $\bar{h}$, $h$, $\bar{\psi}$ and $\psi$ therein corresponding to 
     $\k$, $\Psym{\X},$ $\pi$, $\pi^0$, $\ell$, $\nu$ and the constant function $1$, respectively, and, comparing \cite[equation (9)]{DemMonSlySun14} with \eqref{eqn:J-nu-k}, to conclude that  the functional $\Phi$ in \cite{DemMonSlySun14} corresponds to the map  $-J_\k^\nu$, we conclude that any stationary point $\tilde{\pi}$ 
     of $J_k^\nu$ takes the form 
     \[ \tilde{\pi} (x,z) \propto \ell (x) \ell(z), \quad x, z \in \X,  \]
    for some $\ell \in {\mathcal P}(\X)$ 
    that is a fixed point of the cavity map $\BP:\P(\X) \to \P(\X)$ defined in \eqref{eqn:cavity} with $(\bar\psi,\psi)=(\nu,1)$, that is,
    \begin{align*}
        \BP \hbar(x) = \frac{1}{Z}\nu(x) \cpr{ \sum_{x' \in \X}\hbar(x') }^{\k-1}, \quad x \in \X,
    \end{align*} 
    where $Z$ is the normalizing constant (note that the set of fixed points of the map $\BP$ is denoted by ${\mathcal H}^*$ in \cite{DemMonSlySun14}). 
     Since $\hbar \in {\mathcal P}(\X),$
    the above map reduces to the trivial map $\BP \hbar = \nu,$ which clearly has $\nu$ as its unique fixed point.   
        This proves that $\ell = \nu$ and hence, shows that $\tilde{\pi} = \nu \otimes \nu$ is the only local minimizer of the map in \eqref{eqn:pi-map}. However, since $c' > c > \sref$, by \eqref{eqn:ceta}, Remark \ref{remark:true-law}  and \eqref{eqn:Medge} it follows that  $\nu \otimes \nu \notin \BB_{h}(c'),$ which contradicts the initial supposition that $\tilde{\pi} \in \BB_{h}(c')$. Hence, every pair $(\cc',\tilde{\pi})$ that attains the minimum on the left-hand side of \eqref{eqn:=>1} must satisfy $\cc' = \cc$ and $\tilde{\pi} \in \BB_{h}(\cc).$ This proves \eqref{eqn:=>1}. An exactly analogous argument proves \eqref{eqn:=>2} when $\cc<\sref.$ 
\end{proof}

\section{Existence of boundary minimizers}
\label{sec:boundaryminimizers}

\subsection{Justification of Example \ref{example:boundary-local-minimizers}}\label{subs:pf-boundary-local-minimizers}
    For $q \in \N$, define 
  $J_\k^{(q)}\coloneqq J_\k^{\mathsf{Unif}([q])}$, which by 
\eqref{eqn:J-nu-k} is given by  
\begin{align}\label{eqn:J-qspin}
    J_\k^{(q)}(\pi) =\log q+(\k-1)H(\pi^o)-\frac{\k}{2}H(\pi),\quad \forall \pi\in \Psym{\X}, 
\end{align}
where $H$ is the entropy functional defined in \eqref{defn:entropy}.
    We first show that $\pi_*$ is indeed a  local minimizer by computing the change in $J^{(q)}_\k$ in all directions satisfying the constraint. In fact, any such perturbation of $\pi^*$ in a feasible direction  can be obtained as an  interpolation of $\pi_*$ and some $\xi\in \BB_h(\cc)$. Therefore, fix $\BB_h (\cc)$ and let $\varepsilon > 0$ be sufficiently small such that $\pi_\varepsilon\coloneqq (1-\varepsilon)\pi_*+\varepsilon\xi\in \BB_h(\cc)$. Then by Lemma \ref{lemma:eloge-derivative} and the form of $\pi_*$ given above, we have
\begin{align}
    \lim_{\varepsilon\rightarrow 0^+}\frac{1}{|\varepsilon\log \varepsilon|}\cpr{J_\k^{(q)}(\pi_\varepsilon)-J_\k^{(q)}(\pi_*)}&=(\k-1)\sum_{x\in [q]}\xi^o(x)\indf_{\pr{\pi^o_*(x)=0}}-\frac{\k}{2}\sum_{x,z\in [q]}\xi(x,z)\indf_{\pr{\pi_*(x,z)=0}}\nonumber\\
    &=(\k-1)\xi^o(q)-\frac{\k}{2}\cpr{\xi(q,q)+2\sum_{x\in [q-1]}\xi(q,x)}.\label{eqn:eloge-derivative-1}
\end{align}
Since any $\xi\in \BB_h(\cc)$ satisfies the constraint and is a probability measure, it follows that 
\begin{align}
    \xi(q,q)-2\sum_{x\in [q-1]}\xi(q,x)+\frac{1}{2}\sum_{x,z\in [q-1]}\xi(x,z)   &=\frac{1}{2} \nonumber\\
    \xi(q,q)+2\sum_{x\in [q-1]}\xi(q,x)+\sum_{x,z\in [q-1]}\xi(x,z)&=1.\nonumber
\end{align}
This implies that there exists some $p\in [0,1/8]$  such that  
\begin{equation}
    \cpr{\xi(q,q),\sum_{x\in [q-1]}\xi(q,x),\sum_{x,z\in [q-1]}\xi(x,z)}=\cpr{6p,p,1-8p}.\label{eqn:ex-constraint}
\end{equation}
In particular, this implies $\xi^0(q) = 7p$, which together with \eqref{eqn:eloge-derivative-1} and the fact that $\k\geq 3$, we have for $p>0,$
\begin{align}
    \lim_{\varepsilon\rightarrow 0^+}\frac{1}{|\varepsilon\log \varepsilon|}\cpr{J_\k^{(q)}(\pi_\varepsilon)-J_\k^{(q)}(\pi_*)}=3\k p-7p\geq 2p>0.\label{eqn:eloge-derivative-2}
\end{align}
If $p=0$, we have $\xi^o(q)=0$. Note that the constraint \eqref{eqn:ex-constraint}  is automatically satisfied for any measure $\pi\in \Psym{[q-1]}$. Therefore, by \eqref{eqn:unique-minimizer-of-J}, we know that $\pi_*|_{\Psym{[q-1]}}$, which is the uniform measure on $[q-1]^2$, is the unique global minimizer of $J_\k^{(q-1)}$. Moreover, for any $\pi\in \Psym{[q]}$ that puts no mass on spin value $q$, we have 
\begin{equation*}
    J_\k^{(q)}(\pi)=J_\k^{(q-1)}(\pi)+\log\frac{q}{q-1}.
\end{equation*}
As a result,
\begin{equation*}
    \pr{\pi_*}=\argmin \pr{J_\k^{(q)}(\pi):\pi \in \Psym{[q]}, \pi^o(q)=0}.
\end{equation*}
Together with \eqref{eqn:eloge-derivative-2}, we have shown that $\pi_*$ is a  local minimizer.
\qed

\subsection{Boundary global minimizers when $\nu$ is not uniform}\label{sec:degenerate}

In this section, we first illustrate the transition for global minimizers from being in the interior to becoming on the boundary using a particular example in Section \ref{subs:heuristic-boundary-global}. Then we justify Example \ref{example:boundary-global-minimizer} rigorously in Section \ref{subs:pf-boundary-global-minimizers}. 

\subsubsection{Heuristics for establishing  existence of degenerate global minimizers}\label{subs:heuristic-boundary-global}

Assume that $\X = \{1,-1\}$ and use the parametrization $\pi=\pi[s,t]$ in \eqref{eqn:two-spin-parametrization},   $J^\nu_\k(s,t)$ as in \eqref{eqn:J-in-st} and, for edge potentials satisfying 
$h(1,1)+h(-1,-1)-2h(1,-1)\neq 0$, 
the constraint segment  
$\BB_h(s, t) = \{t = w(h)s: (s,t) \in \Delta\}$ as defined in 
\eqref{eqn:two-spin-slope}. 
We plot the contours of $J^\nu_\k(s,t)$ in Figure \ref{fig:non-uniform-section-degenerate-heatmap} for the case 
 $\nu(-1)=2/3$, $\nu(1)=1/3$ and 
$\kappa = 5$, with each dashed line representing the constraint segment $t = w(h)s$ for some edge potential $h$ and constraint value $\cc=\k h(-1,-1)$. 
An inspection of 
Figure \ref{fig:non-uniform-section-degenerate-section}, which  contains the corresponding 
  plots of the function  $s\mapsto J^\nu_5(s,w(h)s)$, 
   suggests the following: (i)  there exists a critical slope $w^*$ such that when $w(h)<w^*$, $s=0$ is the unique global minimizer of $s\mapsto J^\nu_5(s,w(h)s)$, which implies that $\pi[0,0]=\delta_{(-1,-1)}$ is the unique global minimizer of the edge optimization problem; (ii) when $w(h)=w^*$, there are exactly two global minimizers, with one lying in the interior and one lying on the boundary; (iii) When $w(h)>w^*$, the global minimizer is unique and lies in the interior.  Below, we make rigorous these observations to  show the existence of the boundary global minimizer in Example \ref{example:boundary-global-minimizer}.

\begin{figure}[htbp]
  \centering

  \begin{subfigure}[t]{0.46\textwidth}
    \centering
    \includegraphics[width=\linewidth,height=5.5cm]{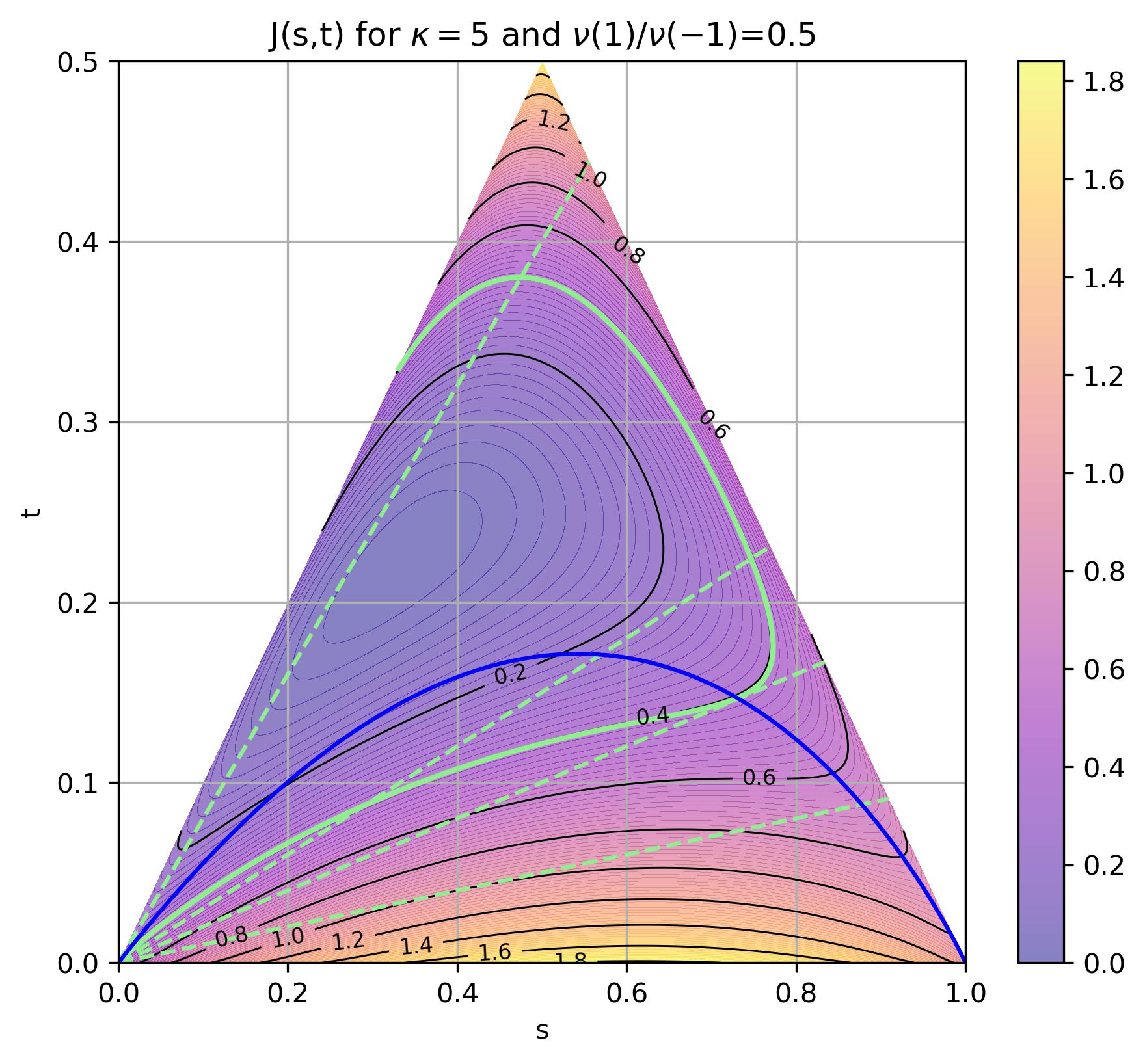}
    \caption{Contour plot for $J^\nu_5(s,t)$:  green solid contour represents  $\{(s,t):J^\nu_5(s,t)=J^\nu_5(0,0)\}$; green dashed lines represent different constraint segments.}
    \label{fig:non-uniform-section-degenerate-heatmap}
  \end{subfigure}
  \hspace{0.5cm}
  \begin{subfigure}[t]{0.41\textwidth}
    \centering
    \includegraphics[width=\linewidth,height=5.5cm]{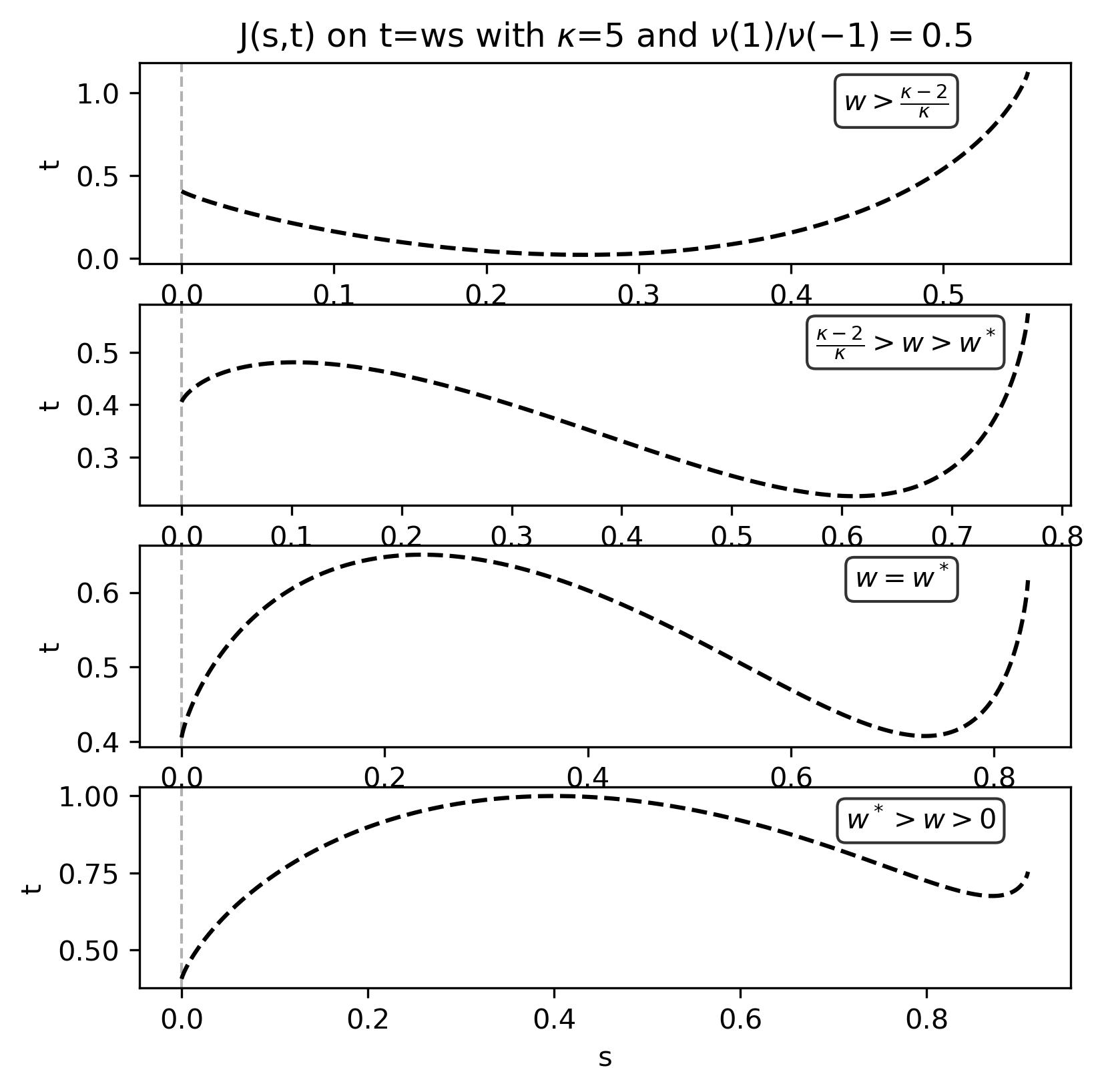}
    \caption{Value of $J^\nu_5(s,t)$ along the  constraint segments $t=ws$}
    \label{fig:non-uniform-section-degenerate-section}
  \end{subfigure}

  \caption{$J^\nu_5(s,t)$ with non-uniform mark distribution $\nu (1)  = 1/3, \nu(-1)  = 2/3$}
  \label{fig:non-uniform-section-degenerate}
\end{figure}

\vspace{-0.05in}

\subsubsection{Justification of Example \ref{example:boundary-global-minimizer}}\label{subs:pf-boundary-global-minimizers}

Define 
\begin{equation}
    w^*\coloneqq \inf \pr{w: \exists s\in (0,1] \text{ such that }(s,ws)\in \Delta \text{ and }J^\nu_\k(s,ws)\leq J^\nu_\k(0,0)}.\label{eqn:crit-phase-transition}
\end{equation}
We claim that in the setting of $\k=5$ and $\nu(1)=1/3$, $\nu(-1)=2/3$, we have $w^*=1/5$. Then, with $h$ as defined in \eqref{eqn:boundary-global-h}, by \eqref{eqn:two-spin-slope}, the slope $w(h)$, defined in \eqref{eqn:two-spin-slope},
is equal to
\begin{equation*}
    w(h)=\dfrac{7-4}{7+4-2\cdot(-5)}=\frac{1}{7}<\frac{1}{5}.
\end{equation*}
Therefore, $s=0$ is the unique global minimizer of the map $s\mapsto J^\nu_\k(s,w(h)s)$, otherwise the existence of another $s'\neq 0$ with $J^\nu_\k(s',w(h)s')\leq J^\nu_\k(0,0)$ would contradict our claim that $w^*=1/5$. As a result, $\pi_*=\pi[0,0]=\delta_{(-1,-1)}$ is the unique global minimizer of the edge optimization problem in this setting.

Below, we prove our claim that $w^*=1/5$ by showing that the minimal value of $J^\nu_\k(s,t)$ in the regime $D\coloneqq \pr{(s,t)\in \Delta: t\leq s/5}$ is $J^\nu_\k(0,0)$, achieved only at two points $(0,0)$ and $(4/5,4/25)$, that is
\begin{equation}
    \argmin\pr{J^\nu_\k(s,t):(s,t)\in D}=\pr{(0,0),(4/5,4/25)}.\label{eqn:optimizer-for-D}
\end{equation}
Note that the $\argmin$ is well-defined since $J^\nu_\k$ is a continuous function on the compact set $D$. To show \eqref{eqn:optimizer-for-D}, we 
partition $D$ into four regimes and separately analyze  the minimizers in each: 
    \begin{enumerate}
        \item $D_0\coloneqq D\cap \pr{(s,t)\in \Delta: s\neq 0, t\neq s/5, 1-s-t\neq 0}$:  We show below that $J^\nu_\k$ has no critical  point, and hence no local minimizer,      in $D_0$.  Note that  \eqref{eqn:Jnu-two-spin} and \eqref{eqn:Jst-dt} imply 
    \begin{equation*}
        \partial_t J^\nu_\k(s,t)=\partial_t J^{(2)}_\k(s,t)=0\quad \text{if and only if}\quad t=s(1-s).
    \end{equation*}
    Thus, using \eqref{eqn:Jnu-two-spin} and \eqref{eqn:Jst-ds} in the first line, and then substituting $t=s(1-s)$, we have
    \begin{align*}
        \partial_s J^\nu_\k(s,t)\mid_{t=s(1-s)}&=\log\cpr{\frac{\nu(-1)}{\nu(1)}}+\partial_s J^{(2)}_\k(s,t)\mid_{t=s(1-s)}\\
        &=\log\cpr{\frac{\nu(-1)}{\nu(1)}}+(1-\k)\log\cpr{\frac{s}{1-s}}+\frac{\k}{2}\log\cpr{\frac{s^2}{(1-s)^2}}\\
        &=\log\cpr{\frac{\nu(-1)}{\nu(1)}}+\log\cpr{\frac{s}{1-s}},
    \end{align*}
    which is zero if and only if $s=\nu(1)$. This shows that $\nabla J^\nu_\k(s,t)=(0,0)$ if and only if $(s,t)=(\nu(1),\nu(1)\nu(-1))$ and $(\nu(1),\nu(1)\nu(-1))=(1/3,2/9)\notin D$.
        \item $D_1\coloneqq D\cap \pr{(s,t)\in \Delta: 1-s-t=0,s\neq 1}$: For any point $(s,t)\in D_1$, pick the constraint segment $\Gamma$ that passes through $(0,0)$ and $(s,t)$. Then by Lemma \ref{lemma:pos_marginal_cannot_be_locmin}, $\pi[s,t]$ cannot be a local minimizer of the edge optimization problem corresponding to $\Gamma$. Together with the fact that $\Gamma\subset D$, it follows that $(s,t)$ cannot attain the minimal value of $J^\nu_\k(s,t)$ on $D$;
        \item $D_2\coloneqq D\cap \pr{(s,t)\in \Delta:t=0}$: For any point $(s,0)\in D_2$, using \eqref{eqn:Jnu-two-spin} in the first line and the fact that entropy is nonnegative in the second line, we have
        \begin{align*}
            J^\nu_\k(s,0)&=(s+1)\log 2-\log\cpr{\frac{4}{3}}-\frac{3}{2}\cpr{s\log s+(1-s)\log(1-s)}\\
            &\geq \log 2-\log\cpr{\frac{4}{3}},
        \end{align*}
    with equality attained if and only if $s=0$. This implies that the unique global minimizer of $J^\nu_\k$ in $D_2$ is $(0,0)$;
        \item $D_3\coloneqq D\cap \pr{(s,t)\in \Delta:t=s/5}$: By direct computation,  for $s\in (0,5/6)$ we have 
        \begin{align*}
            \frac{d}{ds}J^\nu_\k(s,\frac{1}{5}s)&=-3\log 2-\log s+4\log(1-s)-3\log(5-6s);\\
            \frac{d^2}{ds^2}J^\nu_\k(s,\frac{1}{5}s)&=\frac{1}{s(1-s)(5-6s)}\cpr{9s-5}.
        \end{align*}
        Thus $s=4/5$ is a local minimizer since the above implies 
        \begin{equation*}
            \frac{d}{ds}J^\nu_\k(s,\frac{1}{5}s)\mid_{s=\frac{4}{5}}=0\quad \text{and}\quad \frac{d^2}{ds^2}J^\nu_\k(s,\frac{1}{5}s)\mid_{s=\frac{4}{5}}>0.
        \end{equation*}
        Moreover, since $(0,5/6)\ni s\mapsto J^\nu_\k(s,\frac{1}{5}s)$ is a smooth function that has a unique inflection point at $s=5/9$,   $s=4/5$ must be the unique local minimizer in $(0,5/6)$. Furthermore, since $s\in [0,5/6]$ for every $(s,t)\in D_3$ and $J^\nu_\k(5/6,1/6)>J^\nu_\k(4/5,4/25)=J^\nu_\k(0,0)$, we conclude that $(0,0)$ and $(4/5,4/25)$ are the only two global minimizers in $D_3$.
        
    \end{enumerate}
    \qed

    \begin{remark}
         From Figure \ref{fig:non-uniform-section-degenerate},  note that the critical slope should coincide with the tangent line of the level curve passing through $(0,0)$, that is, 
        \begin{equation*}
            w^*=\inf\pr{\frac{t}{s}:J^\nu_\k(s,t)=J^\nu_\k(0,0),s\neq 0}.
        \end{equation*}
        However, showing this requires more delicate computations and is not the main focus of the above  justification, so we proceeded with a slightly less intuitive but simpler argument.
    \end{remark}

\section{An example of a ferromagnetic Ising measure with a consensus that is lower than typical}\label{sec:example-ferro-ising-lower-consensus}

In the ferromagnetic uniqueness regime, the consensus of the unique Ising measure is always higher than typical. However, this is not always true in the  non-uniqueness regime. Although the ferromagnetic Ising measure with boundary condition having the same bias as  the reference measure $\nu$ still leads to a higher consensus, the Ising measure with the opposite boundary condition could result in a consensus that is lower than atypical. Below, we demonstrate this phenomenon via  a numerical example.  

\begin{example}\label{example:ferro-lower-consensus}
   Set $\k=6$, $(\beta,B)=(1.75,5.4)$. Then $h^-\approx -3.22177$ is a fixed point of the cavity map $\Gamma$,  and by  \eqref{eqn:ising-covariance},
    \begin{equation*}
        \bbE_{\ising_\k^-(\beta,B)}\sqpr{\sum_{v=1}^\k m(X_o,X_v)}=\k\cdot \dfrac{e^{2\beta}\cosh(2h^-)-1}{e^{2\beta}\cosh(2h^-)+1}\approx 5.99884722<5.9995104 \approx \k\cdot \tanh^2(B)=\smref.
    \end{equation*}
\end{example}

\bibliographystyle{plain}
\bibliography{biblio}

\scriptsize{\textsc{Ivan Lee: Department of Statistics and Data Science, Yale University, New Haven CT 06511, USA\newline}}
email: \texttt{ivan.lee@yale.edu}

\scriptsize{\textsc{I-Hsun Chen: Division of Applied Mathematics, Brown University, Providence RI 02912, USA\newline}}
email: \texttt{i-hsun\_chen@brown.edu}

\scriptsize{\textsc{Kavita Ramanan: Division of Applied Mathematics, Brown University, Providence RI 02912, USA\newline}}
email: \texttt{kavita\_ramanan@brown.edu}

\scriptsize{\textsc{Sarath Yasodharan: Department of Industrial Engineering and Operations Research, Indian Institute of Technology Bombay, Powai, Mumbai 400076, India\newline}}
email: \texttt{sarath\_yasodharan@iitb.ac.in}

\end{document}